\documentclass{amsart}

\usepackage{amsmath, amssymb, amsthm}
\usepackage{bbm, lmodern, mathrsfs, euscript}
\usepackage{tikz, tikz-cd}
\usetikzlibrary{decorations.pathreplacing}
\usepackage{xcolor}
\usepackage{arydshln}
\usepackage{enumitem}
\usepackage[colorlinks=true, pdfstartview=FitV, linkcolor=blue, citecolor=blue, urlcolor=blue, breaklinks=true]{hyperref}
\usepackage{cleveref}

\usepackage{forloop}
\newcounter{fonts}
\let\eeee\edef
\forloop{fonts}{1}{\value{fonts} < 27}{%
\expandafter\eeee\csname \Alph{fonts}\Alph{fonts}\endcsname{\noexpand\mathbb{\Alph{fonts}}}
\expandafter\eeee\csname b\Alph{fonts}\endcsname{\noexpand\mathbf{\Alph{fonts}}}
\expandafter\eeee\csname c\Alph{fonts}\endcsname{\noexpand\mathcal{\Alph{fonts}}}
\expandafter\eeee\csname f\Alph{fonts}\endcsname{\noexpand\mathfrak{\Alph{fonts}}}
\expandafter\eeee\csname fr\alph{fonts}\endcsname{\noexpand\mathfrak{\alph{fonts}}}
\expandafter\eeee\csname s\Alph{fonts}\endcsname{\noexpand\mathscr{\Alph{fonts}}}
\expandafter\eeee\csname e\Alph{fonts}\endcsname{\noexpand\EuScript{\Alph{fonts}}}
\expandafter\eeee\csname v\alph{fonts}\endcsname{\noexpand\overline{\alph{fonts}}}
}

\setlist[enumerate,1]{label={\rm(\roman*)}}

\newtheorem{thm}{Theorem}[section]
\newtheorem{defn}[thm]{Definition}
\newtheorem{prop}[thm]{Proposition}
\newtheorem{cor}[thm]{Corollary}
\newtheorem{lem}[thm]{Lemma}

\newtheoremstyle{remarkstyle}{.5\baselineskip\@plus.2\baselineskip\@minus.2\baselineskip}{.5\baselineskip\@plus.2\baselineskip\@minus.2\baselineskip}{\rm}{}{\bfseries}{.}{.5em}{}
\theoremstyle{remarkstyle}
\newtheorem{rem}[thm]{Remark}
\AtEndEnvironment{rem}{\hfill $\blacksquare$}

\AtEndEnvironment{rems}{\hfill $\blacksquare$}

\newtheorem{example}[thm]{Example}
\AtEndEnvironment{example}{\hfill $\blacksquare$}

\newtheoremstyle{solutionstyle}{\topsep}{\topsep}{\rm}{}{\it}{.}{.5em}{}
\theoremstyle{solutionstyle}

\AtEndEnvironment{sol}{\qed}
\newtheorem*{sol*}{Solution}
\AtEndEnvironment{sol*}{\qed}

\newcommand{\und}{\underline{\hspace{2ex}}}
\newcommand{\iso}{\overset{\sim}{\longrightarrow}}
\newcommand{\inj}{\hookrightarrow}
\newcommand{\surj}{\twoheadrightarrow}

\newcommand{\GL}{\mathbf{GL}}
\newcommand{\SL}{\mathbf{SL}}

\newcommand{\PGL}{\mathbf{PGL}}

\newcommand{\bCent}{\mathbf{Cent}}
\newcommand{\bNorm}{\mathbf{Norm}}
\newcommand{\bSU}{\mathbf{SU}}

\DeclareMathOperator{\nGL}{GL}

\DeclareMathOperator{\Grp}{\mathfrak{Grp}}
\DeclareMathOperator{\Ab}{\mathfrak{Ab}}
\DeclareMathOperator{\Sets}{\mathfrak{Sets}}
\DeclareMathOperator{\Rings}{\mathfrak{Rings}}

\DeclareMathOperator{\Sch}{\mathfrak{Sch}}

\newcommand{\cHom}{\mathcal{H}\hspace{-0.4ex}\textit{o\hspace{-0.2ex}m}}
\newcommand{\cEnd}{\mathcal{E}\hspace{-0.4ex}\textit{n\hspace{-0.2ex}d}}

\newcommand{\bAut}{\mathbf{Aut}}

\DeclareMathOperator{\Mat}{M}
\DeclareMathOperator{\Span}{Span}

\DeclareMathOperator{\Img}{Img}
\DeclareMathOperator{\Ker}{Ker}

\DeclareMathOperator{\Hom}{Hom}

\DeclareMathOperator{\Aut}{Aut}

\DeclareMathOperator{\Id}{Id}

\DeclareMathOperator{\Nrd}{Nrd}

\DeclareMathOperator{\Sym}{Sym}
\DeclareMathOperator{\Inn}{Inn}
\DeclareMathOperator{\Pic}{Pic}

\DeclareMathOperator{\Dyn}{Dyn} 
\newcommand{\cDyn}{\mathcal{D}\hspace{-0.4ex}\textit{y\hspace{-0.3ex}n}}
\DeclareMathOperator{\Of}{Of} 
\newcommand{\cOf}{\mathcal{O}\hspace{-0.3ex}\textit{f}}
\newcommand{\cOD}[1]{\cOf(\cDyn(#1))}

\DeclareMathOperator{\rank}{rank}

\newcommand{\cFlag}{\cF\hspace{-0.4ex}\ell\hspace{-0.35ex}\textit{a\hspace{-0.25ex}g}}
\newcommand{\LowFlag}{\cFlag^-}
\newcommand{\RaiFlag}{\cFlag^+}
\newcommand{\ConFlag}{\cC\hspace{-0.45ex}\cF\hspace{-0.4ex}\ell\hspace{-0.35ex}\textit{a\hspace{-0.25ex}g}}

\DeclareMathOperator{\SB}{SB}
\newcommand{\cGSB}{\cG\cS\cB}
\newcommand{\LowGSB}{\cGSB^-}
\newcommand{\RaiGSB}{\cGSB^+}
\newcommand{\ConIdeal}{\cC\hspace{-0.4ex}\cI\hspace{-0.4ex}\textit{d\hspace{-0.25ex}e\hspace{-0.25ex}a\hspace{-0.25ex}l}}

\newcommand{\cPar}{\cP\hspace{-0.45ex}\textit{a\hspace{-0.15ex}r}}
\DeclareMathOperator{\Par}{Par}
\newcommand{\bStab}{\mathbf{Stab}}
\newcommand{\bOut}{\mathbf{Out}}
\newcommand{\bInn}{\mathbf{Inn}}

\newcommand{\op}{\mathrm{op}}

\newcommand{\cStiefel}{\cS\hspace{-0.4ex}\mathit{t\hspace{-0.25ex}i\hspace{-0.3ex}e\hspace{-0.25ex}f\hspace{-0.5ex}e\hspace{-0.2ex}l}}
\newcommand{\cIdemp}{\cI\hspace{-0.4ex}\mathit{d\hspace{-0.4ex}e\hspace{-0.35ex}m\hspace{-0.35ex}p}}

\newcommand{\cPL}{\cP\hspace{-0.4ex}\cL}
\newcommand{\ConStiefel}{\cC\hspace{-0.4ex}\cS\hspace{-0.4ex}\mathit{t\hspace{-0.25ex}i\hspace{-0.3ex}e\hspace{-0.25ex}f\hspace{-0.5ex}e\hspace{-0.2ex}l}}
\newcommand{\LowStiefel}{\cStiefel^-}
\newcommand{\RaiStiefel}{\cStiefel^+}
\newcommand{\LowIdemp}{\cIdemp^-}
\newcommand{\RaiIdemp}{\cIdemp^+}

\newcommand{\SplitFlag}{\cS\hspace{-0.4ex}\mathit{p}\hspace{-0.1ex}\ell\hspace{-0.3ex}\mathit{i\hspace{-0.2ex}t}\cF\hspace{-0.4ex}\ell\hspace{-0.35ex}\textit{a\hspace{-0.25ex}g}}

\newcommand{\SubMod}{\cS\hspace{-0.4ex}\mathit{u\hspace{-0.3ex}b}\hspace{-0.2ex}\cM\hspace{-0.4ex}\mathit{o\hspace{-0.2ex}d}}
\newcommand{\SplitSubMod}{\cS\hspace{-0.4ex}\mathit{p}\hspace{-0.1ex}\ell\hspace{-0.3ex}\mathit{i\hspace{-0.2ex}t}\cS\hspace{-0.4ex}\mathit{u\hspace{-0.3ex}b}\hspace{-0.2ex}\cM\hspace{-0.4ex}\mathit{o\hspace{-0.2ex}d}}

\newcommand{\rhoStiefel}{\rho_\cS}
\newcommand{\rhoIdemp}{\rho_\cI}

\newcommand{\lann}{{^0}}

\newcommand{\sw}{\textrm{sw}}

\DeclareMathOperator{\Lex}{Lex}

\newcommand{\subrank}{\mathrm{subrank}}
\newcommand{\gap}{\mathrm{gap}}

\newcommand{\cOSone}{\cO_{\PP_S^n}(1)}
\newcommand{\cOone}{\cO_{\PP^n}(1)}

\DeclareMathOperator{\can}{can}
\DeclareMathOperator{\rad}{rad}

\title{Outer Type Severi-Brauer Schemes}
\begin{document}
\author[C. Ruether]{Cameron Ruether}
\address{The ``Simion Stoilow" Institute of Mathematics of the Romanian Academy, 21 Calea Grivitei Street, 010702 Bucharest, Romania.}
\email{cameronruether@gmail.com}

\thanks{This work was supported by the project ``Group schemes, root systems, and related representations" founded by the European Union - NextGenerationEU through Romania's National Recovery and Resilience Plan (PNRR) call no. PNRR-III-C9-2023-I8, Project CF159/31.07.2023, and coordinated by the Ministry of Research, Innovation and Digitalization (MCID) of Romania.}

\date{May 18, 2026}

\maketitle

\noindent{\bf Abstract:} {We introduce the notion of a lowered flag of $\cO$--modules in order to define a sheaf of flags of ideals isomorphic to the sheaf of parabolic subgroups for the general linear group $\GL_{1,\cA}$ of an Azumaya algebra over a general scheme $S$. This notion is extended to the outer type $A_n$ case and we define a suitable sheaf of flags of ideals isomorphic to the sheaf of parabolic subgroups for a unitary group over $S$. When the group is suitably split these are related to flags of submodules in a vector bundle or in a vector bundle with hermitian form, respectively. We also define a sheaf of tuples of idempotents in the associated algebra which is isomorphic to the sheaf of parabolic and Levi subgroup pairs. We show how the type morphism from parabolic subgroups to the Dynkin scheme can be defined in terms of these sheaves of flags. We review how the Severi-Brauer scheme associated to an Azumaya algebra $\cA$ is isomorphic to a particular fiber of this type morphism and we generalize this idea to the outer case in order to define outer Severi-Brauer schemes. We provide a new approach to Quillen's construction which produces an Azumaya algebra from a Severi-Brauer scheme and we show that an outer version of Quillen's construction also exists for outer Severi-Brauer schemes which produces an algebra with unitary involution.}
\medskip

\bigskip

\section*{Introduction}
{%
\renewcommand{\thethm}{\Alph{thm}}
Over a field, Severi-Brauer varieties are to projective space as central simple algebras are to the matrix algebra. A Severi-Brauer variety over a field $\FF$ is any variety which is \'etale locally isomorphic to some projective space $\PP_\FF^n$ and a central simple algebra is any algebra which is \'etale locally isomorphic to some $\Mat_{n+1}(\FF)$. The construction relating these two types of objects works by sending a central simple $\FF$--algebra $A$ of degree $n+1$ to the space of its dimension $n+1$ right ideals (we use right ideals as our convention, though the construction also works with left ideals) which is denoted by $\SB(A)$ and called the Severi-Brauer variety of $A$. For the matrix algebra $\Mat_{n+1}(\FF)$ this recovers $n$--dimensional projective space since such right ideals in $\Mat_{n+1}(\FF)$ are of the form
\[
I_v = \left\{ \begin{bmatrix} a_0v & a_1v & \cdots & a_nv\end{bmatrix} \mid a_i\in \FF\right\}
\]
for a column vector $0\neq v\in \FF^{n+1}$ and where $I_v = I_w$ if and only if $w$ is a scalar multiple of $v$. This construction provides an equivalence of categories between central simple algebras of degree $n+1$ and Severi-Brauer varieties of dimension $n$. In particular, the group of $\FF$--variety automorphism of $\SB(A)$ agrees with the group $\PGL(A)$ of algebra automorphisms of $A$. Via $\PGL_A$ we also obtain an action of $\GL_{1,A}$ on both objects. Under this action, the stabilizer of a point in $\SB(A)$, equivalently of a dimension $n+1$ right ideal in $A$, is a parabolic subgroup of the reductive linear algebraic group $\GL_{1,A}$. However, not all parabolic subgroups can be obtained in this way by only considering ideals of dimension $n+1$. In fact, one needs to considers ideals of all possible dimensions (all of which are divisible by $n+1$) as well as all flags of ideals. Doing so, one obtains an isomorphism
\[
\{ 0 \subsetneq I_1 \subsetneq I_2 \subsetneq \ldots I_\ell \subsetneq A \} \iso \{\bP \subseteq \GL_{1,\cA} \text{ parabolic}\}
\]
giving by sending a flag of ideals to its stabilizer. It is important to note here than the flags occurring in the flags on the left are of ideals of strictly increasing dimension, i.e., we do not allow $I_j = I_{j+1}$, and that the set on the left contains flags of varying length $\ell$. The length $\ell$ may range from $0$, corresponding to the ``empty" flag $0\subseteq A$ whose stabilizer is all of $\GL_{1,\cA}$, up to a maximum of $n$. This variety of flags is called the generalized Severi-Brauer variety of $A$ and by simply considering flags of the form $(0\subsetneq I \subsetneq A)$ where $I$ is dimension $n+1$, we find $\SB(A)$ as a subvariety.

We aim to produce a similar story working with reductive linear algebraic groups of type $A_n$ over a general base scheme $S$. We consider such groups as sheaves of groups $\bG \colon \Sch_S \to \Grp$ on the large fppf site of $S$, though we note that our arguments also hold for the \'etale topology. When working with modules and algebras they will be $\cO$--modules or algebras where $\cO\colon \Sch_S \to \Rings$ is the structure sheaf of global sections. In this context, we know from \cite{SGA3} that the parabolic subgroups of $\bG$ form a sheaf
\[
\cPar_\bG \colon \Sch_S \to \Sets
\]
which is represented by a scheme which is projective over $S$. Initial attempts at introducing a sheaf of flags which is isomorphic to this sheaf of parabolic subgroups quickly run into problems related to the two features highlighted above, that we do not want equality between subsequent ideals in the flag and that we want the length of flags to be able to vary. Let $\cA \colon \Sch_S \to \Ab$ be an Azumaya $\cO$--algebra. Of course, we do not want equality between subsequent terms since if this is allowed, then unequal flags of right ideals of the form
\[
(0 \subseteq \cI \subseteq \cJ \subseteq \cJ \subseteq \cJ \subseteq \cA) \text{ and } (0 \subseteq \cI \subseteq \cI \subseteq \cJ \subseteq \cJ \subseteq \cA)
\]
will have the same stabilizer. However, imposing inequality between sheaves is only a global property, i.e., even if $\cI\subsetneq \cJ$ there may exists $U\in \Sch_S$ such that $\cI|_U = \cJ|_U$. So, one may think to define restrictions of flags in such a way that if some subsequence of ideals become equal, only one is kept. For example, if the flag
\[
0 \subsetneq \cI_1 \subsetneq \cI_2 \subsetneq \cI_3 \subsetneq \cI_4 \subsetneq \cA 
\]
has the property that $\cI_1|_U \neq \cI_2|_U = \cI_3|_U \neq \cI_4$, then the flag's restriction to $U$ is defined to be
\[
0 \subsetneq \cI_1|_U \subsetneq \cI_2|_U \subsetneq \cI_4|_U \subsetneq \cA|_U.
\]
This unfortunately causes the resulting presheaf to not be separated. Indeed, over a disconnected scheme such as $U\sqcup V$, the $\cO|_{U\sqcup V}$--algebra $\cA = (\cA_1,\cA_2)$ is determined by a $\cO|_U$--algebra $\cA_1$ and a $\cO|_V$--algebra $\cA_2$. Ideals in $\cA$ also come as a pair of ideals. Using flags $0 \subsetneq \cI_1\subsetneq \cI_2 \subsetneq \cA_1$ and $0 \subsetneq \cJ_1 \subsetneq \cJ_2 \subsetneq \cJ_3 \subsetneq \cA_2$, we can construct two unequal flags in $\cA$,
\begin{align*}
&0 \subsetneq (\cI_1,\cJ_1) \subsetneq (\cI_1,\cJ_2) \subsetneq (\cI_2,\cJ_3) \subsetneq \cA, \text{ and}\\
&0 \subsetneq (\cI_1,\cJ_1) \subsetneq (\cI_2,\cJ_2) \subsetneq (\cI_2,\cJ_3) \subsetneq \cA.
\end{align*} 
Then, both of these flags restrict to $0 \subsetneq \cI_1 \subsetneq \cI_2 \subsetneq \cA_1$ over $U$ and both restrict to $0 \subsetneq \cJ_1 \subsetneq \cJ_2 \subsetneq \cJ_3 \subsetneq \cA_2$ over $V$. Our method of avoiding these problems is to define the notion of a \emph{lowered flag}.
\begin{defn}
Let $\cE$ be a finite locally free $\cO$--module. A flag
\[
0 \subseteq \cV_1 \subseteq \ldots \subseteq \cV_\ell \subseteq \cE
\]
is called a \emph{lowered flag} if, for all $T\in \Sch_S$ and $i\in \{0,\ldots,\ell-1\}$,
\[
\cV_i|_T = \cV_{i+1}|_T \;\Rightarrow\; \cV_i|_T = \cE|_T = \cV_{i+1}|_T. 
\]
\end{defn}
Working with lowered flags avoids the problems discussed above by only allowing equality and truncation upon restriction to occur on the right side of the flag. For the example of the $\cA_1$ and $\cA_2$ flags above, their unique gluing into a lowered flag in $\cA$ is
\[
0 \subsetneq (\cI_1,\cJ_1) \subsetneq (\cI_2,\cJ_2) \subsetneq (\cA_1,\cJ_3) \subsetneq \cA
\]
where intuitively each flag is lowered as much as possible over its component. Working with these flags produces a suitable analogue of the isomorphism with the sheaf of parabolic subgroups.
\begin{thm}[\Cref{inner_case}]\label{intro_GSB_iso}
Let $\cA$ be an Azumaya $\cO$--algebra of constant degree $n+1$. Denote by $\LowGSB_\cA$ the sheaf of lowered flags of right ideals in $\cA$ which are locally direct summands of $\cA$. Then, there is a isomorphism of sheaves
\[
\LowGSB_\cA \iso \cPar_{\GL_{1,\cA}}
\]
which sends a flag to its stabilizer.
\end{thm}
The sheaf $\LowGSB_\cA$ is called the \emph{generalized Severi-Brauer sheaf} of $\cA$. Furthermore, if $\cA = \cEnd_\cO(\cE)$ for a finite locally free $\cO$--module $\cE$ of constant rank $n+1$, then these are also isomorphic to the sheaf $\LowFlag_\cE$ of lowered flags of $\cO$--submodules which are locally direct summands of $\cE$.

Also in \cite{SGA3}, one finds a type morphism for the sheaf of parabolic subgroups. Using the isomorphism of \Cref{intro_GSB_iso} we obtain a type morphism
\[
t\colon \LowGSB_\cA \to \cP_n
\]
where $\cP_n$ is the locally constant sheaf of subsets of $\{1,\ldots,n\}$ considered as increasing tuples. The type morphism then sends a flag of ideals to its locally constant tuple of ranks. The sheaf $\cP_n$ is represented by the scheme $\bigsqcup_{r\in \sP(\{1,\ldots,n\})} S$ and $\LowGSB_\cA$ is of course represented by the same scheme, say $X$, as $\cPar_{\GL_{1,\cA}}$. Thus, the type morphism corresponds to a morphism of schemes $X \to \bigsqcup_{r\in \sP(\{1,\ldots,n\})} S$. Mirroring the field case, the Severi-Brauer scheme $\SB(\cA)$ is the fiber over the $(1)$ component of this map.

The above describes the picture for inner type $A$ objects over a scheme, but the title of this paper is Outer Severi-Brauer Schemes. Any twisted form of $\GL_{n+1}$ is a unitary group $\bG = \bU_{(\cB,\tau)}$ for an Azumaya algebra $(f\colon L\to S,\cB,\tau)$ with involution of the second kind. In detail, this is the data of
\begin{enumerate}
\item a degree two \'etale cover of schemes $f\colon L \to S$,
\item an Azumaya $\cO|_L$--algebra of constant degree $n+1$, and
\item an $\cO$--linear involution $\tau\colon f_*(\cB) \to f_*(\cB)$ over $S$ which is $\cL=f_*(\cO|_L)$--semi-linear.
\end{enumerate}
In the third point, the \'etale cover $L\to S$ comes with a canonical order two automorphism $i\colon L \to L$ over $S$ and this in turn induces an order two $\cO$--algebra isomorphism $i^*\colon \cL \to \cL$. Requiring that the involution $\tau$ be $\cL$--semi-linear simply means that
\[
\tau(ab) = i^*(a)\tau(b)
\]
for all sections $a\in \cL$ and $b\in f_*(\cB)$. For such data, we also define a sheaf of flags isomorphic to $\cPar_\bG$. This sheaf is denoted $\LowGSB_{(\cB,\tau)}$ and involves management of some technicalities, see \Cref{outer_case}. Roughly, it consists of lowered flags of $\cL$--right ideals
\[
0 \subseteq \cI_1 \subseteq \ldots \subseteq \cI_\ell \subseteq f_*(\cB)
\]
such that this flag is equal to
\[
0 \subseteq \tau(\lann\cI_\ell) \subseteq \ldots \subseteq \tau(\lann\cI_1) \subseteq f_*(\cB)
\]
where $\lann\cI$ is the left annihilator of an ideal and applying $\tau$ makes this a right ideal again. This then gives us the outer version of \Cref{intro_GSB_iso}.
\begin{thm}[\Cref{outer_case}]\label{intro_GSB_outer_iso}
Let $(f\colon L\to S,\cB,\tau)$ be as above and $\bG = \bU_{(\cB,\tau)}$. Then there is an isomorphism
\[
\LowGSB_{(\cB,\tau)} \iso \cPar_\bG
\]
which sends a flag to its stabilizer.
\end{thm}
In the outer case, the type morphism takes values in a twist of $\cP_n$, namely $\cOD{\bG}$, which is the sheaf of open and closed subschemes of the Dynkin scheme of $\bG$. It is represented by a scheme $\Of(\Dyn(\bG))$ which is a disjoint union of copies of $S$ and of $L$, where one copy of $L$ is locally isomorphic to the $(1)$ and $(n)$ components of $\bigsqcup_{r\in \sP(\{1,\ldots,n\})} S$.
\begin{thm}[\Cref{outer_SB_schemes}]\label{intro_outer_fiber}
Consider the map of representing schemes corresponding to the type morphism $t\colon \LowGSB_{(\cB,\tau)} \to \cOD{\bG}$. The fiber over the $(1)$ and $(n)$ component $L \subseteq \Of(\Dyn(\bG))$ is the Severi-Brauer scheme of $\cB$,
\[
\SB(\cB) \to L.
\]
Furthermore, there is a canonical isomorphism $\LowGSB_{(\cB,\tau)} \iso \LowGSB_{(\cB^\op,\tau)}$ which yields an isomorphism fitting into the diagram
\[
\begin{tikzcd}
\SB(\cB) \ar[d] \ar[r,"\sim"] & \SB(\cB^\op) \ar[d] \\
L \ar[r,"i"] & L.
\end{tikzcd}
\]
\end{thm}
Because of the correspondence between Azumaya algebras and Severi-Brauer schemes, given a Severi-Brauer scheme $Y\to S$ we may discuss the opposite Severi-Brauer scheme $Y^\op \to S$. That is, if $Y=\SB(\cA)$ then $Y^\op = \SB(\cA^\op)$. In an unsurprising reflection of the definition of an involution of the second kind, we define an outer Severi-Brauer scheme.
\begin{defn}\label{defn_outer_SB_scheme}
An \emph{outer type Severi-Brauer scheme} of constant relative dimension $n$ is the data of $(f\colon L \to S,P \to L,\tau)$ where
\begin{enumerate}
\item $L\to S$ is a degree $2$ \'etale cover with canonical automorphism $i\colon L \to L$,
\item $P \to L$ is a Severi-Brauer scheme over $L$ of constant relative dimension $n$,
\item $\tau \colon P \iso P^\op$ is an $S$--scheme isomorphism making the diagram
\[
\begin{tikzcd}
P \ar[r,"\tau"] \ar[d] & P^\op \ar[d] \\
L \ar[r,"i"] & L
\end{tikzcd}
\]
commute.
\end{enumerate}
\end{defn}

It is clear from construction that \Cref{intro_outer_fiber} gives us an outer Severi-Brauer scheme starting with an Azumaya algebra with involution of the second kind. To provide a construction in reverse, we generalize Quillen's construction from the inner case. We review Quillen's construction as presented in \cite{Kollar}. Given a Severi-Brauer scheme $P \to S$, i.e. a scheme \'etale locally isomorphic to relative projective space $\PP_S^n \to S$, there is an $\cO|_P$--module $\cF(P)$ which fits into the unique non-split exact sequence
\[
0 \to \cO|_P \to \cF(P) \to \cT_{P/S} \to 0
\]
where $\cT_{P/S}$ is the tangent bundle. Then, one considers $f_*(\cEnd_{\cO|_P}(\cF(P)))$ which is an Azumaya $\cO$--algebra. When $P=\SB(\cA)$ then this recovers $\cA$. We show that there is a slight variation to Quillen's construction which may be used in the inner case, and that with our definition of outer Severi-Brauer scheme that it generalizes to the outer case.
\begin{thm}[\Cref{Quillen_variation} and \Cref{outer_SB_schemes}]
Let $P\to S$ be a Severi-Brauer scheme. Then, the pushforward $f_*(\cF(P))$ already has a natural structure of an Azumaya $\cO$--algebra and when $P = \SB(\cA)$ we have $f_*(\cF(P)) \cong \cA$.

Consider a diagram
\[
\begin{tikzcd}
P \ar[r,"\tau"] \ar[d,"g"] & P^\op \ar[d,"g"] \\
L \ar[r,"i"] & L
\end{tikzcd}
\]
defining an outer Severi-Brauer scheme. Then, the isomorphism $\tau$ induces a canonical isomorphism $\cF(P)\iso \tau^*(\cF(P^\op))$ whose pushforward to $S$ is an isomorphism
\[
f_*\big(g_*(\cF(P))\big) \iso f_*\big(g_*(\tau^*(\cF(P^\op)))\big) \cong f_*\big(g_*(\cF(P^\op))\big)\otimes_{i^*} \cL
\]
which is equivalent to the information of an $\cL$--semi-linear involution
\[
\tau' \colon f_*\big(g_*(\cF(P))\big) \to f_*\big(g_*(\cF(P))\big).
\]
This makes $(L\to S, g_*(\cF(P)),\tau')$ an Azumaya algebra with involution of the second kind. 

When the outer Severi-Brauer scheme arises from $(L\to S,\cB,\tau)$, we recover $g_*(\cF(P)) \cong \cB$ and $\tau' = \tau$.
\end{thm}

In addition to our consideration of the sheaf of parabolic subgroups of $\bG$, we also consider the sheaf of pairs $\bL \subseteq \bP$ of a Levi subgroup inside a parabolic subgroup, denoted $\cPL_\bG$. Instead of flags, we show that this is isomorphic to a sheaf of idempotents from the respective algebra. For an Azumaya algebra $\cA$, a tuple of pairwise orthogonal idempotents is a tuple $(e_1,\ldots,e_{\ell+1})$ such that each $e_i^2=e_i\in \cA$, $e_ie_j=0$ for $i\neq j$, $\sum_{i=1}^{\ell+1}e_i = 1$. The notion of a lowered tuple of idempotents is similar to the notion of a lowered flag. Such a tuple is lowered if for all $T\in \Sch_S$,
\[
e_i|_T = 0 \Rightarrow e_j|_T=0 \text{ for all } j\geq i.
\] 
These tuples form a sheaf, denoted $\LowIdemp_\cA$, in which restriction of sections involves truncating zeros on the right. In the outer case, defining the suitable notion of tuples of idempotents for an Azumaya algebra with involution of the second kind $(f\colon L\to S,\cB,\tau)$ involves similar technicalities as those we glossed over when discussing $\LowGSB_{(\cB,\tau)}$ above. To give a rough idea here as well, it consists of lowered tuples of idempotent in $f_*(\cB)$ such that
\[
(e_1,\ldots,e_{\ell+1}) = (\tau(e_{\ell+1}),\ldots,\tau(e_1)).
\]
We denote the sheaf of such tuples by $\LowIdemp_{(\cB,\tau)}$. Once given a suitable tuple of pairwise orthogonal idempotents $\ve = (e_1,\ldots,e_{\ell+1})$, we may use it to define a cocharacter $\lambda_{\ve} \colon \GG_m \to \bG = \bU_{(\cB,\tau)}$ and then apply the limit subgroup construction (detailed in \Cref{limit_subgroups}) to obtain a Levi and parabolic pair $\bL(\lambda_{\ve})\subseteq \bP(\lambda_{\ve})$. We may also use it to construct a flag in $\LowGSB_{(\cB,\tau)}$ by setting
\[
\overline{\cI}_{\ve} = (0\subseteq e_1\cdot f_*(\cB) \subseteq (e_1+e_2)\cdot f_*(\cB) \subseteq \ldots \subseteq (e_1+\ldots+e_\ell)\cdot f_*(\cB) \subseteq f_*(\cB)).
\]
Finally, all of these constructions fit together nicely.
\begin{thm}[\Cref{idempotent_section}]
Let $(f\colon L\to S,\cB,\tau)$ be an Azumaya algebra with involution of the second kind. The map
\begin{align*}
\LowIdemp_{(\cB,\tau)} &\to \cPL_{\bG} \\
\ve &\mapsto (\bL(\lambda_{\ve}),\bP(\lambda_{\ve}))
\end{align*}
is an isomorphism of sheaves. It fits into a commutative diagram
\[
\begin{tikzcd}
\LowIdemp_{(\cB,\tau)} \ar[r,"\sim"] \ar[d] & \cPL_\bG \ar[d] \\
\LowGSB_{(\cB,\tau)} \ar[r,"\sim"] & \cPar_\bG
\end{tikzcd}
\]
where the downward map on the left is the construction detailed above, the downward map on the right is simply $(\bL,\bP)\mapsto \bP$, and the isomorphism on the bottom is the isomorphism of \Cref{intro_GSB_outer_iso}.
\end{thm}
}

\noindent \textbf{Acknowledgements:} I would like to express thanks to Erhard Neher who noticed the connection between idempotent elements and pairs of a Levi subgroup and a parabolic subgroup. Additionally, I would like to thank Philippe Gille for many helpful discussions about the project.

\section{Preliminaries}
We work over a fixed base scheme $S$, in particular, we work on the ringed site $(\Sch_S,\cO)$ with the fppf topology. In general, uppercase latin letters are used for schemes, calligraphic letters for sheaves on $\Sch_S$, and uppercase bold letters for sheaves of groups on $\Sch_S$.

We are primarily concerned with reductive groups of type $A_n$ which are related to Azumaya algebras of degree $n+1$, so we fix the convention that $d=n+1$.

\subsection{Locally Constant Sheaves}\label{Locally_Constant_Sheaves}
Let $n\geq 2$ be a natural number and let $N = \{1,\ldots,n\}_S := S^{\sqcup n}$ be the $S$--scheme which represents the sheaf
\begin{align}
\cN \colon \Sch_S &\to \Sets \label{constant_sheaf_n}\\
T&\mapsto \{f\colon T \to \{1,\ldots,n\} \mid f \text{ continuous}\} \nonumber
\end{align}
where $\{1,\ldots,n\}$ is given the discrete topology. In other notation, $\cN=\underline{\{1,\ldots,n\}}$, which is the sheafification of the constant presheaf of the set $\{1,\ldots,n\}$. The set $\{1,\ldots,n\}$ has an involution, i.e., an order two bijection, given by reflecting the list $(1,\ldots,n)$ about its midpoint. In particular,
\begin{align*}
\und^\vee \colon \{1,\ldots,n\} &\to \{1,\ldots,n\} \\
r &\mapsto n+1-r.
\end{align*}
It is clear that $(r^\vee)^\vee = r$. This map induces an analogous involution $\und^\vee \colon \cN \to \cN$. Here, we reuse the notation $\und^\vee$, and will do so again further below, since the particular meaning will be clear from context and all variations are related to this reflection.

By writing sets in increasing order, we identify the power set $\sP(\{1,\ldots,n\}):= \sP_n$ with the set of tuples
\[
\{(r_1,\ldots,r_\ell) \mid 1\leq \ell \leq n,\; r_i \in \{1,\ldots,n\},\; r_i < r_{i+1}\} \cup \{()\}
\]
where the empty tuple $()$ corresponds to the empty set $\O \subset \{1,\ldots,n\}$. There are two important involutions on this set. First, we have the analogue of the reflection above, namely $\und^\vee \colon \sP_n \to \sP_n$ defined by
\[
(r_1,r_2,\ldots,r_{\ell-1},r_\ell)^\vee = (n+1-r_\ell,\, n+1-r_{\ell-1}, \ldots, n+1-r_2,\, n+1-r_1).
\]
The set of symmetric tuples will be denoted by
\begin{equation}\label{eq_Sym_Pn}
\Sym(\sP_n) = \{\vr \in \sP_n \mid \vr^\vee = \vr\}.
\end{equation}
For each non-symmetric tuple $\vr$, its orbit will be $\{\vr,\vr^\vee\}$, and one of these tuples will be strictly smaller with respect to the lexicographical order on $\sP_n$. Thus, we set
\begin{equation}\label{eq_Lex_Pn}
\Lex(\sP_n) = \{\vr \in \sP_n \mid \vr <_{\text{lex}} \vr^\vee\}
\end{equation}
which is a set of representatives of the size two orbits.

Second, we have $\und^c \colon \sP_n \to \sP_n$ defined by subset compliment within $\{1,\ldots,n\}$. That is,
\[
(r_1,\ldots,r_\ell)^c = (m_1,\ldots,m_{n-\ell})
\]
if and only if $\{r_1,\ldots,r_\ell\}^c = \{m_1,\ldots,m_{n-\ell}\}$. It is also clear that $(\vr^c)^c = \vr$. Furthermore, it is easy to verify that these two operations commute;
\[
(\vr^c)^\vee = (\vr^\vee)^c.
\]

The scheme $(\sP_n)_S =  S^{\sqcup (2^n)}$ represents the sheaf
\begin{align}
\cP_n \colon \Sch_S &\to \Sets \label{eq_sheaf_cP_n} \\
T &\mapsto \{f\colon T \to \sP_n \mid f \text{ continuous}\} \nonumber
\end{align}
where $\sP_n$ is given the discrete topology. By abuse of notation, given a tuple $\vr \in \sP_n$, we will write $\vr \in \cP_n(T)$ to refer to the constant function $T \to \sP_n$ with image $\{\vr\}$. Such sections will be called \emph{constant sections}. Both the scheme $(\sP_n)_S$ and the sheaf $\cP_n$ inherit order two automorphisms corresponding to $\und^\vee$ and $\und^c$ on $\sP_n$, which we denote with the same notation.

By \cite[Exp. XXVI, 3.1]{SGA3}, given a twisted constant scheme $T\in \Sch_S$, i.e., a scheme which is locally isomorphic to $S^{\sqcup m}$ for some $m\geq 1$, we can consider the sheaf
\begin{align*}
\cOf(T) \colon \Sch_S &\to \Sets \\
T' &\mapsto \{U\subseteq T\times_S T' \mid U \text{ is both an open and closed subscheme}\},
\end{align*}
denoted with $\cOf$ from the French \emph{ouverts et ferm\'es}. This sheaf is represented by a scheme $\Of(T)$ which is also a twisted constant scheme. The construction of $\Of(T)$ is compatible with base change, $\Of(T\times_S T') = \Of(T)\times_S T'$, as well as with descent. The basic example of this construction is
\[
\Of(S^{\sqcup n}) \cong (\sP_n)_S
\]
and thus $\cOf(S^{\sqcup n}) = \cP_n$.

\subsection{Reductive Groups of Type $A_n$}\label{reductive_groups}
Following \cite[\S 3]{CF}, we recall the definitions and descriptions of groups of type $A_n$. The split examples of groups of type $A_n$ are the familiar general, special, and projective linear groups, which we view here as group sheaves.
\begin{align*}
\GL_d \colon \Sch_S &\to \Grp \\
T &\mapsto \Mat_d(\cO(T))^\times \\
\SL_d \colon \Sch_S &\to \Grp \\
T &\mapsto \{ B \in \Mat_d(\cO(T)) \mid \det(B)=1\} \\
\PGL_d \colon \Sch_S &\to \Grp \\
T &\mapsto \Aut_{\cO|_T\textrm{--alg}}(\Mat_d(\cO|_T)).
\end{align*}
These are all reductive groups while $\SL_d$ is simple simply connected and $\PGL_d$ is simple adjoint. If $\bG$ is any of the three groups above, then the group of inner automorphisms is simply
\[
\bInn(\bG) = \bG/Z(\bG) \cong \PGL_d.
\]
Thus, the inner forms of these groups are all related to $\PGL_d$--torsors, i.e., Azumaya algebras of degree $d$. Respectively, their inner forms are
\begin{align*}
\GL_{1,\cA} \colon \Sch_S &\to \Grp \\
T &\mapsto \cA(T)^\times \\
\SL_{1,\cA} \colon \Sch_S &\to \Grp \\
T &\mapsto \{ a \in \cA(T) \mid \Nrd_\cA(a)=1\} \\
\PGL_\cA \colon \Sch_S &\to \Grp \\
T &\mapsto \Aut_{\cO|_T\textrm{--alg}}(\cA|_T)
\end{align*}
for some Azumaya algebra $\cA$ of degree $d$. If $\cA = \cEnd(\cE)$ for an $\cO$--module $\cE$, we will also write $\GL_\cE$, $\SL_\cE$, and $\PGL_\cE$ for these groups.

Any group automorphism of $\GL_d$, $\SL_d$, or $\PGL_d$ induces a graph automorphism of the Dynkin diagram
\[
\begin{tikzpicture}
\filldraw[black] (-5,0) circle (2pt);
\draw (-5,0) -- (-4,0);
\filldraw[black] (-4,0) circle (2pt);
\draw (-4,0) -- (-3.5,0);
\node at (-3,0) {\ldots};
\draw (-2.5,0)--(-2,0);
\filldraw[black] (-2,0) circle (2pt);
\draw (-2,0)--(-1,0);
\filldraw[black] (-1,0) circle (2pt);
\node at (-5,-0.3) {1};
\node at (-4,-0.3) {2};
\node at (-2,-0.3) {n-1};
\node at (-1,-0.3) {n};
\end{tikzpicture}
\]
of type $A_n$. An automorphism is inner if and only if it induces the trivial graph automorphism. Excluding the the case of type $A_1$, where the graph has no non-trivial automorphisms, we have an action of $\ZZ/2\ZZ$ on these groups by outer automorphisms. For $d\geq 3$, the automorphism
\begin{align}
\theta \colon \GL_d &\to \GL_d \label{GLd_outer_aut} \\
B &\mapsto (B^t)^{-1} \nonumber
\end{align}
is outer and of order two, and thus we use it to define the $\ZZ/2\ZZ$--action on $\GL_d$. This action immediately restricts to an action on $\SL_d$, and since $\theta$ stabilizes the center it also uniquely induces an automorphism $\theta'\colon \PGL_d \to \PGL_d$ which behaves locally as $\theta'(\Inn(B)) = \Inn((B^t)^{-1})$. In particular, we can consider the group
\[
\PGL_d \rtimes \ZZ/2\ZZ
\]
where the semi-direct product is defined by $\ZZ/2\ZZ$ action on $\PGL_d$ via $\theta'$. The following description is rephrased from \cite[3.4.0.81]{CF}.
\begin{thm}\label{split_group_automorphisms}
Let $\bG$ be one of $\GL_d$, $\SL_d$, or $\PGL_d$ for $d\geq 3$. Then,
\[
\bAut(\bG) = \PGL_d \rtimes \ZZ/2\ZZ
\]
where $\PGL_d$ acts by inner automorphism and $\ZZ/2\ZZ$ acts by $\theta'$.
\end{thm}

To describe non-split groups of type $A_n$, including the outer twists, this automorphism group is related to Azumaya algebras with involution of the second type. These are objects related to degree $2$ \'etale covers of the base scheme $S$, which are scheme morphisms $L\to S$ which are locally isomorphic to the cover by the disjoint union $S\sqcup S \to S$. This disjoint cover is called the split \'etale cover or the trivial \'etale cover. A degree $2$ \'etale cover $L\to S$ comes with a unique order two $S$--automorphism $i\colon L \iso L$ which is locally isomorphic to the switch $\sw \colon S\sqcup S \iso S\sqcup S$. Since such a cover $f \colon L\to S$ is an affine morphisms, under the correspondence of \cite[Tag 01SA]{Stacks} it corresponds to a commutative $\cO$--algebra
\[
\cL = f_*(\cO|_L) \colon \Sch_S \to \Ab
\]
which is locally isomorphic to $\cO\times\cO$ and the $\cO$--algebra structure map $\cO \to \cL$ is locally isomorphic to the diagonal. The isomorphism $i\colon L \to L$ then appears as an order two isomorphism $i^* \colon \cL \to \cL$ locally isomorphic to the switch
\begin{align*}
\sw \colon \cO\times \cO &\iso \cO\times\cO \\
(a,b) &\mapsto (b,a)
\end{align*}
and it therefore fixes $\cO \subset \cL$. Given a $\cO|_L$--algebra $\cB$ over $L$, the pushforward $f_*(\cB)$ is naturally a $\cL$--algebra and of course also a $\cO$--algebra.

\begin{defn}
Using notation as above, an Azumaya algebra with involution of the second kind is $(f\colon L\to S,\cB,\tau)$ where
\begin{enumerate}
\item $L\to S$ is a degree $2$ \'etale cover, 
\item $\cB$ is an Azumaya $\cO|_L$--algebra on $L$, and
\item $\tau \colon f_*(\cB) \to f_*(\cB)$ is an $i^*$--semilinear involution of the $\cL$--algebra $f_*(\cB)$. That is,
\[
\tau(cb) = i^*(c)\tau(b)
\]
for $c\in \cL$ and $b\in f_*(\cB)$.
\end{enumerate}
An isomorphism $(g,\varphi)\colon (f_1\colon L_1\to S,\cB_1,\tau_1) \iso (f_2\colon L_2\to S,\cB_2,\tau_2)$ is given by
\begin{enumerate}\setcounter{enumi}{3}
\item an $S$--isomorphism $g\colon L_1 \iso L_2$ (which is automatically compatible with $i_1$ and $i_2$), equivalently an $\cO$--algebra isomorphism $g^*\colon \cL_2 \iso \cL_1$, and
\item $\varphi \colon f_{2*}(\cB_2) \iso f_{1*}(\cB_1)$ an isomorphism of $\cO$--algebras such that
\begin{align*}
\varphi(cb) &= g^*(c)\varphi(b), \text{ and}\\
\varphi(\tau_2(b)) &= \tau_1(\varphi(b))
\end{align*}
for $c\in \cL_2$ and $b\in f_{2*}(\cB_2)$. That is, $\varphi$ is compatible with both $g^*$ and the involutions $\tau_1$ and $\tau_2$.
\end{enumerate}
\end{defn}

The split example is $(S\sqcup S \to S,(\Mat_d(\cO),\Mat_d(\cO)),\tau_d)$, where $\tau_d$ is the switch transpose
\begin{align}
\tau_d \colon \Mat_d(\cO)\times\Mat_d(\cO) &\to \Mat_d(\cO)\times\Mat_d(\cO) \label{split_second_involution} \\
(B_1,B_2) &\mapsto (B_2^t,B_1^t). \nonumber
\end{align}
which is clearly compatible with the switch automorphism $\sw \colon \cO\times \cO \iso \cO\times\cO$. By \cite[2.7.0.35]{CF}, all Azumaya algebras with involution of the second kind are twisted forms over $S$ of this split example. The automorphism group of this split Azumaya algebra with involution of the second type is also
\[
\PGL_d \rtimes \ZZ/2\ZZ
\]
where the non-trivial constant section $\overline{1} \in \ZZ/2\ZZ$ acts by the switch on both $\cO\times\cO$ and $\Mat_d(\cO)\times\Mat_d(\cO)$, and where $\varphi \in \PGL_d$ acts trivially on $\cO\times\cO$ and by
\[
\varphi \cdot (B_1,B_2) = (\varphi(B_1),\theta'(\varphi)(B_2))
\]
for $(B_1,B_2)\in \Mat_d(\cO)\times\Mat_d(\cO)$. Furthermore, we can identify $\GL_d$ and $\SL_d$ within this algebra, namely for $T\in \Sch_S$
\begin{align*}
\GL_d(T) &= \{ (B,(B^t)^{-1}) \in \Mat_d(\cO(T))\times\Mat_d(\cO(T))\},\\
\SL_d(T) &= \{ (B,(B^t)^{-1}) \in \Mat_d(\cO(T))\times\Mat_d(\cO(T)) \mid \det(B) = 1\},
\end{align*}
and this identification causes the action of $\PGL_d\rtimes \ZZ/2\ZZ$ on $\Mat_d(\cO)\times\Mat_d(\cO)$ to restrict to its action on $\GL_d$ and $\SL_d$ by group automorphisms. Thus, all twisted forms of these groups will be found within an Azumaya algebra with involution of the second kind. In particular, they will be one of the following groups.
\begin{defn}
Let $(f\colon L\to S,\cB,\tau)$ be an Azumaya algebra with involution of the second type. Its \emph{unitary group} is
\begin{align*}
\bU_{(\cB,\tau)} \colon \Sch_S &\to \Grp \\
T &\mapsto \{ b \in f_*(\cB)(T) \mid \tau(b)=b^{-1}\}
\end{align*} 
and its \emph{special unitary group} is
\begin{align*}
\bSU_{(\cB,\tau)} \colon \Sch_S &\to \Grp \\
T &\mapsto \{ b \in f_*(\cB)(T) \mid \tau(b)=b^{-1},\Nrd_{f_*(\cB)}(b)=1\}.
\end{align*}
\end{defn}
\begin{thm}
Let $\bG$ be a form of $\GL_d$. Then $\bG = \bU_{(\cB,\tau)}$ for some Azumaya algebra with involution of the second type $(L\to S,\cB,\tau)$. Likewise, if $\bG$ is a form of $\SL_d$, then it is some $\bSU_{(\cB,\tau)}$.
\end{thm}
\begin{proof}
These results are contained in \cite[3.5.0.87]{CF} and \cite[3.5.0.92]{CF}.
\end{proof}

This result includes the case of inner forms of these groups. When twisting by inner automorphisms, the split \'etale extension $S\sqcup S\to S$ is not affected, and thus $(S\sqcup S \to S,(\Mat_d(\cO),\Mat_d(\cO),\tau_d))$ twists into $(S\sqcup S\to S,(\cA,\cA^\op),\sw)$ for an Azumaya $\cO$--algebra $\cA$ of degree $d$. Here, the switch involution acts as $\sw(a_1,a_2^\op) = (a_2,a_1^\op)$. Then we have the identifications
\begin{align*}
\GL_{1,\cA} &= \bU_{((\cA,\cA^\op),\sw)}, \text{ and}\\
\SL_{1,\cA} &= \bSU_{((\cA,\cA^\op),\sw)}.
\end{align*}

\begin{rem}
Technically, in \cite{CF} the authors use $(\Mat_d(\cO),\Mat_d(\cO)^\op)$ in their split example, and then the involution is simply the switch $(B_1,B_2^\op) \mapsto (B_2,B_1^\op)$ as in the case of an inner form. However, since the transpose provides an isomorphism $\Mat_d(\cO) \iso \Mat_d(\cO)^\op$, as well as in the reverse direction, this is equivalent to what we have stated above. We find that using the transpose makes the connection to $\theta$ and $\theta'$ clearer.
\end{rem}

\subsection{Hermitian Forms}
Sufficiently locally, such that the algebra becomes an endomorphism algebra, Azumaya algebras with involution of the second kind are related to regular $1$--hermitian forms. Let $f\colon L \to S$ be a degree $2$ \'etale cover with its canonical order two isomorphism $i \colon L \iso L$, equivalently $i^* \colon \cL \to \cL$. Given a finite locally free $\cO|_L$--module $\cH$, the pushforward $f_*(\cH)$ is a $\cL$--module and a \emph{hermitian form} on $\cH$ is a map
\[
h\colon f_*(\cH)\times f_*(\cH) \to \cL
\]
which is sesquilinear, i.e. additive in both arguments and such that for $c_1,c_2 \in \cL$ and $x,y\in f_*(\cH)$ we have
\[
h(c_1x,c_2y) = i^*(c_1)c_2 h(x,y).
\]
For $\lambda \in \cL(S)$ such that $\lambda \cdot i^*(\lambda)=1$, a $\lambda$--hermitian form further satisfies the condition that $i^*(h(x,y))=\lambda h(y,x)$ and therefore a $1$--hermitian form satisfies
\[
i^*(h(x,y)) = h(y,x).
\]
From now on we exclusively consider $1$--hermitian forms and outside of this subsection we will refer to them simply as hermitian forms. Finally, any hermitian form induces a map of $\cO$--modules
\begin{align*}
\tilde{h} \colon f_*(\cH) &\to \cHom_{\cL}(f_*(\cH),\cL) \\
x &\mapsto h(x,\und).
\end{align*}
The hermitian form is called \emph{regular} if $\tilde{h}$ is an isomorphism. Given a regular $1$--hermitian form, it has a involution of the second kind
\[
\tau_h \colon f_*(\cEnd_{\cO|_L}(\cH)) \to f_*(\cEnd_{\cO|_L}(\cH)),
\]
noting that $f_*(\cEnd_{\cO|_L}(\cH))\cong \cEnd_\cL(f_*(\cH))$ by \cite[Tag 01SB]{Stacks}, where $\tau_h$ is defined uniquely by requiring that
\[
h(x,\varphi(y)) = h(\tau_h(\varphi)(x),y)
\]
for all $\varphi \in \cEnd_\cL(f_*(\cH))$ and all $x,y\in f_*(\cH)$. 

Conversely, if $(L\to S,\cB,\tau)$ is an Azumaya algebra with involution of the second kind and $\cB = \cEnd_{\cO|_L}(\cH)$ for an $\cO|_L$--module $\cH$, then there exists a regular $1$--hermitian form $h\colon f_*(\cH)\times f_*(\cH) \to \cL$, unique up to a scalar in $\cL^\times$, such that $\tau= \tau_h$.

Given a regular hermitian form $(f\colon L \to S,\cH,h)$, setting
\begin{align*}
\GL_{(\cH,h)} \colon \Sch_S &\to \Grp \\
T &\mapsto \{\varphi \in f_*(\GL_\cH)(T) \mid h(\varphi(x),\varphi(y))=h(x,y)\}
\end{align*}
defines the general linear group of $(\cH,h)$. Then, there is a canonical isomorphism $\GL_{(\cH,h)}\cong \bU_{(\cEnd(\cH),\tau_h)}$.

The split example of a regular $1$--hermitian form is the following. On the trivial \'etale cover $S\sqcup S \to S$ we take the module $\cH_d = (\cO^d,\cO^d)$ where we view sections of $\cO^d$ as column vectors. The hermitian form is then defined by
\begin{align}
h_d\colon (\cO^d\times\cO^d) \times (\cO^d\times\cO^d) &\to \cO_S\times\cO_S \label{eq_split_hermitian_form} \\
((x,y),(z,w)) &\mapsto (y^t z, w^t x) \nonumber
\end{align}
where $\und^t$ denotes transpose. The form has automorphism group $\GL_d\rtimes \ZZ/2\ZZ$ which acts as follows. For a section $a\in \GL_d$, it acts as the identity on the \'etale cover $S\sqcup S \to S$ and as
\[
a\cdot(x,y) = (ax,(a^{-1})^t y)
\]
on the module. The non-trivial constant element $\overline{1}\in \ZZ/2\ZZ$ acts as the switch automorphism both on $S\sqcup S$ and on the module so that
\[
\overline{1}\cdot (x,y) = (y,x).
\]
It is clear this is a well-defined action since $\ZZ/2\ZZ$ acts on $\GL_d$ by the inverse-transpose. This action identifies the subgroup $\GL_d < \GL_d\rtimes \ZZ/2\ZZ$ with $\GL_{(\cH_d,h_d)}$. 

Any regular $1$--hermitian form of constant rank $d$ is locally isomorphic to the split example above, and therefore is the twist by a $(\GL_d\times \ZZ/2\ZZ)$--torsor of the split example. Thus, for any other such hermitian form $(L\to S,\cH,h)$, we have that
\[
(L\to S,\cH,h) \cong \cK \wedge^{\GL_d\rtimes\ZZ/2\ZZ} (S\sqcup S \to S,\cH_d,h_d)
\]
for some torsor $\cK$.

\subsection{Dynkin Schemes}
Consider the group $\bU(\cB,\tau)$ for an algebra with involution of the second type $(L\to S,(\cB,\tau))$. This corresponds to a $(\PGL_d\rtimes \ZZ/2\ZZ)$--torsor, which we denote by $\cK$. Since $\ZZ/2\ZZ$ acts on $\PGL_d$ by the outer automorphism $\theta'$, which corresponds to the reflection of the Dynkin diagram of type $A_n$, the map $\bAut(\PGL_d) \surj \bOut(\PGL_d)$ corresponds to the canonical projection $\PGL_d\rtimes \ZZ/2\ZZ \to \ZZ/2\ZZ$. The pushforward of $\cK$ along this canonical map yields a $\ZZ/2\ZZ$--torsor $\cL$ which corresponds to the \'etale cover $L\to S$. In particular, this means that $\cL \wedge^{\ZZ/2\ZZ} (S\sqcup S) \cong L$.

As in \cite[Exp. XXIV, 3.3]{SGA3}, we set the Dynkin scheme of $\GL_d$ to be $\Dyn(\GL_d)= S^{\sqcup n}$. As in \Cref{Locally_Constant_Sheaves}, $S^{\sqcup n}$ represents the constant sheaf $\cN$ of \eqref{constant_sheaf_n}, and correspondingly we can think of the components of $S^{\sqcup n}$ as being labelled by the vertices of the Dynkin diagram $A_n$, i.e.
\[
\Dyn(\GL_d) = \bigsqcup_{i=1}^n S_i
\]
where each $S_i = S$. The $\ZZ/2\ZZ$--action on $\cN$ by $\und^\vee \colon \cN \to \cN$ corresponds to the $\ZZ/2\ZZ$--action on $\Dyn(\GL_d)$ which acts as the switch on the pairs $S_r \sqcup S_{n+1-r}$ when $r\neq n+1-r$, and in the case where $n=2m+1$ is odd, it simply acts as the identity on the middle component $S_{m+1}$ since $(m+1)^\vee = m+1$. Continuing to follow \cite[Exp. XXIV, 3.3]{SGA3}, we set
\[
\Dyn(\bG) = \cL \wedge^{\ZZ/2\ZZ} \Dyn(\GL_d).
\]

\begin{lem}
Let $\bG = \bU(\cB,\tau)$ be the group of type $A_n$ associated to $(L\to S,(\cB,\tau))$. Let $\cL$ be the $\ZZ/2\ZZ$--torsor associated to $L\to S$ as reviewed above. Then,
\[
\Dyn(\bG) \cong \begin{cases} L_{(1,n)} \sqcup L_{(2,n-1)} \sqcup \ldots \sqcup L_{(m,m+1)} & n=2m \text{ is even,} \\ L_{(1,n)} \sqcup L_{(2,n-1)} \sqcup \ldots \sqcup L_{(m,m+2)} \sqcup S_{m+1} & n=2m+1 \text{ is odd.} \end{cases}
\]
where each $L_{(r,r^\vee)}=L$.
\end{lem}
\begin{proof}
In the case that $n=2m$ is even, we write
\[
\Dyn(\GL_d) = (S_1 \sqcup S_n) \sqcup (S_2 \sqcup S_{n-1}) \sqcup \ldots \sqcup (S_m \sqcup S_{m+1})
\]
where now the $\ZZ/2\ZZ$--action acts as the switch on each pair. Thus, $\cL \wedge^{\ZZ/2\ZZ} \Dyn(\GL_d)$ is $m$ copies of the twist $\cL \wedge^{\ZZ/2\ZZ} (S\sqcup S) = L$. Thus it is clear we obtain the result stated above. Of course, when $n=2m+1$ is odd the picture is identical outside of the trivial action of $\ZZ/2\ZZ$ on $S_{m+1}$, and thus the result is clear in this case as well.
\end{proof}
The action of $\ZZ/2\ZZ$ on $\Dyn(\GL_d)$ clearly commutes with itself, so there is also a descended $\ZZ/2\ZZ$ action on $\Dyn(\bG)$ by an order two isomorphism which we denote
\[
\Dyn(i) \colon \Dyn(\bG) \iso \Dyn(\bG).
\]
By virtue of how $\Dyn(\bG)$ was defined, this action preserves each $L$ component of $\Dyn(\bG)$ and on such a component it locally agrees with the switch on $S\sqcup S$. The canonical $S$--automorphism $i\colon L \iso L$ of the \'etale cover $L\to S$ is exactly the twist by $\cL$ of the switch, so the map $\Dyn(i)$ appears as
\[
\begin{tikzcd}[column sep=-1ex]
\Dyn(\bG) \ar[d,"\Dyn(i)"] & = & L \ar[d,"i"] & \sqcup & L \ar[d,"i"] & \sqcup & \ldots & \sqcup & L \ar[d,"i"] & \sqcup & S \ar[d,"\Id"] \\
\Dyn(\bG)& = & L & \sqcup & L & \sqcup & \ldots & \sqcup & L & \sqcup & S 
\end{tikzcd}
\]
where the final $S$ component only occurs if $n$ is odd.

Because the sheaf $\cN$ is represented by the scheme $\Dyn(\GL_d)$, its twist by $\cL$ will be represented by $\Dyn(\bG)$. Thus, we denote
\[
\cDyn(\bG) = \cL \wedge^{\ZZ/2\ZZ} \cN
\]
to be the twisted sheaf. The $\ZZ/2\ZZ$--action on $\cDyn(\bG)$ corresponding to $\Dyn(i)$ locally agrees with $\und^\vee$, so we also denote it by
\[
\und^\vee \colon \cDyn(\bG) \iso \cDyn(\bG).
\]

We will actually make more use of the scheme of open and closed subschemes of the Dynkin scheme of $\bG$. As mentioned in \cite[Exp. XXVI, 3.1]{SGA3}, when $T \in \Sch_S$ is a twisted constant scheme, the construction of the scheme of open and closed subschemes $\Of(T)$ commutes with base change and is compatible with descent, thus
\begin{align*}
\Of(\Dyn(\bG)) &= \Of(\cL \wedge^{\ZZ/2\ZZ} \Dyn(\GL_d)) \\
&= \cL \wedge^{\ZZ/2\ZZ} \Of(\Dyn(\GL_d)) \\
&= \cL \wedge^{\ZZ/2\ZZ} (\sP_n)_S
\end{align*}
where $\ZZ/2\ZZ$ acts on $(\sP_n)_S$ via $\und^\vee$. In order to track components, we write
\[
(\sP_n)_S = \bigsqcup_{\vr \in \sP_n} S_{\vr}
\]
as a labeled disjoint union where each $S_{\vr} = S$. In order to track components in the following lemma, we recall the definition of $\Sym(\sP_n)$ from \eqref{eq_Sym_Pn} and of $\Lex(\sP_n)$ from \eqref{eq_Lex_Pn}.

\begin{lem}\label{cOD_description}
Let $\bG = \bU(\cB,\tau)$ be the group of type $A_n$ associated to $(L\to S,(\cB,\tau))$. Let $\cL$ be the $\ZZ/2\ZZ$--torsor associated to $L\to S$. Then,
\[
\Of(\Dyn(\bG)) = \left( \bigsqcup_{\vr \in \Lex(\sP_n)} L_{(\vr,\vr^\vee)} \right) \sqcup \left( \bigsqcup_{\vr\in \Sym(\sP_n)} S_{\vr} \right)
\]
where each $L_{(\vr,\vr^\vee)} = L$ and each $S_{\vr} = S$.
\end{lem}
\begin{proof}
Since $\Of(\Dyn(\bG)) = \cL \wedge^{\ZZ/2\ZZ} (\sP_n)_S$, we write
\[
(\sP_n)_S = \left( \bigsqcup_{\vr \in \Lex(\sP_n)} S_{\vr}\sqcup S_{\vr^\vee} \right) \sqcup \left( \bigsqcup_{\vr\in \Sym(\sP_n)} S_{\vr} \right)
\]
where $\ZZ/2\ZZ$ acts as the switch on each pair in the $\Lex(\sP_n)$ component, and trivially on each copy of $S$ in the $\Sym(\sP_n)$ component. Then, using the fact that $\cL \wedge^{\ZZ/2\ZZ} (S\sqcup S) = L$, we obtain the result.
\end{proof}

Since the sheaf $\cP_n$ is represented by the scheme $(\sP_n)_S$, the twisted sheaf $\cL \wedge^{\ZZ/2\ZZ} \cP_n$ is represented by $\Of(\Dyn(\bG))$. Thus, we denote
\begin{equation}
\cOD{\bG} = \cL \wedge^{\ZZ/2\ZZ} \cP_n \label{eq_OD_G}
\end{equation}
to be this twisted sheaf. Since both of the involutions $\und^\vee$ and $\und^c$ on $\sP_n$ commute with $\und^\vee$ which was used for twisting, the involutions on $\cP_n$ descend to involutions
\begin{align*}
\und^\vee \colon \cOD{\bG} &\iso \cOD{\bG} \\
\und^c \colon \cOD{\bG} &\iso \cOD{\bG}
\end{align*}
on the twisted sheaf. As for the case of $\Dyn(\bG)$ discussed above, the involution $\und^\vee$ on $\cOD{\bG}$ corresponds to the order two $S$--isomorphism
\begin{equation}\label{eq_OfDyn_i}
\Of(\Dyn(i)) \colon \Of(\Dyn(\bG)) \iso \Of(\Dyn(\bG))
\end{equation}
which acts as $i\colon L_{(\vr,\vr^\vee)} \iso L_{(\vr,\vr^\vee)}$ on each $L$ component and as the identity on the $S_{\vr}$ components.

\subsection{Parabolics and Levi Subgroup}\label{prelim_parabolics_and_levis}
\subsubsection{Parabolic Subgroups}
Following \cite[\S 1, Exp. XXVI]{SGA3}, the parabolic subgroups of a reductive group are those closed subgroups which are smooth over the base scheme and which contain a Borel subgroup. For a group $\bG \colon \Sch_S \to \Grp$ of type $A_n$, we denote by
\[
\cPar_\bG \colon \Sch_S \to \Sets
\]
the sheaf of parabolic subgroups in $\bG$. By \cite[3.3(ii), Exp. XXVI]{SGA3} the sheaf $\cPar_\bG$ is representable and we denote this representing scheme by $\Par(\bG) \in \Sch_S$. This sheaf of parabolic subgroups comes with the \emph{type morphism} of \cite[Exp. XXVI, 3.2]{SGA3} which is a map
\begin{equation}\label{type_map_parabolics}
t \colon \cPar_\bG \to \cOD{\bG}.
\end{equation}
The corresponding map of representing schemes $\Par(\bG) \to \Of(\Dyn(\bG))$ is smooth, projective, and has geometrically integral fibers, also by \cite[3.3(ii), Exp. XXVI]{SGA3}. We write $\cPar_{\vr,\bG}$ for the fiber of $t$ over a section $\vr\in \cOD{\bG}$. 

To add some intuition for what this type morphism is doing, we can look at how it behaves locally when $\bG$ split. For $\bG = \GL_d$ with $d=n+1$, this is of course a group of inner type $A_n$, and so we know from \Cref{cOD_description} that
\[
\cOD{\bG} = \cP_n
\]
which is the constant sheaf associated to the power set $\sP_n$ on $n$ elements. First, we focus on standard parabolics, i.e., those of the form
\begin{align}
\bP \colon \Sch_S &\to \Grp \nonumber \\
T &\mapsto \left\{ \begin{bmatrix} B_1 & * & \hdots & * \\ 0 & B_2 & \hdots & * \\ \vdots & \vdots & \ddots & \vdots \\ 0 & 0 & \hdots & B_k \end{bmatrix} \in \GL_d(T) \mid B_i \in \GL_{m_i}(T) \right\} \label{standard_parabolic}
\end{align}
for a fixed ordered partition $m_1+m_2+\ldots+m_k = d$, written with square blocks on the diagonal, zeroes below the diagonal blocks, and arbitrary entries above them. The standard realization of the root system of type $A_n$ within $\RR^d$ has the system of simple roots
\[
\{\alpha_1 = e_1-e_2, \alpha_2 = e_2-e_3,\ldots, \alpha_n = e_n-e_{n+1}\}
\]
which are associated to the upper off diagonal entries of a matrix. Thus, the negative roots correspond to the lower off diagonal and in particular to the root subgroups
\begin{align*}
\bU_{-\alpha_i} \colon \Sch_S &\to \Grp \\
T &\mapsto \{ I+cE_{i+1,i} \mid c\in \cO(T)\}
\end{align*}
where $I$ is the identity matrix and $E_{i+1,i}$ has a $1$ in the $(i+1)^\text{st}$ row and $i^\text{th}$ column. A parabolic of the form above is generated by the upper triangular matrices along with all negative root subgroups which it contains. For $1\leq i \leq n$, the root subgroup $\bU_{-\alpha_i}$ is contained in the parabolic $\bP$ above if and only if $i$ is not one of the partial sums $m_1+\ldots+m_j$ arising from the partition of $d$. Thus, by forming the increasing tuple of partial sums and taking the compliment,
\[
(m_1, m_1+m_2, \ldots, \sum_{j=1}^{k-2}m_j, \sum_{j=1}^{k-1}m_j)^c,
\]
we obtain a list of the negative root subgroups which generate $\bP$. This constant section of $\cP_n$ is defined to be the type of $\bP$ under the morphism \eqref{type_map_parabolics}.

\begin{example}
The parabolic subgroup
\[
\bP = \left\{ \begin{bmatrix} * & * & * & * \\ * & * & * & * \\ 0 & 0 & * & * \\ 0 & 0 & * & *\end{bmatrix} \in \GL_4 \right\}
\]
of $\GL_4$ has $t(\bP)=(1,3)$ since exactly the negative root subgroups
\begin{align*}
\bU_{-\alpha_1} &= \left\{ \begin{bmatrix} 1 & 0 & 0 & 0 \\ c & 1 & 0 & 0 \\ 0 & 0 & 1 & 0 \\ 0 & 0 & 0 & 1 \end{bmatrix} \mid c \in \cO \right\}, \text{ and} \\
\bU_{-\alpha_3} &= \left\{ \begin{bmatrix} 1 & 0 & 0 & 0 \\ 0 & 1 & 0 & 0 \\ 0 & 0 & 1 & 0 \\ 0 & 0 & c & 1 \end{bmatrix} \mid c \in \cO \right\}
\end{align*}
are included in $\bP$. 
\end{example}

For a general parabolic subgroup $\bP \in \cPar_{\GL_d}(T)$, it is conjugate to a unique parabolic subgroup $\bP'$ which contains the Borel subgroup of upper triangular matrices. Then, sufficiently locally, say over a cover $\{U_i \to T\}_{i\in I}$, the subgroup $\cP'|_{U_i}$ will become a standard parabolic and have a constant type $t(\cP'|_{U_i}) \in \cP_n(U_i)$. These local types will glue to produce $t(\bP') \in \cP_n(T)$, and the type morphism sets
\[
t(\bP) = t(\bP').
\]
This describes the type map $t\colon \cPar_{\GL_d} \to \cP_n$. The automorphism group $\bAut(\GL_d)$ also acts naturally on $\cPar_{\GL_d}$. Since the type of a parabolic is already defined by comparing it to a conjugate, inner automorphisms of $\GL_d$ do not change the type of a parabolic subgroup. However, the outer automorphism $\theta \colon \GL_d \to \GL_d$ of \eqref{GLd_outer_aut} sends a standard parabolic generated by negative root subgroups $\bU_{-\alpha_{i_1}},\ldots,\bU_{-\alpha_{i_k}}$ to the opposite parabolic, i.e., containing the lower triangular matrices and generated by the positive root subgroups $\bU_{\alpha_{i_1}},\ldots,\bU_{\alpha_{i_k}}$. Conjugating this opposite parabolic by the matrix with ones on the second diagonal produces the standard parabolic generated by
\[
\bU_{-\alpha_{n+1-i_k}},\ldots,\bU_{-\alpha_{n+1-i_1}},
\]
meaning that $t(\theta(\bP)) = t(\bP)^\vee$ for all parabolic subgroups. Thus, the type morphism
\[
t\colon \cPar_{\GL_d} \to \cP_n
\]
is $\ZZ/2\ZZ$--equivariant and respects the surjection $\bAut(\GL_d)=\PGL_d\rtimes \ZZ/2\ZZ \surj \ZZ/2\ZZ$. Thus, for an arbitrary group $\bG$ of type $A_n$, the type morphism
\[
t\colon \cPar(\bG) \to \cOD{\bG}
\]
is the twist of the type morphism in the split case by the associated $\PGL_d\rtimes \ZZ/2\ZZ$--torsor.

\subsubsection{Levi Subgroups}
Continuing to follow \cite[\S 1, Exp. XXVI]{SGA3}, Levi subgroups are defined in the context of a reductive groups $\bG$ and a parabolic subgroup $\bP \subseteq \bG$. The parabolic subgroup contains its unipotent radical $\rad_u(\bP)\subset \bP$ and then $\bP/\rad_u(\bP)$ is reductive by definition. A \emph{Levi subgroup of $\bP$} is any reductive subgroup $\bL \subseteq \bP$ such that
\[
\bP = \rad_u(\bP)\rtimes \bL, 
\]
equivalently such that the composition $\bL \inj \bP \surj \bP/\rad_u(\bP)$ is an isomorphism. Levi subgroups of $\bP$ are precisely linked with the maximal tori of $\bP$.
\begin{prop}[{\cite[1.6 (ii), Exp. XXVI]{SGA3}}]
Let $\bG$ be a reductive group scheme over $S$ and $\bP \subseteq \bG$ a parabolic subgroup. For each maximal torus $\bT \subseteq \bP$ there exists a reductive subgroup $\bL \subseteq \bP$ such that
\begin{enumerate}
\item $\bL$ is the largest reductive subgroup of $\bP$ which contains $\bT$, and
\item $\bP = \rad_u(\bP)\rtimes \bL$.
\end{enumerate}
\end{prop}
As a corollary of the above proposition, we know that for such subgroups $\bT \subseteq \bL \subseteq \bP$ we have the following identities on normalizers.
\begin{align*}
\bNorm_\bP(\bL) &= \bL \text{, and} \\
\bNorm_\bP(\bT) &= \bNorm_\bL(\bT).
\end{align*}
Finally, we point out that by \cite[Corollaire 1.8, Exp. XXVI]{SGA3}, any two Levi subgroups of $\bP$ are conjugate by an element of $\bP(S)$, in fact more precisely they are conjugate by a unique element of $\rad_u(\bP)(S)$.

As an example, consider the standard parabolic subgroup of $\bP \subseteq \GL_d$ from \eqref{standard_parabolic},
\begin{align*}
\bP \colon \Sch_S &\to \Grp \\
T &\mapsto \left\{ \begin{bmatrix} B_1 & * & \hdots & * \\ 0 & B_2 & \hdots & * \\ \vdots & \vdots & \ddots & \vdots \\ 0 & 0 & \hdots & B_k \end{bmatrix} \in \GL_d(T) \mid B_i \in \GL_{m_i}(T) \right\}
\end{align*}
for a fixed ordered partition $m_1 + m_2 + \ldots + m_k = d$. This $\bP$ contains the standard maximal torus of diagonal matrices in $\GL_d$ and the corresponding Levi subgroup is
\begin{align*}
\bL \colon \Sch_S &\to \Grp \\
T &\mapsto \left\{ \begin{bmatrix} B_1 & 0 & \hdots & 0 \\ 0 & B_2 & \hdots & 0 \\ \vdots & \vdots & \ddots & \vdots \\ 0 & 0 & \hdots & B_k \end{bmatrix} \in \GL_d(T) \mid B_i \in \GL_{m_i}(T) \right\}
\end{align*}
and so $\bL \cong \GL_{m_1}\times \GL_{m_2}\times\ldots\times \GL_{m_k}$. All other Levi subgroups of $\bP$ are then conjugates of this one.

\subsubsection{Limit Subgroups}\label{limit_subgroups}
Given a cocharacter $\lambda \colon \GG_m \to \bG$ for a linear algebraic group $\bG$, we can construct an associated parabolic and Levi pair in the following way. For any scheme $T\in \Sch_S$, denote by $T_m$ the scheme which represents $\GG_m|_T$ and denote by $T_a$ the scheme which represents $\GG_a|_T$. Thus, we get affine morphisms of schemes
\[
\begin{tikzcd}
T_m \ar[dr] \ar[rr] & & T_a \ar[dl] \\
& T &
\end{tikzcd}
\]
dual to the diagram of $\cO|_T$--algebras
\[
\begin{tikzcd}
\cO|_T[t,t^{-1}] & & \cO|_T[t] \ar[ll] \\
 & \cO|_T. \ar[ur] \ar[ul] & 
\end{tikzcd}
\]
Since $\GG_m(T_m) = (\cO|_T[t,t^{-1}])^\times$, it contains the section $t$ and so we have a section $\lambda(t) \in \bG(T_m)$. We now define the \emph{limit subgroup} $\bP(\lambda)\colon \Sch_S \to \Grp$ to behave as
\[
T \mapsto \{h\in \bG(T) \mid \lambda(t)h|_{T_m}\lambda(t)^{-1} \in \Img(\bG(T_a)\xrightarrow{\textrm{res}}\bG(T_m))\}.
\]
for $T\in \Sch_S$. This is a rephrasing of the definition of limit subgroup used in \cite[7.1]{Gille}.

The limit subgroup $\bP(\lambda) \subseteq \bG$ is a parabolic subgroup. Setting $\bL(\lambda) = \bCent_\bG(\lambda)$ to be the centralizer we obtain a parabolic--Levi pair $(\bP(\lambda),\bL(\lambda))$.

\begin{example}\label{example_limit_computation}
Consider the cocharacter $\lambda \colon \GG_m \to \GL_3$ defined by
\[
\lambda(a) = \begin{bmatrix} a^{-1} & & \\ & a^{-2} & \\ & & a^{-2} \end{bmatrix}.
\]
For a scheme $T\in \Sch_S$, we simply have that
\begin{align*}
\GL_3(T) &= \nGL(\cO(T)), \\
\GL_3(T_a) &= \nGL(\cO(T)[t]), \text{ and} \\
\GL_3(T_m) &= \nGL(\cO(T)[t,t^{-1}])
\end{align*}
where the restriction maps are the obvious injections. Thus, considering a generic matrix $B= [b_{ij}]_{i,j=1}^3 \in \nGL(\cO(T))$, we can compute that
\[
\lambda(t)B\lambda(t)^{-1} = \begin{bmatrix} b_{11} & t b_{12} & t b_{13} \\ t^{-1} b_{21} & b_{22} & b_{23} \\ t^{-1} b_{31} & b_{32} & b_{33} \end{bmatrix}.
\]
This belongs to the image of $\GL_3(T_a) \to \GL_3(T_m)$ if and only if the terms involving $t^{-1}$ do not occur, i.e., if $b_{21} = b_{31} = 0$. Thus,
\[
\bP(\lambda)(T) = \left\{\begin{bmatrix} b_{11} & b_{12} & b_{13} \\ 0 & b_{22} & b_{23} \\ 0 & b_{32} & b_{33} \end{bmatrix} \in \GL_3(T) \right\}.
\]
\end{example}

A useful fact is that conjugates of limit subgroups correspond to conjugates of the cocharacter.
\begin{lem}\label{conjugating_limits_Levis}
Let $\lambda \colon \GG_m \to \bG$ be a cocharacter and let $(\bP(\lambda),\bL(\lambda))$ be the associated limit parabolic and Levi subgroups. For a section $g\in \bG(S)$, we have that
\[
(\bP(g\lambda g^{-1}),\bL(g\lambda g^{-1})) = (g\bP(\lambda)g^{-1},g\bL(\lambda)g^{-1}).
\]
\end{lem}
\begin{proof}
The claim about Levi subgroups is immediate since
\[
\bCent_\bG(g\lambda g^{-1}) = g\bCent_\bG(\lambda)g^{-1}.
\]
To check that the parabolics are equal, consider $T\in \Sch_S$. The group $\bP(g\lambda g^{-1})(T)$ is
\[
\left\{ h \in \bG(T) \mid (g\lambda g^{-1})(t)h|_{T_m} (g\lambda g^{-1})(t^{-1}) \in \Img(\bG(T_a) \to \bG(T_m))\right\}.
\]
Considering the expression $(g\lambda g^{-1})(t)h|_{T_m} (g\lambda g^{-1})(t^{-1})$ and omitting the restriction notation, this is
\[
g \lambda(t) g^{-1} h g \lambda(t^{-1}) g^{-1} \in \Img(\bG(T_a) \to \bG(T_m)).
\]
Since $g\in \bG(T)$, we have that $g|_{T_m} \in \Img(\bG(T_a) \to \bG(T_m))$ and thus stabilizes this image. This means that
\[
\lambda(t) g^{-1} h g \lambda(t^{-1}) \in \Img(\bG(T_a) \to \bG(T_m)),
\]
i.e., that $g^{-1}hg \in \bP(\lambda)(T)$. Thus,
\begin{align*}
\bP(g\lambda g^{-1})(T) &= \{ h \in \bG(T) \mid g^{-1}hg \in \bP(\lambda)(T)\} \\
&= \{h\in \bG(T) \mid h\in g\bP(\lambda)(T) g^{-1}\} \\
&= (g\bP(\lambda)g^{-1})(T)
\end{align*}
and so $\bP(g\lambda g^{-1}) = g\bP(\lambda)g^{-1}$ as parabolic subgroups of $\bG$.
\end{proof}

\section{Flags, Generalized Severi-Brauer Sheaves, and Parabolic Subgroups: The Inner Case}\label{inner_case}
Flags of subspaces in a vector space over a field are well studied objects with close ties to linear algebraic groups and their homogeneous spaces. However, naive extensions of the definition to submodules of vector bundles over schemes may encounter problems. For example, consider a scheme $S$ and let $\cE$ be a finite locally free $\cO$--module of constant rank $d$. A flag in $\cE$ is an ordered set of submodules
\[
\cV_0 = 0 \subseteq \cV_1 \subseteq \cV_2 \subseteq \ldots \subseteq \cV_\ell \subseteq \cE
\]
for some integer $\ell \geq 0$ where each $\cV_i$ is a finite locally free $\cO$--module which is locally a direct summand of $\cE$, i.e., there is a cover $\{U_i \to S\}_{i\in I}$ over which we can write $\cE|_{U_i}=\cV_j|_{U_i}\oplus \cW_{ij}$ for each $i\in I$ and $1\leq j\leq \ell$. We fix the convention that $\cV_0 = 0$ and thus will omit writing $\cV_0$. We allow the empty flag, $0 \subseteq \cE$, which is the unique flag with $\ell=0$. Ultimately, we aim to show that the scheme of flags in $\cE$ is isomorphic to the scheme of parabolic subgroups in $\GL_\cE$, with a flag being mapped to its stabilizer subgroup. However, this will not produce an isomorphism with our current notion of flag. A trivial example is that the flags
\[
0 \subseteq \cV \subseteq \cE \text{ and } 0 \subseteq \cV \subseteq \cV \subseteq \cE
\]
will of course have the same stabilizer subgroup. So, we need to avoid duplicates in our flags. However, say the scheme $S = S_1 \sqcup S_2$ has two disjoint components and so $\cE = (\cE_1,\cE_2)$ where each $\cE_i$ is a finite locally free $\cO|_{S_i}$--module of constant rank $d$. Then, for $\cV \subsetneq \cV' \subseteq \cE_1$ and $\cW \subseteq \cE_2$, the two flags
\[
0 \subseteq (\cV,\cW) \subseteq (\cV',\cE_2) \subseteq (\cE_1,\cE_2) \text{ and } 0 \subseteq (\cV,\cW) \subseteq (\cV',\cW) \subseteq (\cE_1,\cE_2)
\]
also have the same stabilizer subgroup in $\GL_\cE = (\GL_{\cE_1},\GL_{\cE_2})$, despite neither flag containing any duplicates.

\subsection{Lowered Flags}
To remedy the problems discussed above, we introduce the notion of a \emph{lowered flag}.
\begin{defn}\label{defn_lowered_flag}
Let $\cE$ be a finite locally free $\cO$--module. A flag
\[
0 \subseteq \cV_1 \subseteq \ldots \subseteq \cV_\ell \subseteq \cE
\]
is called a \emph{lowered flag} if, for all $T\in \Sch_S$ and $i\in \{0,\ldots,\ell-1\}$,
\[
\cV_i|_T = \cV_{i+1}|_T \;\Rightarrow\; \cV_i|_T = \cE|_T = \cV_{i+1}|_T. 
\]
\end{defn}
Returning to the examples above, the flag
\[
0 \subseteq (\cV,\cW) \subseteq (\cV',\cW) \subseteq (\cE_1,\cE_2)
\]
is not lowered if $\cW\neq \cE_2$ since restricting this flag to $S_2 \in \Sch_S$ produces $0 \subseteq \cW \subseteq \cW \subseteq \cE_2$, which contains a non-$\cE_2$ duplicate. On the other hand, (with a few extra assumptions on $\cV$ and $\cV'$), the flag
\[
0 \subseteq (\cV,\cW) \subseteq (\cV',\cE_2) \subseteq (\cE_1,\cE_2)
\]
is lowered.

The notion of a lowered flag is stable under base change.
\begin{lem}
Let $\cE$ be a finite locally free $\cO$--module. Let $0 \subseteq \cV_1 \subseteq \ldots \subseteq \cV_\ell \subseteq \cE$ be a lowered flag. Then, for all $T \in \Sch_S$, the flag
\[
0 \subseteq \cV_1|_T \subseteq \ldots \subseteq \cV_\ell|_T \subseteq \cE|_T
\]
is also a lowered flag in $\cE|_T$.
\end{lem}
\begin{proof}
If we have $Y \in \Sch_T$ and $i$ such that
\[
(\cV_i|_T)|_Y = (\cV_{i+1}|_T)|_Y,
\]
this is equivalent to $\cV_i|_Y = \cV_{i+1}|_Y$ where $Y\in \Sch_S$, and so $\cV_{i-1}|_Y = \cE|_Y = \cV_i|_Y$. Thus, the restricted flag is lowered as claimed.
\end{proof}

We will restrict our attention to lowered flags in modules of constant rank.
\begin{lem}\label{bounding_flag_length}
Let $\cE$ be a finite locally free $\cO$--module of constant rank $d$. If
\[
0 \subseteq \cV_1 \subseteq \ldots \subseteq \cV_\ell \subseteq \cE
\]
is a lowered flag with $\cV_\ell \neq \cE$, then $\ell < d$.
\end{lem}
\begin{proof}
Since each $\cV_j$ is a finite locally free $\cO$--module, we may find a cover $\{U_i \to S\}_{i\in I}$ over which all $\cV_j|_{U_i}$ are constant rank. Consider the restriction of the flag
\[
0 \subsetneq \cV_1|_{U_i} \subsetneq \ldots \subsetneq \cV_\ell|_{U_i} \subsetneq \cE|_{U_i}
\]
to some $U_i$. If for some $j$, $\rank(\cV_j|_{U_i})=\rank(\cV_{j+1}|_{U_i})$, then we must have that $\cV_j|_{U_i} = \cV_{j+1}|_{U_i} = \cE|_{U_i}$ since the flag is lowered. Thus, the local ranks form a sequence which strictly increases until it reaches $d$. This imposes the condition that
\[
j \leq \rank(\cV_j|_{U_i})
\]
for each $j$. Hence, if $\ell \geq d$, then $d \leq \rank(\cV_\ell|_{U_i}) \leq d$ for all $i\in I$ and so $\cV_\ell|_{U_i}=\cE|_{U_i}$. This means that $\cV_\ell = \cE$ globally, which contradicts the assumption that $\cV_\ell \neq \cE$. Therefore, we conclude that $\ell < d$, finishing the proof.
\end{proof} 

Let $\cE$ be of constant rank $d$. By considering $S\in \Sch_S$ and applying the definition, we see that a lowered flag must be of the form
\[
0 \subsetneq \cV_1 \subsetneq \ldots \subsetneq \cV_\ell \subsetneq \cE \subseteq \cE \subseteq \ldots \subseteq \cE,
\]
with $\ell < d$ by \Cref{bounding_flag_length}. Here $\cV_i \subsetneq \cV_{i+1}$ simply means that $\cV_i\subseteq \cV_{i+1}$ as $\cO$--modules and there exists some $T\in \Sch_S$ where $\cV_i(T) \neq \cV_{i+1}(T)$. Of course, the redundant occurrences of $\cE$ on the right can be truncated so we only consider lowered flags with a single $\cE$. The only cost to be paid when truncating is a bit of care when defining base change.
\begin{defn}\label{defn_flag_sheaf}
Let $\cE$ be a finite locally free $\cO$--module with constant rank $d$. We define the functor
\begin{align*}
\LowFlag_\cE \colon \Sch_S &\to \Sets \\
T &\mapsto \left\{ 0 \subseteq \cV_1 \subseteq \ldots \subseteq \cV_\ell \subseteq \cE|_T \mid \begin{array}{l} \ell\geq 0,\; \cV_\ell \neq \cE|_T, \text{ each } \cV_j \text{ is} \\ \text{locally a direct summand, and} \\ \text{the flag is lowered}\end{array}\right\}.
\end{align*}
whose restriction map with respect to a morphism $Y \to T$ in $\Sch_S$ behaves as
\[
(0 \subseteq \cV_1 \subseteq \ldots \subseteq \cV_\ell \subseteq \cE|_T) \mapsto (0 \subseteq \cV_1|_Y \subseteq \ldots \subseteq \cV_k|_Y \subseteq \cE|_Y)
\]
where $k \in \{1,\ldots,\ell\}$ is the greatest index such that $\cV_k|_Y \neq \cE|_Y$.
\end{defn}

\begin{prop}\label{lowered_flags_sheaf}
Let $\cE$ be a finite locally free $\cO$--module with constant rank $d$. The functor $\LowFlag_\cE\colon \Sch_S \to \Sets$ defined above is a sheaf. Furthermore, for $T\in \Sch_S$ it satisfies
\[
(\LowFlag_\cE)|_T = \LowFlag_{(\cE|_T)}.
\]
\end{prop}
\begin{proof}
The fact that $(\LowFlag_\cE)|_T = \LowFlag_{(\cE|_T)}$ is clear from \Cref{defn_flag_sheaf} (and does not rely on $\cE$ being constant rank). Thus, it is sufficient to show that $\LowFlag_\cE$ satisfies the sheaf condition for a cover $\{U_i \to S\}_{i\in I}$ of $S$.

First, we argue that it is separated. Assume that
\[
0 \subseteq \cV_1 \subseteq \ldots \subseteq \cV_\ell \subseteq \cE \text{ and } 0 \subseteq \cW_1 \subseteq \ldots \subseteq \cW_m \subseteq \cE
\]
are two sections of $\LowFlag_\cE(S)$ which agree over every $U_i$. Without loss of generality $\ell \leq m$, and so this means that
\[
(0 \subseteq \cV_1|_{U_i} \subseteq \ldots \subseteq \cV_{p_i}|_{U_i} \subseteq \cE|_{U_i}) = (0 \subseteq \cW_1|_{U_i} \subseteq \ldots \subseteq \cW_{q_i}|_{U_i} \subseteq \cE|_{U_i})
\]
where $p_i$ and $q_i$ are the appropriate largest indices. In particular $p_i=q_i$, and so if $\ell < m$, then $q_i \leq \ell < m$ which means that all restrictions of $\cW_m$ are being truncated, i.e., that $\cW_m|_{U_i}=\cV|_{U_i}$ for all $i\in I$. However, this means that $\cW_m=\cV$ in which case the second section was not a lowered flag. This contradiction means that $\ell = m$. Then, for each $i\in I$ the local equality gives us that
\[
\cV_k|_{U_i} = \cW_k|_{U_i}
\]
if $k\leq p_i$ or $\cV_k|_{U_i}=\cE|_{U_i}=\cW_k|_{U_i}$ if $k>p_i$, and thus $\cV_k = \cW_k$ globally for all $k \in \{1,\ldots,\ell\}$. Hence, the two sections of $\LowFlag_\cE(S)$ are in fact equal.

Now, we argue that $\LowFlag_\cE$ allows gluing. Let
\[
(0 \subseteq \cV_{1,i} \subseteq \ldots \subseteq \cV_{\ell_i,i} \subseteq \cE|_{U_i}) \in \LowFlag_\cE(U_i)
\]
be sections which agree on overlaps. By \Cref{bounding_flag_length}, each $\ell_i < d$ and so we can set $m = \max(\{\ell_i \mid i\in I\})$. Now, for $k\in \{1,\ldots,m\}$ we set
\[
\cW_{k,i} = \begin{cases} \cV_{k,i} & \text{if } k\leq \ell_i \\ \cE|_{U_i} & \text{if } k > \ell_i. \end{cases}
\]
From the assumption that our sections agree on overlaps, we see that $\cW_{k,i}|_{U_{ij}} = \cW_{k,j}|_{U_{ij}}$ as well, as so there exists a glued submodule $\cW_k \subseteq \cE$ such that $\cW_k|_{U_i} = \cW_{k,i}$ for each $i\in I$. It is clear that each $\cW_k$ produced this way is locally a direct summand of $\cE$ since the $\cV_{k,i}$ are locally direct summands of $\cE|_{U_i}$ by assumption. Additionally, $\cW_m \neq \cE$ since we have $\ell_i=m$ for at least one $i\in I$. Now, our candidate flag is
\[
0 \subseteq \cW_1 \subseteq \ldots \subseteq \cW_m \subseteq \cE,
\]
but we must show that it is lowered. Let $T\in \Sch_S$ and $k$ be such that $\cW_k|_T = \cW_{k+1}|_T$. Restricting the cover of $S$, we get the cover $\{T_i = T\times_S U_i \to T\}_{i\in I}$ of $T$. Restricting to each $T_i$, we have
\[
(\cW_k|_T)|_{T_i} = (\cW_k|_{U_i})|_{T_i} = \begin{cases} \cV_{k,i}|_{T_i} & \text{if } k\leq \ell_i \\ \cE|_{T_i} & \text{if } k > \ell_i. \end{cases}
\]
and likewise for $\cW_{k+1}|_{T_i}$. If $k\geq \ell_i$, we have that $\cW_k|_{T_i} = \cE|_{T_i} = \cW_{k+1}|_{T_i}$ immediately, and when $k< \ell_i$ we obtain that
\[
\cV_{k,i}|_{T_i} = \cW_k|_{T_i} = \cW_{k+1}|_{T_i} = \cV_{k+1,i}|_{T_i}
\]
which forces everything to be $\cE|_{T_i}$ in this case as well since $0 \subseteq \cV_{1,i} \subseteq \ldots \subseteq \cV_{\ell_i,i} \subseteq \cE|_{U_i}$ is lowered. Therefore, we conclude that
\[
\cW_k|_T = \cE|_T = \cW_{k+1}|_T
\]
also over $T$, meaning that $0 \subseteq \cW_1 \subseteq \ldots \subseteq \cW_m \subseteq \cE$ is lowered. Hence,
\[
(0 \subseteq \cW_1 \subseteq \ldots \subseteq \cW_m \subseteq \cE) \in \LowFlag_\cE(S)
\]
is a section which glues the compatible sections we started with. This shows $\LowFlag_\cE$ is a sheaf and finishes the proof.
\end{proof}

If one wishes to work with flags of constant rank, they may consider the presheaf of such objects $\ConFlag_\cE \colon \Sch_S \to \Sets$ which is defined on $T\in \Sch_S$ by
\begin{equation}\label{eqn_con_flags}
T \mapsto \left\{ 0 \subseteq \cV_1 \subseteq \ldots \subseteq \cV_\ell \subseteq \cE|_T \mid \begin{array}{l}\ell\geq 0,\; \cV_\ell \neq \cE|_T, \text{ each } \cV_i \text{ is locally\ } \\ \text{a direct summand, is constant rank,} \\ \text{and }\rank(\cV_i)<\rank(\cV_{i+1}) \end{array}\right\}
\end{equation}
with naturally defined restrictions since all the restrictions of any such flag of length $\ell$ will also have length $\ell$. However, this also means that $\ConFlag$ is not a sheaf, since for example on a disconnected scheme $S=S_1 \sqcup S_2$ sections of the form
\[
(0 \subseteq \cV \subseteq \cV' \subseteq \cE|_{S_1}) \in \ConFlag_\cE(S_1) \text{ and } (0 \subseteq \cW \subseteq \cE|_{S_2}) \in \ConFlag_\cE(S_2)
\]
will trivially agree on the empty overlap, but there cannot exist a section of $\ConFlag_\cV(S)$ which restricts to the above two since they have differing lengths. Thus, we consider the sheafification of $\ConFlag_\cE$.

\begin{prop}\label{sheafification_of_ConFlag}
Let $\cE$ be a finite locally free $\cO$--module of constant rank $d$. Consider the presheaf $\ConFlag_\cE$ defined above and the sheaf $\LowFlag_\cE$ of \Cref{defn_lowered_flag}. There is a natural injection of presheaves
\begin{align*}
\ConFlag_\cE &\inj \LowFlag_\cE \\
(0 \subseteq \cV_1 \subseteq \ldots \subseteq \cV_\ell \subseteq \cE) & \mapsto (0 \subseteq \cV_1 \subseteq \ldots \subseteq \cV_\ell \subseteq \cE)
\end{align*}
which satisfies the universal property of sheafification. In particular, it shows that
\[
(\ConFlag_\cE)^\sharp \cong \LowFlag_\cE.
\]
\end{prop}
\begin{proof}
It is clear that we have the injection given above since the condition on the ranks of a section $(0 \subseteq \cV_1 \subseteq \ldots \subseteq \cV_\ell \subseteq \cE) \in \ConFlag_\cE$ guarantees that
\[
\cV_i|_T \neq \cV_{i+1}|_T
\]
for all $T\in \Sch_S$, and thus the flag is lowered.

To show that this injection satisfies the universal property of sheafification, we must show that any section of $\LowFlag_\cE$ locally belongs to the image of $\ConFlag_\cE \inj \LowFlag_\cE$. Thus, consider a section
\[
(0 \subseteq \cV_1 \subseteq \ldots \subseteq \cV_\ell \subseteq \cE|_T) \in \LowFlag_\cE(T)
\]
for some $T\in \Sch_S$. Since each $\cV_k$ is finite locally free, we can find a cover $\{T_i \to T\}_{i\in I}$ over which all $\cV_k$ simultaneously become of constant rank. However, it is then clear that each restriction
\[
(0 \subseteq \cV_1 \subseteq \ldots \subseteq \cV_\ell \subseteq \cE|_T)|_{T_i} = (0 \subseteq \cV_1|_{T_i} \subseteq \ldots \subseteq \cV_{p_i}|_{T_i} \subseteq \cE|_{T_i})
\]
is a flag of constant rank submodules and thus belongs to the image of $\ConFlag_\cE(T_i)\inj \LowFlag_\cE(T_i)$. Hence, $(\ConFlag_\cE)^\sharp \cong \LowFlag_\cE$ as claimed and we are done.
\end{proof}

\begin{rem}\label{non-truncated_flags}
Instead of choosing to truncate our lowered flags and paying the associated costs, one can instead do the following. Let $\cE$ be a finite locally free $\cO$--module of constant rank $d$. In light of \Cref{bounding_flag_length}, any truncated lowered flag has length $\ell < d$, and so they can all be uniquely extended to a lowered flag of length exactly $d-1$. Thus, defining the functor
\begin{align*}
\cF_\cE \colon \Sch_S &\to \Sets \\
T &\mapsto \{ 0 \subseteq \cV_1 \subseteq \ldots \subseteq \cV_{d-1} \subseteq \cE|_T \mid \text{the flag is lowered} \}
\end{align*}
one obtains a clear isomorphism
\begin{align*}
\LowFlag_\cE &\iso \cF_\cE \\
(0 \subseteq \cV_1 \subseteq \ldots \subseteq \cV_\ell \subseteq \cE) &\mapsto (0 \subseteq \cV_1 \subseteq \ldots \subseteq \cV_\ell \subseteq \cE \subseteq \cE \subseteq \ldots \subseteq \cE).
\end{align*}
Working with $\cF_\cE$, the proof that it is a sheaf is simpler than the proof given for \Cref{lowered_flags_sheaf} since we do not need to track changing lengths of flags. However, the map $\ConFlag_\cE \inj \cF_\cE$ can no longer be seen as a literal subsheaf, and instead also requires the extension of flags on the right. Later on, we will define the type morphisms for lowered flags, which will be a map $\LowFlag_\cE \to \cP_n$ based on the ranks of the components. If one does this naively with $\cF_\cE$, the resulting tuples of the form $(r_1,r_2,\ldots,r_\ell,d,d,\ldots,d)$ do not occur in $\cP_n$, so truncation would be required at this point, or a discussion of the largest component of the flag not equal to $\cE$, or one could replace $\cP_n$ with a sheaf of lowered tuples, etc. For these later applications, we choose to include the truncation in the definition of $\LowFlag_\cE$.
\end{rem}

\begin{rem}\label{raised_flags}
One may also work with the symmetric notion of \emph{raised flags}, where a flag
\[
0 \subseteq \cV_1 \subseteq \ldots \subseteq \cV_\ell \subseteq \cE
\]
is raised if for all $T\in \Sch_S$ and $i$ such that $\cV_{i-1}|_T = \cV_i|_T$, then $\cV_{i-1}|_T = 0 = \cV_i|_T$. In this case, raised flags truncate on the left after restrictions and thus it is more natural to index flags in the reverse direction, i.e., as
\[
0 \subseteq \cV_{\ell} \subseteq \cV_{\ell-1} \subseteq \ldots \subseteq \cV_1 \subseteq \cE.
\]
Raised flags similarly form a sheaf, which we denote $\RaiFlag_\cE$. Flags in $\ConFlag_\cE$ are also raised flags, and so when $\cE$ is of constant rank $d$ there is an analogue of \Cref{sheafification_of_ConFlag} saying that the injection
\[
\ConFlag_\cE \inj \RaiFlag_\cE
\]
has the universal property of sheafification and so $(\ConFlag_\cE)^\sharp \cong \RaiFlag_\cE$ as well. Of course, this means we obtain a commutative diagram
\[
\begin{tikzcd}
\LowFlag_\cE \ar[rr,"\varphi_{\cFlag}"] \ar[dr,hookleftarrow] & & \RaiFlag_\cE \\
 & \ConFlag_\cE \ar[ur,hookrightarrow] & 
\end{tikzcd}
\]
where $\varphi_{\cFlag}$ is a uniquely defined isomorphism of sheaves. Because this (or a similar) raising isomorphism will be of use later, we take a moment here to detail how this isomorphism behaves.

Let $(0\subseteq \cV_1 \subseteq \ldots \subseteq \cV_\ell \subseteq \cE)\in \LowFlag_\cE(S)$ be a global section. Since vector bundles on schemes are Zariski locally of constant rank, we can find a Zariski cover $S=\bigcup_{i\in I} U_i$ over which all $\cV_j|_{U_i}$ are constant rank. Since the flag is lowered, for each $i\in I$ the restriction appears as
\[
0 \subseteq \cV_1|_{U_i} \subseteq \ldots \subseteq \cV_{k_i}|_{U_i} \subseteq \cE|_{U_i}
\]
for some integer $k_i \leq \ell$, and
\[
\rank_{\cO|_{U_i}}(\cV_j|_{U_i}) < \rank_{\cO|_{U_i}}(\cV_{j+1}|_{U_i})
\]
for all $0\leq j \leq k_i$. This means that for any further restriction, i.e. along any $T\to U_i$, the length of the restricted flag will remain $k_i$. Thus, we see that the length of the flag is also locally constant and furthermore is constant over each $U_i$. Therefore, we may group the $U_i$ together by length of the restriction, defining $W_j = \bigcup_{i\in I \atop k_i=j} U_i$, to obtain a finite disjoint Zariski cover
\[
S = W_{j_1} \sqcup W_{j_2} \sqcup \ldots \sqcup W_{j_m} \sqcup W_\ell
\]
for a subset $\{j_1,\ldots,j_m\} \subseteq \{1,\ldots,\ell-1\}$. By construction, the restriction of the flag to $W_{j_i}$ has length $j_i$. Any global submodule $\cV\subset \cE$ can then be written as a tuple of its restrictions to these $W_{j_i}$, i.e.,
\[
\cV = (\cV|_{W_{j_1}},\ldots,\cV|_{W_{j_m}},\cV|_{W_\ell}).
\]
Doing this for the components of the flag (and writing it vertically), we see that we obtain a staircase pattern
\[
\begin{tikzpicture}
\matrix (M)[matrix of math nodes]{
\cE & = & (\cE|_{W_{j_1}} & \cE|_{W_{j_2}} & \ldots & \cE|_{W_{j_m}} & \cE|_{W_\ell}) \\
\cV_\ell & = & (\cE|_{W_{j_1}} & \cE|_{W_{j_2}} & \ldots & \cE|_{W_{j_m}} & \cV_\ell|_{W_\ell}) \\
\vdots & \quad & & & \vdots & & \\
\cV_{j_m} & = & (\cE|_{W_{j_1}} & \cE|_{W_{j_2}} & \ldots & \cV_{j_m}|_{W_{j_m}} & \cV_{j_m}|_{W_\ell}) \\
\vdots & \quad & & & \vdots & & \\
\vdots & \quad & & & \vdots & & \\
\cV_{j_2} & = & (\cE|_{W_{j_1}} & \cV_{j_2}|_{W_{j_2}} & \ldots & \cV_{j_2}|_{W_{j_m}} & \cV_{j_2}|_{W_\ell}) \\
\vdots & \quad & & & \vdots & & \\
\cV_{j_1} & = & (\cV_{j_1}|_{W_{j_1}} & \cV_{j_1}|_{W_{j_2}} & \ldots & \cV_{j_1}|_{W_{j_m}} & \cV_{j_1}|_{W_\ell}) \\
\vdots & \quad & & & \vdots & & \\
\cV_1 & = & (\cV_1|_{W_{j_1}} & \cV_1|_{W_{j_2}} & \ldots & \cV_1|_{W_{j_m}} & \cV_1|_{W_\ell}) \\
0 & = & (0 & 0 & \ldots & 0 & 0) \\
};
\draw[thick,black](M-8-2.south east)--(M-9-4.north west);
\draw[thick,black](M-9-4.north west)--(M-7-4.north west);
\draw[thick,black](M-7-4.north west)--(M-7-4.north east);
\draw[thick,black](M-7-4.north east)--([xshift=-0.9ex,yshift=-1ex]M-6-5.north west);
\draw[thick,black]([xshift=-0.9ex,yshift=-1ex]M-6-5.north west)--([xshift=1ex,yshift=-1ex]M-6-5.north east);
\draw[thick,black]([xshift=1ex,yshift=-1ex]M-6-5.north east)--(M-4-6.north west);
\draw[thick,black](M-4-6.north west)--(M-4-7.north west);
\draw[thick,black](M-4-7.north west)--([xshift=-0.7ex]M-2-7.north west);
\draw[thick,black]([xshift=-0.7ex]M-2-7.north west)--(M-2-7.north east);
\end{tikzpicture}
\]
above which we simply have restrictions of $\cE$ and below which we have proper submodules. The raising isomorphism $\varphi_{\cFlag}$ then slides the non-trivial part of the flag over each component to the top, augmenting with zeros below. That is, it produces the flag
\[
\begin{tikzpicture}
\matrix (M)[matrix of math nodes]{
\cE & = & (\cE|_{W_{j_1}} & \cE|_{W_{j_2}} & \ldots & \cE|_{W_{j_m}} & \cE|_{W_\ell}) \\
\cV'_\ell & = & (\cV_{j_1}|_{W_{j_1}} & \cV_{j_2}|_{W_{j_2}} & \ldots & \cV_{j_m}|_{W_{j_m}} & \cV_\ell|_{W_\ell})\\
\vdots & \quad & & & \vdots & & \\
\cV'_{\ell+1-{j_1}} & = & (\cV_1|_{W_{j_1}} & \cV_{j_2+1-j_1}|_{W_{j_2}} & \ldots & \cV_{j_m+1-j_1}|_{W_{j_m}} & \cV_{\ell+1-j_1}|_{W_\ell})\\
\vdots & \quad & & & \vdots & & \\
\cV'_{\ell+1-{j_2}} & = & (0 & \cV_1|_{W_{j_2}} & \ldots & \cV_{j_m+1-j_2}|_{W_{j_m}} & \cV_{\ell+1-j_2}|_{W_\ell} )\\
\vdots & \quad & & & \vdots & & \\
\vdots & \quad & & & \vdots & & \\
\cV'_{\ell+1-{j_m}} & = & (0 & 0 & \ldots & \cV_1|_{W_{j_m}} & \cV_{\ell+1-j_m}|_{W_\ell})\\
\vdots & \quad & & & \vdots & & \\
\cV'_1 & = & (0 & 0 & \ldots & 0 & \cV_1|_{W_\ell})\\
0 & = & (0 & 0 & \ldots & 0 & 0) \\
};
\draw[thick,black](M-4-3.south west)--(M-4-4.south west);
\draw[thick,black](M-4-4.south west)--([xshift=-3ex]M-6-4.south west);
\draw[thick,black]([xshift=-3ex]M-6-4.south west)--([xshift=3ex]M-6-4.south east);
\draw[thick,black]([xshift=3ex]M-6-4.south east)--([xshift=-1ex,yshift=-0.5ex]M-7-5.south west);
\draw[thick,black]([xshift=-1ex,yshift=-0.5ex]M-7-5.south west)--([xshift=1ex,yshift=-0.5ex]M-7-5.south east);
\draw[thick,black]([xshift=1ex,yshift=-0.5ex]M-7-5.south east)--([xshift=-3.4ex]M-9-6.south west);
\draw[thick,black]([xshift=-3.4ex]M-9-6.south west)--([yshift=-0.2ex]M-9-7.south west);
\draw[thick,black]([yshift=-0.2ex]M-9-7.south west)--([xshift=-3ex]M-11-7.south west);
\draw[thick,black]([xshift=-3ex]M-11-7.south west)--(M-11-7.south east);
\end{tikzpicture}
\]
with the zeroes below the (horizontally reflected) staircase pattern and non-zero submodules above. Note that the $W_\ell$ column is fixed and otherwise the order of the modules in each column is preserved.
\end{rem}

\begin{lem}\label{GL_action_on_LowFlag}
Let $\cE$ be a finite locally free $\cO$--module of constant rank $d$. The natural action of $\GL_\cE$ on $\cE$ induces an action on $\LowFlag_\cE$. For $B \in \GL_\cE$, it behaves as
\[
B\cdot (0 \subseteq \cV_1 \subseteq \ldots \subseteq \cV_\ell \subseteq \cE) = (0 \subseteq B\cV_1 \subseteq \ldots \subseteq B\cV_\ell \subseteq \cE).
\]
\end{lem}
\begin{proof}
Whenever we have $\cE|_T = \cV_i|_T \oplus \cW$ for some $T\in \Sch_S$ and $1\leq i \leq \ell$, then we also have $\cE|_T = B|_T\cV_i|_T \oplus B|_T\cW$, so this action preserves the property of being locally a direct summand. Thus, we only need to check that the resulting section is still a lowered flag, then it is obvious that this is the group action induced as claimed. However, since $B$ is invertible, if for some $T\in \Sch_S$ we have
\[
(B\cV_i)|_T = (B\cV_{i+1})|_T,
\]
then by applying $B^{-1}|_T$ we have that $\cV_i|_T = \cV_{i+1}|_T$. This implies that $\cV_i|_T = \cE|_T \cV_{i+1}|_T$, and so reapplying $B|_T$ yields
\[
(B\cV_i)|_T = \cE|_T = (B\cV_{i+1})|_T.
\]
Thus, the resulting flag is also lowered and we are done.
\end{proof}

\subsection{Generalized Severi-Brauer Sheaves}
When working with an Azumaya algebra of degree $d$, the Severi-Brauer scheme can be defined as the scheme representing the sheaf of right ideals of rank $d$, i.e., of reduced rank $1$ (see \Cref{Severi-Brauer_review}). The word generalized means that we instead allow flags of right ideals, and therefore to obtain a sheaf with must consider lowered flags of right ideals.
\begin{defn}\label{generalized_Severi-Brauer}
Let $\cA$ be an Azumaya $\cO$--algebra of constant degree $d=n+1$. Define $\LowGSB_\cA$ to be the subsheaf of $\LowFlag_\cA$ consisting of lowered flags
\[
0 \subseteq \cI_1 \subseteq \cI_2 \subseteq \ldots \subseteq \cI_\ell \subseteq \cA
\]
where each $\cI_j$ is a right ideal of $\cA$. Each $\cI_j$ must also locally be a direct summand of $\cA$, as this condition is imposed by $\LowFlag_\cA$. Symmetrically, define $\RaiGSB_\cA \subseteq \RaiFlag_\cA$ to be the subsheaf of raised flags of right ideals in $\cA$.
\end{defn}
Since the only additional condition imposed on the modules in a flag, namely being a right ideal, can be checked locally, it is clear that $\LowGSB_\cA$ is in fact a sheaf. For a flag of right ideals $(0 \subseteq \cI_1 \subseteq \cI_2 \subseteq \ldots \subseteq \cI_\ell \subseteq \cA) \in \LowGSB_\cA(S)$, each $\cI_j$ is a finite locally free $\cO$--module which is locally a direct summand of $\cA$ since this condition is imposed by $\LowFlag_\cA$. However, the property of being a right ideal of $\cA$ imposes a further restriction on the rank of $\cI_j$. In particular, $\rank(\cI_j)$ is divisible by $d$, i.e., for any $T\in \Sch_S$ for which $\cI|_T$ is constant rank, $\rank(\cI|_T)=dr_j$ for some $0\leq r_j \leq d$. This can be seen by localizing sufficiently far such that $\cA$ becomes a matrix algebra where the description of direct summand right ideals is analogous to the case over a field.

\begin{rem}\label{separable_algebras}
Within a separable $\cO$--algebra, i.e., an algebra which is locally isomorphic to a finite direct product of Azumaya algebras, we can equally well discuss flags of right ideals. This will come up later when we consider a degree $2$ \'etale cover $f\colon L\to S$ and an Azumaya algebra $\cB$ over $L$. Then, the pushforward $f_*(\cB)$ is a separable $\cO$--algebra which is not an Azumaya $\cO$--algebra, but nevertheless we can discuss $\LowGSB_{f_*(\cB)}$ and $\RaiGSB_{f_*(\cB)}$.
\end{rem}

\begin{lem}
Let $\cA$ be an Azumaya $\cO$--algebra of constant degree $d$. For any section
\[
(0 \subseteq \cI_1 \subseteq \cI_2 \subseteq \ldots \subseteq \cI_\ell \subseteq \cA) \in \LowGSB_\cA(S)
\]
we have $\ell < d$.
\end{lem}
\begin{proof}
Since the given flag of right ideals is also a section of $\LowFlag_\cA(S)$, we know that $\cI_\ell \neq \cA$. Now, assume $\ell \geq d$. Consider a cover $\{U_i \to S\}_{i\in I}$ over which all $\cI_j$ become of constant rank. Each flag
\[
0 \subseteq \cI_1|_{U_i} \subseteq \cI_2|_{U_i} \subseteq \ldots \subseteq \cI_\ell|_{U_i} \subseteq \cA|_{U_i}
\]
is a lowered flag of right ideals of constant rank, and so $\rank(\cI_j|_{U_i}) \geq \min(dj,d^2)$ since any potential equalities between components can only occur on the right. Thus, if $\ell \geq d$, then $\rank(\cI_\ell|_{U_i})=d^2$ which means that $\cI_\ell|_{U_i}=\cA|_{U_i}$ for all $i \in I$. Of course, this means that $\cI_\ell = \cA$ globally, giving a contradiction. Thus, $\ell < d$ as claimed.
\end{proof}

It will also be helpful to view $\LowGSB_\cA$ as the sheafification of a particular presheaf. We define
\[
\ConIdeal_\cA \subset \ConFlag_\cA
\]
to be the subfunctor of flags consisting of constant rank, locally direct summand submodules which are also right ideals.
\begin{lem}\label{sheafification_of_ConIdeal}
There is a commutative diagram of presheaves
\[
\begin{tikzcd}
\ConIdeal_\cA \ar[d,hookrightarrow] \ar[r,hookrightarrow] & \LowGSB_\cA \ar[d,hookrightarrow] \\
\ConFlag_\cA \ar[r,hookrightarrow] & \LowFlag_\cA
\end{tikzcd}
\]
where the injection $\ConIdeal_\cA \inj \LowGSB_\cA$ has the univerisal property of sheafification and so $(\ConIdeal_\cA)^\sharp \cong \LowGSB_\cA$.
\end{lem}
\begin{proof}
The commutativity of the diagram is clear from the definitions of the presheaves and sheaves. For any $T\in \Sch_S$ and any section
\[
(0 \subseteq \cI_1 \subseteq \ldots \subseteq \cI_\ell \subseteq \cA|_T) \in \LowGSB|_\cA(T),
\]
we may view it as a section in $\LowFlag_\cA(T)$ where we know that it locally belongs to the image of $\ConFlag_\cA$ by \Cref{sheafification_of_ConFlag}. However, the restrictions of each $\cI_j$ are still right ideals, and therefore the section locally belongs to the image of
\[
\ConIdeal_\cA \inj \ConFlag_\cA \inj \LowFlag_\cA.
\]
Then, by commutativity of the diagram and the fact that all morphisms are injective, the given section in $\LowGSB|_\cA(T)$ locally belongs to the image of
\[
\ConIdeal_\cA \inj \LowGSB_\cA
\]
and so this map has the universal property of sheafification, as claimed.
\end{proof}

\begin{cor}
There is a commutative diagram of presheaves
\[
\begin{tikzcd}
\ConIdeal_\cA \ar[d,hookrightarrow] \ar[r,hookrightarrow] & \RaiGSB_\cA \ar[d,hookrightarrow] \\
\ConFlag_\cA \ar[r,hookrightarrow] & \RaiFlag_\cA
\end{tikzcd}
\]
where the injection $\ConIdeal_\cA \inj \RaiGSB_\cA$ has the univerisal property of sheafification and so $(\ConIdeal_\cA)^\sharp \cong \RaiGSB_\cA$. Therefore, we also have a commutative diagram
\[
\begin{tikzcd}
\LowGSB_\cA \ar[rr,"\varphi_{\cGSB}"] \ar[dr,hookleftarrow] & & \RaiGSB_\cA \\
 & \ConIdeal_\cA \ar[ur,hookrightarrow] & 
\end{tikzcd}
\]
where $\varphi_{\cGSB}$ is a unique sheaf isomorphism.
\end{cor}
\begin{proof}
This follows by symmetry from the proof \Cref{sheafification_of_ConIdeal} by using \Cref{raised_flags} in place of \Cref{sheafification_of_ConFlag}, and then the universal property of sheafification produces the isomorphism $\varphi_{\cGSB}$.
\end{proof}

The generalized Severi-Brauer sheaf also comes with a natural action of $\GL_{1,\cA}$.
\begin{lem}\label{GL_action_on_GSB}
Let $\cA$ be an Azumaya $\cO$--algebra of constant degree $d$. The natural action of $\GL_{1,\cA}$ on $\cA$ by conjugation induces an action on $\LowGSB_\cA$. For $a \in \GL_{1,\cA}$, it behaves as
\[
a\cdot (0 \subseteq \cI_1 \subseteq \ldots \subseteq \cI_\ell \subseteq \cA) = (0 \subseteq a\cI_1 a^{-1} \subseteq \ldots \subseteq a\cI_\ell a^{-1} \subseteq \cA).
\]
Equivalently, since the $\cI_k$ are right ideals and so $a \cI_k a^{-1} = a \cI_k$, this is the left action of $a$ by multiplication.
\end{lem}
\begin{proof}
Since the conjugation action on right ideals is equivalent to the left multiplication action, we can view the action of $\GL_{1,\cA}$ on $\LowGSB_\cA$ as the restriction of the action of $\GL_\cA = \cEnd(\cA)^\times$ on $\LowFlag_\cA$ where $\GL_{1,\cA} \inj \cEnd(\cA)^\times$ as left multiplication operators. From this point of view, the fact that this action is well defined follows immediately from \Cref{GL_action_on_LowFlag}.
\end{proof}

Now, let $\cE$ be a finite locally free $\cO$--module of constant rank $d$ and consider the Azumaya $\cO$--algebra $\cEnd(\cE)$. For a submodule $\cV \subseteq \cE$ which is locally a direct summand, we define a right ideal of $\cEnd(\cE)$ by setting
\[
\cI_\cV = \{\varphi \in \cEnd(\cE) \mid \Img(\varphi)\subseteq \cV \},
\]
which is isomorphic to $\cHom_\cO(\cE,\cV)$ as an $\cO$--module. Note that here, $\Img(\varphi)$ is the sheafified image sheaf. Conversely, given a right ideal $\cI \subseteq \cEnd(\cE)$ which is locally a direct summand, we define a submodule of $\cE$ via an intersection of submodules
\[
\cV_\cI = \bigcap\{ \cW \subseteq \cE \mid \Img(\varphi)\subseteq \cW \text{ for all } \varphi \in \cI\},
\]
i.e., this is the minimal submodule of $\cE$ which contains all images of map in $\cI$. Equivalently, we could define a subpresheaf $\bigcup_{\varphi \in \cI} \Img(\varphi) \colon \Sch_S \to \Sets$ of $\cE$ by
\[
T \mapsto \left(\bigcup_{\varphi \in \cI(T)} \Img(\varphi(T))\right) \subseteq \cE(T)
\]
which contains all the presheaf images of $\varphi \in \cI$, and then
\[
\cV_\cI = \left( \bigcup_{\varphi \in \cI} \Img(\varphi) \right)^\sharp
\]
is the sheafification.

\begin{lem}\label{submodule_ideal_inverse}
Let $\cE$ be a finite locally free $\cO$--module of constant rank d. Let $\cV \subseteq \cE$ be a submodule which is locally a direct summand and let $\cI \subseteq \cEnd(\cE)$ be a right ideal which is locally a direct summand. First,
\begin{enumerate}
\item \label{submodule_ideal_inverse_i} the right ideal $\cI_\cV$ is locally a direct summand of $\cEnd(\cE)$, and
\item \label{submodule_ideal_inverse_ii} the submodule $\cV_\cI$ is locally a direct summand of $\cE$.
\end{enumerate}
Secondly, the two constructions are mutually inverse. In particular
\begin{enumerate}\setcounter{enumi}{2}
\item \label{submodule_ideal_inverse_iii} $\cV_{\cI_\cV} = \cV$, and
\item \label{submodule_ideal_inverse_iv} $\cI_{\cV_\cI} = \cI$.
\end{enumerate}
Finally, the constructions are order preserving with respect to inclusion.
\begin{enumerate}\setcounter{enumi}{4}
\item \label{submodule_ideal_inverse_v} If $\cV_1\subseteq \cV_2 \subseteq \cE$ are locally direct summands, then $\cI_{\cV_1} \subseteq \cI_{\cV_2}$, and
\item \label{submodule_ideal_inverse_vi} If $\cI_1 \subseteq \cI_2 \subseteq \cEnd(\cE)$ are locally direct summand right ideals, then $\cV_{\cI_1} \subseteq \cV_{\cI_2}$.
\end{enumerate}
\end{lem}
\begin{proof}
Since we only care if the modules are direct sums locally, we may assume we have sufficiently localized and simply consider the case when they are already direct sums.

\noindent \ref{submodule_ideal_inverse_i}: Assume $\cE = \cV \oplus \cW$ for some other submodule $\cW\subseteq \cE$. Then,
\[
\cEnd(\cE) = \cEnd(\cV) \oplus \cHom(\cV,\cW) \oplus \cHom(\cW,\cV) \oplus \cEnd(\cW)
\]
and $\cI_\cV = \cHom(\cE,\cV) = \cEnd(\cV) \oplus \cHom(\cW,\cV)$ is clearly a direct summand of the above expression.

\noindent \ref{submodule_ideal_inverse_ii}: Here, we may further localize until $\cE \cong \cO^d$ and thus $\cEnd(\cE) \cong \Mat_d(\cO)$. In this case, if over some $T\in \Sch_S$ we have a matrix
\[
\begin{bmatrix} c_1 & c_2 & \hdots & c_d \end{bmatrix} \in \cI(T)
\]
where $c_i \in \cO^d(T)$ are column vectors, then because $\cI$ is a right ideal we also have that
\[
\begin{bmatrix} c_i & 0 & \hdots & 0 \end{bmatrix} \in \cI(T)
\]
for any $i$. Thus, one sees that
\[
\cV_\cI(T) = \{ c \in \cO^d(T) \mid  \begin{bmatrix} c & 0 & \hdots & 0 \end{bmatrix} \in \cI(T) \}.
\]
Now, by assumption $\Mat_d(\cO) = \cI \oplus \cN$ for some $\cO$--module $\cN$, which is not necessarily also an ideal. So, for any $v\in \cO^d(T)$ we obtain a unique decomposition
\[
\begin{bmatrix} v & 0 & \hdots 0 \end{bmatrix} = \begin{bmatrix} c_v & c_2 & \hdots & c_d \end{bmatrix} + \begin{bmatrix} w_v & -c_2 & \hdots & -c_d \end{bmatrix}
\]
with $\begin{bmatrix} c_v & c_2 & \hdots & c_d \end{bmatrix} \in \cI(T)$ and $\begin{bmatrix} w_v & -c_2 & \hdots & -c_d \end{bmatrix} \in \cN(T)$. In particular, this means that $c_v \in \cV_\cI(T)$ and $w_v \notin \cV_\cI(T)$ if $w_v \neq 0$. Now, we can set
\[
\cW(T) = \{ w_v \in \cO^d(T) \mid v\in \cV(T)\}
\]
to obtain a submodule $\cW \subseteq \cO^d$ for which it is clear that $\cO^d = \cV_\cI \oplus \cW$. Thus, $\cV_\cI$ is a direct summand as desired.

\noindent \ref{submodule_ideal_inverse_iii} and \ref{submodule_ideal_inverse_iv}: Since all objects are sheaves, it again suffices to show that the constructions are mutually inverse sufficiently locally. Thus, in the case when $\cE = \cO^d$, the considerations in \ref{submodule_ideal_inverse_ii} above show that for $T\in \Sch_S$,
\begin{align*}
\cI_\cV(T) &= \{ \begin{bmatrix} c_1 & c_2 & \hdots & c_d \end{bmatrix} \mid c_i \in \cV(T) \} \\
\cV_\cI(T) &= \{ c\in \cO^d(T) \mid \begin{bmatrix} c & 0 & \hdots & 0 \end{bmatrix} \in \cI(T) \}.
\end{align*}
Then it is clear that $\cV_{\cI_\cV} = \cV$ and $\cI_{\cV_\cI} = \cI$.

\noindent \ref{submodule_ideal_inverse_v}: If $\cV_1 \subseteq \cV_2$, then endomorphisms of $\cE$ which have image contained in $\cV_1$ also have image contained in $\cV_2$, so $\cI_{\cV_1} \subseteq \cI_{\cV_2}$.

\noindent \ref{submodule_ideal_inverse_vi}: If $\cI_1 \subseteq \cI_2$, then
\[
\{\cW \subseteq \cE \mid \Img(\varphi) \subseteq \cW \text{ for all } \varphi \in \cI_2\} \subseteq \{\cW \subseteq \cE \mid \Img(\varphi) \subseteq \cW \text{ for all } \varphi \in \cI_1\}
\]
and so taking intersections will reverse this inclusion, thus $\cV_1 \subseteq \cV_2$.
\end{proof}

\begin{cor}\label{flags_to_SB}
We have an isomorphism of functors $\LowFlag_\cE \iso \LowGSB_{\cEnd(\cE)}$ defined by
\begin{align*}
(0 \subseteq \cV_1 \subseteq \ldots \subseteq \cV_\ell \subseteq \cE) \mapsto &(0 \subseteq \cI_{\cV_1} \subseteq \ldots \subseteq \cI_{\cV_\ell} \subseteq \cEnd(\cE)) \\
(0 \subseteq \cV_{\cI_1} \subseteq \ldots \subseteq \cV_{\cI_\ell} \subseteq \cE) \reflectbox{$\mapsto$} &(0 \subseteq \cI_1 \subseteq \ldots \subseteq \cI_\ell \subseteq \cEnd(\cE)).
\end{align*}
Furthermore, this isomorphism is equivariant with respect to the $\GL_\cE = \GL_{1,\cEnd(\cE)}$ actions of \Cref{GL_action_on_LowFlag} and \Cref{GL_action_on_GSB}.
\end{cor}
\begin{proof}
The passage from submodules to ideals is bijective and order preserving by \Cref{submodule_ideal_inverse}, thus it is clear that our proposed map is an isomorphism so long as lowered flags are sent to lowered flags of ideals. However, this is also clear from bijectivity. Consider a section
\[
(0\subseteq \cV_1 \subseteq \ldots \subseteq \cV_\ell \subseteq \cE) \in \LowFlag_\cE
\]
and assume that there exists $1\leq i \leq \ell-1$ and $T\in \Sch_S$ such that
\[
\cI_{\cV_i}|_T = \cI_{\cV_{i+1}}|_T.
\]
By bijectivity, this implies that $\cV_i|_T = \cV_{i+1}|_T$ and so $\cV_i|_T = \cE|_T = \cV_{i+1}|_T$ since the original flag is lowered. Thus,
\[
\cI_{\cV_i}|_T = \cEnd(\cE)|_T = \cI_{\cV_{i+1}}|_T
\]
which shows that the flag of right ideals is lowered as well.

To see the equivariance, one can compute that for $a\in \GL_\cE$,
\begin{align*}
\cI_{(a\cV_i)} &= \{\varphi \in \cEnd(\cE) \mid \Img(\varphi)\subseteq a\cV_i \} \\
&= \{a\varphi \in \cEnd(\cE) \mid \Img(\varphi)\subseteq \cV_i \} \\
&= a\cI_{\cV_i} \\
&= a\cI_{\cV_i}a^{-1}.
\end{align*}
Then, since the map is bijective, this also shows that $\cV_{(a\cI_ia^{-1})} = a\cV_{\cI_i}$. Hence, the isomorphism is equivariant as claimed, concluding the proof.
\end{proof}

\subsection{Isomorphisms with Parabolic Subgroups}\label{isos_with_parabolics_inner}
We first consider the case when $\bG = \GL_\cE$ for a finite locally free $\cO$--module $\cE$ of constant rank $d$. Given a section
\[
\overline{\cV} = (0 \subseteq \cV_1 \subseteq \ldots \subseteq \cV_\ell \subseteq \cO^d) \in \LowFlag_\cE,
\]
we set $\bStab(\overline{\cV}) \subseteq \bG$ to be the stabilizer of this section under the natural action of \Cref{GL_action_on_LowFlag}.
\begin{lem}\label{flag_stabilizers_are_parabolic}
The stabilizer $\bStab(\overline{\cV})$ is a parabolic subgroup of $\GL_\cE$.
\end{lem}
\begin{proof}
Since $\cPar_\bG$ is a sheaf, a group which is locally parabolic will also be parabolic globally. Thus, we may localize without loss of generality. In particular, we may assume that $\cE=\cO^d$ and that each $\cV_i$ is a free direct summand of $\cV_{i+1}$, and thus also a direct summand of $\cO^d$. Then, we can choose a basis $\{v_1,\ldots,v_d\}$ of $\cO^d$ such that
\[
\{v_1,\ldots,v_{k_i}\}
\]
is a basis of $\cV_i$ for an increasing sequence of integers $1 \leq k_1 < k_2 < \ldots < k_\ell < d$. By setting $m_1=k_1$, $m_i= k_i-k_{i-1}$ for $2\leq i \leq \ell$, and $m_{\ell+1} = d - k_\ell$, the increasing sequence of integers produces an ordered partition $m_1+m_2+\ldots +m_{\ell+1} =d$. Then, writing matrices in $\Mat_d(\cO)$ with respect to the basis $\{v_1,\ldots,v_d\}$, we see that
\[
\bStab(\overline{\cV}) = \left\{ \begin{bmatrix} B_1 & * & \hdots & * \\ 0 & B_2 & \hdots & * \\ \vdots & \vdots & \ddots & \vdots \\ 0 & 0 & \hdots & B_{\ell+1} \end{bmatrix} \in \GL_d \mid B_i \in \GL_{m_i} \right\}
\]
with zeroes below the block diagonal and arbitrary entries above the block diagonal. We recognize this as a standard parabolic subgroup of $\GL_d$, finishing the proof.
\end{proof}

\begin{cor}\label{type_of_ConFlag}
Let $\overline{\cV} \in \LowFlag_\cE$ be a flag of constant rank modules, i.e., $\overline{\cV} \in \ConFlag_\cE$. Under the type map of \eqref{type_map_parabolics} we have
\[
t(\bStab(\overline{\cV}))=(\rank(\cV_1),\rank(\cV_2),\ldots,\rank(\cV_\ell))^c.
\]
\end{cor}
\begin{proof}
We localize such that all modules in the flag become free direct summands and then we may refer to the ordered partition $m_1+m_2+\ldots +m_{\ell+1} =d$ occurring in the proof of \Cref{flag_stabilizers_are_parabolic}. Since the partial sums of this partition are the $k_i = \rank(\cV_i)$ by construction, we see that
\[
t(\bStab(\overline{\cV}))=(\rank(\cV_1),\rank(\cV_2),\ldots,\rank(\cV_\ell))^c.
\]
Since the modules in the flag are constant rank by assumption, this local description of the type also applies globally, hence the result.
\end{proof}

\begin{cor}\label{stabilizers_are_parabolic}
Let $\bG = \GL_{1,\cA}$ be the general linear group of an Azumaya $\cO$--algebra $\cA$ of constant degree $d$. Given a section
\[
\overline{\cI} = (0\subseteq \cI_1 \subseteq \ldots \subseteq \cI_\ell \subseteq \cA) \in \LowGSB_\cA,
\]
we set $\bStab(\overline{\cI}) \subseteq \bG$ to be the stabilizer of this section under the natural action of \Cref{GL_action_on_GSB}. This stabilizer is a parabolic subgroup of $\bG$.
\end{cor}
\begin{proof}
Again, we may work sufficiently locally and assume that $\cA = \Mat_d(\cO)$. Then, using the isomorphsim of \Cref{flags_to_SB} the flag of ideals comes from a flag of submodules $\overline{\cV} \in \LowFlag_{\cO^d}$. Furthermore, since this isomorphism is equivariant with respect to the actions of $\bG$, we have that $\bStab(\overline{\cI})=\bStab(\overline{\cV})$. We then know this is a parabolic subgroup by \Cref{flag_stabilizers_are_parabolic}.
\end{proof}

Based on the results of \Cref{flag_stabilizers_are_parabolic} and \Cref{stabilizers_are_parabolic}, we now have maps
\begin{align*}
\LowFlag_\cE &\to \cPar_{\GL_\cE} & \LowGSB_\cA &\to \cPar_{\GL_{1,\cA}} \\
\overline{\cV}&\mapsto \bStab(\overline{\cV}) & \overline{\cI} &\mapsto \bStab(\overline{\cI}). 
\end{align*}
which turn out to be isomorphisms.

\begin{prop}\label{flag_isomorphism_with_Par}
Let $\cE$ be a finite locally free $\cO$--module of constant rank $d$ and set $\bG = \GL_\cE = \GL_{1,\cEnd(\cE)}$. The two maps $\LowFlag_\cE \to \cPar_\bG$ and $\LowGSB_{\cEnd(\cE)} \to \cPar_\bG$, which send a flag to its stabilizer, are both isomorphisms. Furthermore, they fit into the commutative diagram of isomorphisms
\[
\begin{tikzcd}
\LowFlag_\cE \ar[rr] \ar[dr] & & \LowGSB_{\cEnd(\cE)} \ar[dl] \\
 & \cPar_\bG & 
\end{tikzcd}
\]
where the top horizontal map is the isomorphism of \Cref{flags_to_SB}.
\end{prop}
\begin{proof}
It is clear that the diagram is commutative since the isomorphism of \Cref{flags_to_SB} is $\bG$--equivariant. Thus, it is sufficient to show that $\LowFlag_\cE \to \cPar_\bG$ is an isomorphism.

Consider two sections of $\LowFlag_\cE$,
\begin{align*}
\overline{\cV} &= (0\subseteq \cV_1 \subseteq \ldots \subseteq \cV_\ell \subseteq \cE), \text{ and} \\
\overline{\cW} &= (0\subseteq \cW_1 \subseteq \ldots \subseteq \cW_k \subseteq \cE)
\end{align*}
and assume $\bStab(\cV)=\bStab(\cW)$. Since the flags will be equal if and only if they are equal locally, we may localize until all of the occurring modules are free direct summands of $\cE=\cO^d$. Then by \Cref{type_of_ConFlag}, we have that
\begin{align*}
t(\bStab(\cV)) &= (\rank(\cV_1),\rank(\cV_2),\ldots,\rank(\cV_\ell))^c, \text{ and} \\
t(\bStab(\cW)) &= (\rank(\cW_1),\rank(\cW_2),\ldots,\rank(\cW_k))^c
\end{align*}
which of course must be equal. This implies that $k=\ell$ and that $\rank(\cV_i)=\rank(\cW_i)$ for all $1\leq i \leq \ell$. Therefore, since all modules are free direct summands of $\cO^d$, we can then find some $g\in \Mat_d(\cO)$ such that $g\overline{\cV} = \overline{\cW}$. Hence,
\[
\bStab(\overline{\cV})= \bStab(\overline{\cW}) = \bStab(g\overline{\cV}) = g\bStab(\overline{\cV})g^{-1}
\]
showing that $g$ belongs to the normalizer of $\bStab(\overline{\cV})$. However, $\bStab(\overline{\cV})$ is a parabolic subgroup and by \cite[Exp. XXII, 5.8.5]{SGA3} parabolic subgroups are their own normalizer, thus $g\in \bStab(\overline{\cV})$. This means that $\overline{\cW} = g\overline{\cV} = \overline{\cV}$, so the stabilizer construction is injective.

Finally, any parabolic is locally and up to conjugation a standard parabolic \eqref{standard_parabolic} corresponding to some ordered partition $m_1+\ldots m_{\ell+1} =d$. Let $\{e_1,e_2,\ldots,e_d\}$ be the standard basis of $\cO^d$ and let $\cV_i = \Span(e_1,\ldots,e_{m_1+\ldots+m_i})$. From the description occurring in the proof of \Cref{flag_stabilizers_are_parabolic}, one sees that such a standard parabolic is the stabilizer of the flag
\[
0\subseteq \cV_1 \subseteq \ldots \subseteq \cV_\ell \subseteq \cO^d.
\]
Then, applying the inverse conjugation to this flag will produce a flag whose stabilizer is locally the initial parabolic. Thus, the map $\LowFlag_\cE \to \cPar_\bG$ is also surjective as a map of sheaves. Therefore it is an isomorphism as claimed, finishing the proof.
\end{proof}

\begin{cor}\label{SB_isomorphism_with_Par}
Let $\cA$ be an Azumaya $\cO$--algebra of constant degree $d$. The map
\begin{align*}
\LowGSB_\cA &\to \cPar_{\GL_{1,\cA}} \\
\overline{\cI} &\mapsto \bStab(\overline{\cI}). 
\end{align*}
is an isomorphism of sheaves.
\end{cor}
\begin{proof}
Locally, this map becomes the map $\LowGSB_{\cEnd(\cE)} \to \cPar_{\GL_\cE}$ of \Cref{flag_isomorphism_with_Par} which is an isomorphism. Thus, it is also an isomorphism globally.
\end{proof}

\subsection{Type Morphisms for the Inner Case}\label{type_morphisms_inner}
Let $\cE$ be a finite locally free $\cO$--module of constant rank $d=n+1$. We want to define a type morphism for lowered flags which encodes the sequence of ranks of submodules occurring in a given flag as an increasing tuple. However, the submodules participating in a flag need not be of constant rank, and so this type morphism must instead land in the locally constant sheaf $\cP_n$ of \eqref{eq_sheaf_cP_n}. Nevertheless, in the case that a flag is comprised of constant rank modules it is clear what we desire the type morphism to be and therefore it is straightforward to define the type morphism for the presheaf $\ConFlag_\cE$ of \eqref{eqn_con_flags}. Namely, we consider the map of presheaves
\begin{align}
\ConFlag_\cE &\to \cP_n \label{eq_ConFlag_type} \\
(0\subseteq \cV_1 \subseteq \cV_2 \subseteq \ldots \subseteq \cV_\ell \subseteq \cE) &\mapsto (\rank(\cV_1),\rank(\cV_2),\ldots,\rank(\cV_\ell)) \nonumber
\end{align}
with the trivial flag $0\subseteq \cE$ being mapped to the empty tuple $()$. Since $\cV_\ell$ is constant rank and not equal to $\cE$ by assumption, $\rank(\cV_\ell)\leq n$ and so this map indeed lands in $\cP_n$. Then, since $\cP_n$ is itself a sheaf, this map of presheaves will induce a unique map from the sheafification of $\ConFlag_\cE$ into $\cP_n$, which by \Cref{sheafification_of_ConFlag} is $\LowFlag_\cE$.
\begin{defn}\label{type_for_LowFlag}
Let $\cE$ be a finite locally free $\cO$--module of constant rank $d=n+1$. The \emph{type morphism} for the sheaf of lowered flags $\LowFlag_\cE$ is the sheaf morphism
\[
t_{\cFlag} \colon \LowFlag_\cE \to \cP_n
\]
induced from the presheaf morphism \eqref{eq_ConFlag_type} above. For a section $\vr \in \cP_n(S)$ we define the sheaf of \emph{flags of type $\vr$ in $\cE$} to be the fiber of $t_{\cFlag}$ over $\vr$, i.e.,
\[
\LowFlag_{\vr,\cE} = t_{\cFlag}^{-1}(\vr) \subseteq \LowFlag_\cE.
\]
\end{defn}

Let $\cA$ be an Azumaya $\cO$--algebra of constant degree $d$. It is similarly straightforward to define a type morphism from $\ConIdeal_\cA$ to $\cP_n$. However, since the rank of any constant rank right ideal in $\cA$ is a multiple of $d$, we define the map of presheaves
\begin{align}
\ConIdeal_\cA &\to \cP_n \label{eq_ConIdeal_type} \\
(0\subseteq \cI_1 \subseteq \cI_2 \subseteq \ldots \subseteq \cI_\ell \subseteq \cE) &\mapsto \left(\frac{\rank(\cI_1)}{d},\frac{\rank(\cI_2)}{d},\ldots,\frac{\rank(\cI_\ell)}{d}\right). \nonumber
\end{align}
\begin{defn}\label{defn_inner_type_GSB}
Let $\cA$ be an Azumaya $\cO$--algebra of constant degree $d$. The \emph{type morphism} for the generalized Severi-Brauer sheaf $\LowGSB_\cA$ is the sheaf morphism
\[
t_{\cGSB} \colon \LowGSB_\cA \to \cP_n
\]
induced from the presheaf morphism \eqref{eq_ConIdeal_type} above. For a section $\vr \in \cP_n(S)$ we define the \emph{Severi-Brauer sheaf of type $\vr$} to be the fiber of $t_{\cGSB}$ over $\vr$, i.e.,
\[
\LowGSB_{\vr,\cA} = t_{\cGSB}^{-1}(\vr) \subseteq \LowGSB_\cA.
\]
\end{defn}

\begin{lem}\label{type_commutativity_inner}
Let $\cE$ be a finite locally free $\cO$--module of constant rank $d$ and set $\bG = \GL_\cE = \GL_{1,\cEnd(\cE)}$. The commutative diagram of \Cref{flag_isomorphism_with_Par} extends to include type morphisms in the following way.
\[
\begin{tikzcd}
\LowFlag_\cE \ar[rr] \ar[dr] \ar[dd,swap,"t_{\cFlag}"] & & \LowGSB_{\cEnd(\cE)} \ar[dl] \ar[dd,"t_{\cGSB}"] \\[-5ex]
 & \cPar_\bG \ar[d,"t"] & \\
\cP_n \ar[r,"\und^c"] & \cP_n \ar[r,"\und^c"] & \cP_n.
\end{tikzcd}
\]
In particular, since the map $(\und^c)$ is order $2$, the diagram
\[
\begin{tikzcd}
\LowFlag_\cE \ar[rr] \ar[dr,swap,"t_{\cFlag}"] & & \LowGSB_{\cEnd(\cE)} \ar[dl,"t_{\cGSB}"] \\
 & \cP_n & 
\end{tikzcd}
\]
commutes.
\end{lem}
\begin{proof}
Given a constant rank submodule $\cV \subseteq \cE$, we know that $\cI_\cV = \cEnd(\cE,\cV)$ and that
\[
\rank(\cI_\cV) = d\cdot \rank(\cV).
\]
Thus, we have a commutative diagram of presheaves
\[
\begin{tikzcd}
\ConFlag_\cE \ar[rr] \ar[dr,swap,"\eqref{eq_ConFlag_type}"] & & \ConIdeal_{\cEnd(\cE)} \ar[dl,"\eqref{eq_ConIdeal_type}"] \\
 & \cP_n & 
\end{tikzcd}
\]
where the horizontal map sends
\[
(0\subseteq \cV_1 \subseteq \ldots \subseteq \cV_\ell \subseteq \cE) \mapsto (0\subseteq \cI_{\cV_1} \subseteq \ldots \subseteq \cI_{\cV_\ell} \subseteq \cEnd(\cE)).
\]
Thus, the induced diagram of maps of sheaves also commutes, showing that the isomorphism $\LowFlag_\cE \to \LowGSB_\cE$ respects the type morphisms and the second diagram commutes as claimed.

Next, the map
\begin{align*}
\ConFlag_\cE &\to \cPar_\bG \\
\overline{\cV} &\mapsto \bStab(\overline{\cV})
\end{align*}
fits into the diagram
\[
\begin{tikzcd}
\ConFlag_\cE \ar[r] \ar[d,"\eqref{eq_ConFlag_type}"] & \cPar_\bG \ar[d,"t"] \\
\cP_n \ar[r,"\und^c"] & \cP_n
\end{tikzcd}
\]
which commutes by \Cref{type_of_ConFlag}. Thus, the analogues diagram of sheaves with $\LowFlag_\cE$ in place of $\ConFlag_\cE$ also commutes. Finally, considering that the triangle in the first diagram consists of isomorphisms and combining the above two facts shows the overall commutativity of the first diagram as well, finishing the proof.
\end{proof}

\begin{cor}\label{inner_type_GSB_Par}
Let $\cA$ be an Azumaya $\cO$--algebra of constant degree $d$. We have a commutative diagram
\[
\begin{tikzcd}
\LowGSB_\cA \ar[r] \ar[d,"t_{\cGSB}"] & \cPar_{\GL_{1,\cA}} \ar[d,"t"] \\
\cP_n \ar[r,"\und^c"] & \cP_n
\end{tikzcd}
\]
\end{cor}
\begin{proof}
This diagram locally becomes of the form
\[
\begin{tikzcd}
\LowGSB_{\cEnd(\cE)} \ar[r] \ar[d,"t_{\cGSB}"] & \cPar_{\GL_\cE} \ar[d,"t"] \\
\cP_n \ar[r,"\und^c"] & \cP_n
\end{tikzcd}
\]
which commutes by \Cref{type_commutativity_inner}, thus is also commutes globally as claimed.
\end{proof}

\section{Hermitian Flags and Outer Severi-Brauer Sheaves}\label{outer_case}
We now consider a group $\bG$ which is an arbitrary form of $\GL_d$, possibly an outer form, and we define outer analogues of flags and Severi-Brauer sheaves which will once again be isomorphic to the sheaf of parabolic subgroups $\cPar_\bG$. Since $\bG$ is a form of $\GL_d$, we know that $\bG = \bU_{(\cB,\tau)}$ for some Azumaya algebra with involution of the second kind $(f\colon L\to S,\cB,\tau)$ of constant degree $d$.

\subsection{Hermitian Flags}
In the case that $\cB = \cEnd_{\cO|_L}(\cH)$ for a $\cO|_L$--module $\cH$ of constant rank $d$, then $\tau$ will be the adjoint involution of some hermitian regular form $h\colon f_*(\cH) \times f_*(\cH) \to \cL$. We use this hermitian form to define orthogonality. For an $\cL$--submodule $\cV \subseteq f_*(\cH)$, we set $\cV^{\perp_h} \subseteq f_*(\cH)$ to be the $\cL$--submodule which behaves on $T\in \Sch_S$ as
\[
T \mapsto \{x \in f_*(\cH)(T) \mid h(x|_{T'},v)=0 \text{ for all } T'\in \Sch_T, v\in \cV(T')\}.
\]
\begin{example}\label{split_hermitian_perp}
Recall the split hermitian form of \eqref{eq_split_hermitian_form} which is
\begin{align*}
h_d\colon (\cO^d\times\cO^d) \times (\cO^d\times\cO^d) &\to \cO_S\times\cO_S \\
((x,y),(z,w)) &\mapsto (y^t z, w^t x).
\end{align*}
For an $\cO$--submodule $\cV\subset \cO^d$, we set $\cV^\perp \subseteq \cO^d$ to be the usual perpendicular submodule, i.e., the subsheaf of sections $x\in \cO^d$ such that $x^tv=0$ for all $v\in \cV$. If $\cV$ is of constant rank $r$, then $\cV^\perp$ is of constant rank $d-r$. Furthermore, $(\cV^\perp)^\perp =\cV$.

Any $(\cO\times\cO)$--submodule of $\cO^d\times\cO^d$ is of the form $\cV_1\times\cV_2 \subseteq \cO^d\times\cO^d$ for $\cO$--submodules $\cV_1,\cV_2\subseteq \cO^d$. Its perpendicular with respect to $h_d$ is simply
\[
(\cV_1\times\cV_2)^{\perp_{h_d}} = \cV_2^\perp\times\cV_1^\perp.
\]
\end{example}

Ultimately, given a regular hermitian form $(f\colon L \to S,\cH,h)$, we want to define a sheaf of flags within $f_*(\cH)$ such that when we consider the case of the split hermitian form of \eqref{eq_split_hermitian_form} our new sheaf is isomorphic with $\LowFlag_{\cO^d}$. In this way, these new sheaves will be outer twists of $\LowFlag_{\cO^d}$. The idea will be to look at flags in $f_*(\cH)$ which are symmetric with respect to orthogonality, i.e., those such that
\[
(0\subseteq \cV_1\subseteq \ldots \subseteq \cV_\ell \subseteq f_*(\cH)) = (0\subseteq \cV_\ell^{\perp_h} \subseteq \ldots \subseteq \cV_1^{\perp_h} \subseteq f_*(\cH)),
\]
meaning that $\cV_i^{\perp_h} = \cV_{\ell+1-i}$ since the perpendicular operation is order reversing on submodules. However, there is a pitfall which we must avoid. Keeping \Cref{split_hermitian_perp} in mind, symmetric flags of $\cO\times\cO$--submodules in the split hermitian form will appear as
\[
0 \subseteq \cV_1\times\cV_\ell^\perp \subseteq \cV_2 \times \cV_{\ell-1}^\perp \subseteq \ldots \subseteq \cV_\ell \times\cV_1^\perp \subseteq \cO^d\times\cO^d
\]
and we would like to simply map this to the flag $0\subseteq \cV_1 \subseteq \cV_2 \subseteq \ldots \subseteq \cV_\ell \subseteq \cO^d$ to achieve an isomorphism with $\LowFlag_{\cO^d}$. Unfortunately, it may occur that $\cV_1=\cV_2 \neq \cO^d$, and thus this is not a valid section of $\LowFlag_{\cO^d}$, even if the starting flag is valid, meaning we have $\cV_1\times \cV_\ell^\perp \neq \cV_2\times \cV_{\ell-1}^{\perp}$. For example, consider a flag
\[
0 \subseteq \cV_1 \subseteq \cV_2 \subseteq \cV_3 \subseteq \cO^4
\] 
of constant rank submodules with $\rank(\cV_i)=i$. Then, the flag
\[
0 \subseteq \cV_1\times\cV_3^\perp \subseteq \cV_1\times\cV_2^\perp \subseteq \cV_2\times\cV_1^\perp \subseteq \cV_3\times \cV_1^\perp \subseteq \cO^4\times \cO^4
\]
has $\cO$--ranks $(2,3,5,6)$ and is a valid section of $\LowFlag_{f_*(\cH)}$ which is symmetric with respect to the perpendicular operation. However, its naive projection onto the first factor produces $(0 \subseteq \cV_1 \subseteq \cV_1 \subseteq \cV_2 \subseteq \cV_3 \subseteq \cO^4)$ which is not a valid lowered flag. To forbid such flags from occuring, we introduce the notion of the $\cL$--gap of an inclusion of submodules. It will be a locally constant integer value, so we define it via sheafification.
\begin{defn}
Let $L\to S$ be a finite \'etale cover of schemes of degree $k$ with corresponding commutative $\cO$--algebra $\cL\colon \Sch_S \to \Rings$. Let $\cM \colon \Sch_S \to \Ab$ be a finite locally free $\cL$--module of constant rank $d$. We define the sheaf of length $2$ flags of $\cL$--submodules of $\cM$ which are locally direct summands,
\begin{align*}
\cFlag_{\cL,\cM}^2 \colon \Sch_S &\to \Sets \\
T &\mapsto \left\{\cV\subseteq \cW \mid \begin{array}{l} \cV,\cW \subseteq \cM|_T \text{ are }\cL|_T\text{--submodules, and} \\ \text{they are locally direct summands}\end{array}\right\}.
\end{align*}
Note that here we allow sections where $\cV=\cW$. Over any $T\in \Sch_S$ where $\cL|_T \cong (\cO|_T)^k$, which happens locally, we will have that $\cM|_T \cong \cM_{1,T}\times\ldots\times \cM_{k,T}$ for finite locally free $\cO|_T$--modules $\cM_{i,T}$ all of constant rank $d$. Based on this, we define a presheaf
\[
\SplitFlag_{\cL,\cM}^2 \colon \Sch_S \to \Sets
\]
of length $2$ flags consisting of split constant rank $\cL$--submodules which are direct summands. It behaves on $T\in \Sch_S$ simply as
\[
T \mapsto \O
\]
if $\cL|_T \not\cong (\cO|_T)^k$, and it behaves as
\[
T \mapsto \left\{ \cV_1\times\ldots\times\cV_k \subseteq \cW_1\times\ldots\times\cW_k \mid \begin{array}{l} \cV_i,\cW_i\subseteq \cM_{i,T} \text{ are direct summands,} \\ \text{and they have constant } \cO|_T\text{--rank} \end{array}\right\}
\]
if $\cL|_T \cong (\cO|_T)^k$. The restriction maps come from submodule restriction within $\cM$ and therefore may to involve permutations of the factors of $\cO^k$ depending on the local isomorphisms $\cL|_T \cong (\cO|_T)^k$ chosen.
\end{defn}

\begin{lem}\label{Flag2_sheafification}
The canonical injection $\SplitFlag_{\cL,\cM}^2 \inj \cFlag^2_{\cL,\cM}$ has the universal property of sheafification and thus shows that
\[
(\SplitFlag_{\cL,\cM}^2)^\sharp \cong \cFlag^2_{\cL,\cM}.
\]
\end{lem}
\begin{proof}
Any $\cL$--submodule $\cV \subseteq \cM$ locally takes the form
\[
\cV_1\times\ldots\cV_k \subseteq \cM_{1,T}\times\ldots\times\cM_{k,T} = \cM|_T
\]
over any $T\in \Sch_S$ where $\cL|_T \cong (\cO|_T)^k$ since this is the form of any $(\cO|_T)^k$--submodule. If the submodule $\cV$ is locally a direct summand then each $\cV_i$ must be a locally direct summand of $\cM_{i,T}$. Because $\cM$ was assumed to be finite locally free as an $\cL$--module, each $\cM_{i,T}$ is finite locally free as an $\cO$--module and this $\cV_i$ is also finite locally free as an $\cO$--module. Thus, localizing further if necessary, we may assume that each $\cV_i$ is of constant $\cO|_T$--rank. Thus, any section $\cV \subset \cW \in \cFlag^2_{\cL,\cM}(S)$ locally belongs to the image of $\SplitFlag^2_{\cL,\cM}$. The same hold for any section of $\cFlag^2_{\cL,\cM}$, which shows that we have the universal property of sheafification as desired.
\end{proof}

\begin{defn}\label{defn_L_gap}
Let $L\to S$ be an \'etale cover of schemes of degree $k$ with corresponding commutative $\cO$--algebra $\cL\colon \Sch_S \to \Rings$. Let $\cM \colon \Sch_S \to \Ab$ be a finite locally free $\cL$--module of constant rank $d$. We define a map of presheaves
\[
\SplitFlag^2_{\cL,\cM} \to \underline{\{0,\ldots,d\}}
\]
by requiring that for each $T\in \Sch_S$ for which $\cL|_T\cong (\cO|_T)^k$, and so $\SplitFlag^2_{\cL,\cM}(T)$ is nonempty, it acts as
\[
\cV_1\times\ldots\times\cV_k \subseteq \cW_1\times\ldots\times\cW_k \mapsto \min(\{ \rank|_{\cO|_T}(\cW_i)-\rank_{\cO|_T}(\cV_i)\mid 1\leq i \leq k\})
\]
and of course it is simply the inclusion $\O \inj \underline{\{0,\ldots,d\}}$ when $\SplitFlag^2_{\cL,\cM}(T)=\O$.

Since $\underline{\{0,\ldots,d\}}$ is a sheaf, the above map induces a unique map of sheaves
\[
\gap_\cL \colon \cFlag^2_{\cL,\cM} \to \underline{\{0,\ldots,d\}}
\]
which we call the $\cL$--gap.
\end{defn}
As a first comment, we point out that the presheaf map defined above is actually well defined (despite possible permutations appearing in the restriction maps of $\SplitFlag^2_{\cL,\cM}$) since the image is invariant under permutations of the factors. Intuitively, the $\cL$--gap aims to capture the maximal $n$ such that we have a commutative diagram of injections of $\cL$--modules
\[
\begin{tikzcd}
\cV\oplus \cL^n \ar[dr,hookrightarrow,"\exists"] & \\
\cV \ar[r,hookrightarrow] \ar[u,hookrightarrow] & \cW.
\end{tikzcd}
\]
Stated naively like this we do not get a notion which is stable under base change, so the $\cL$--gap locally captures this idea and then combines the result into a locally constant section. In the problematic example given above,
\[
0 \subseteq \cV_1\times\cV_3^\perp \subseteq \cV_1\times\cV_2^\perp \subseteq \cV_2\times\cV_1^\perp \subseteq \cV_3\times \cV_1^\perp \subseteq \cO^4\times \cO^4
\]
it is clear that
\[
\gap_\cL(\cV_1\times\cV_3^\perp \subseteq \cV_1\times\cV_2^\perp) = 0
\]
and so we will avoid such flags by disallowing subsequent components from having an $\cL$--gap of zero.

Using the $\cL$--gap, we define the sheaf of \emph{$\cL$--lowered} flags.
\begin{defn}\label{defn_Flags_LH}
We define the presheaf $\ConFlag_{\cL,\cH} \colon \Sch_S \to \Sets$ to be the subpresheaf of $\ConFlag_{f_*(\cH)} \colon \Sch_S \to \Sets$ of those flags of  constant $\cO$--rank submodules which are locally direct sumamnds
\[
0 \subseteq \cV_1 \subseteq \ldots \subseteq \cV_\ell \subseteq f_*(\cH)
\]
where each $\cV_i\subseteq f_*(\cH)$ is also an $\cL$--submodule and where $\gap_\cL(\cV_i\subseteq \cV_{i+1})>0$, i.e., sufficiently locally whenever the $\cL$--gap is constant it must be non-zero.

We say that a flag
\[
(0 \subseteq \cV_1 \subseteq \ldots \subseteq \cV_\ell \subseteq f_*(\cH)) \in \LowFlag_{f_*(\cH)}
\]
is $\cL$--lowered if it consists of $\cL$--submodules of $f_*(\cH)$ which are locally direct summands and if for every $T\in \Sch_S$ we have
\[
\gap_{\cL|_T}(\cV_i|_T\subseteq \cV_{i+1}|_T)=0 \Rightarrow \cV_i|_T = f_*(\cH)|_T = \cV_{i+1}|_T.
\]

We define the sheaf $\LowFlag_{\cL,\cH} \colon \Sch_S \to \Sets$ to be the subsheaf of $\LowFlag_{f_*(\cH)}$ consisting of $\cL$--lowered flags. Symmetrically, we also have the sheaf $\RaiFlag_{\cL,\cH}$ of $\cL$--raised flags. 
\end{defn}

\begin{lem}\label{sheafification_of_ConFlag_LH}
The sheaf $\LowFlag_{\cL,\cH}$ defined above is a sheaf. Furthermore, the canonical inclusion $\ConFlag_{\cL,\cH} \inj \LowFlag_{\cL,\cH}$ satisfies the universal property of sheafification, showing that
\[
(\ConFlag_{\cL,\cH})^\sharp \cong \LowFlag_{\cL,\cH}.
\]
Similarly, the inclusion $\ConFlag_{\cL,\cH} \inj \RaiFlag_{\cL,\cH}$ shows that $(\ConFlag_{\cL,\cH})^\sharp \cong \RaiFlag_{\cL,\cH}$.
\end{lem}
\begin{proof}
Since $\cL$ itself is a sheaf, the property of being an $\cL$--submodule is stable under base change and gluing. Being $\cL$--lowered is stable under base change by definition and therefore we only need to verify that being $\cL$--lowered respects gluing.

Without loss of generality, we may work with a section $(0\subseteq \cV_1 \subseteq \ldots \subseteq \cV_\ell \subseteq f_*(\cH)) \in \LowFlag_{f_*(\cH)}(S)$ over $S$ consisting of $\cL$--submodules. Assume we have a cover $\{U_i \to S\}_{i\in I}$ over which each of the restrictions
\[
(0 \subseteq \cV_1|_{U_i} \subseteq \ldots \subseteq \cV_{k_i}|_{U_i} \subseteq f_*(\cH)|_{U_i})
\]
is $\cL|_{U_i}$--lowered. Let $T\in \Sch_S$ and assume that for some index $j$ we have
\[
\gap_{\cL|_T}(\cV_j|_T \subseteq \cV_{j+1}|_T)=0.
\]
Restricting our cover to $T$ produces a cover $\{T_i=T\times_S U_i \to T\}_{i\in I}$, and for each $i\in I$ we have
\[
\gap_{\cL|_{T_i}}(\cV_j|_{T_i} \subseteq \cV_{j+1}|_{T_i})=0.
\]
Hence, either by truncation within $\LowFlag_{f_*(\cH)}$ if $j\geq k_i$ or by the assumption that the restriction to $U_i$ is $\cL|_{U_i}$--lowered, we know that
\[
\cV_j|_{T_i} = f_*(\cH)|_{T_i} = \cV_{j+1}|_{T_i}.
\]
Since this holds over the cover of $T$, we conclude that $\cV_j|_T = f_*(\cH)|_T = \cV_{j+1}|_T$ which shows that our original flag is $\cL$--lowered. Thus, we have shown that $\LowFlag_{\cL,\cH}$ is a sheaf.

Any flag in $\LowFlag_{\cL,\cH}$ of course also belongs to $\LowFlag_{f_*(\cH)}$, and therefore locally lies in the image of $\ConFlag_{f_*(\cH)}$ by \Cref{sheafification_of_ConFlag}. The property that the components are $\cL$--submodules which are locally direct summands is preserved by restriction, and so the local sections in $\ConFlag_{f_*(\cH)}$ will also have this property. Let
\[
0 \subseteq \cV_1 \subseteq \ldots \subseteq \cV_\ell \subseteq f_*(\cH)
\]
be a section of $\ConFlag_{f_*(\cH)}$ which is the restriction of a section from $\LowFlag_{\cL,\cH}$, in particular assume that it is $\cL$--lowered. Since the $\cO$--ranks of $\cV_i$ are constant and we have that $\rank_\cO(\cV_i)<\rank_\cO(\cV_{i+1})$, we know that for all $T\in \Sch_S$ we have
\[
\cV_i|_T \neq \cV_{i+1}|_T.
\]
Hence, by contrapositive, $\gap_{\cL|_T}(\cV_i|_T\subseteq \cV_{i+1}|_T)\neq 0$ which globally means that
\[
\gap_{\cL|_T}(\cV_i|_T\subseteq \cV_{i+1}|_T) > 0.
\]
Therefore, we have shown that sections of $\LowFlag_{\cL,\cH}$ locally belong to $\ConFlag_{\cL,\cH}$, and so we are done.
\end{proof}

The sheaf of hermitian flags we are working towards will be the fixed points in $\LowFlag_{\cL,\cH}$ under an order two automorphism related to the perpendicular operation. The perpendicular operation is order reversing on submodule inclusion, therefore if we are given a flag
\[
\overline{\cV} = (0 \subseteq \cV_1 \subseteq \cV_2 \subseteq \ldots \subseteq \cV_{\ell-1} \subseteq \cV_\ell \subseteq f_*(\cH)) \in \ConFlag_{\cL,\cH}
\]
we can apply the perpendicular operation on the flag as a whole and produce a second flag
\[
\overline{\cV}^{\perp_h} = (0 \subseteq \cV_\ell^{\perp_h} \subseteq \cV_{\ell-1}^{\perp_h} \subseteq \ldots \subseteq \cV_2^{\perp_h} \subseteq \cV_1^{\perp_h} \subseteq f_*(\cH)).
\]
We verify now that the resulting flag remains in $\ConFlag_{\cL,\cH}$ with an easy lemma.
\begin{lem}\label{hermitian_orthogonal_rank}
Let $(\cH,h)$ be a regular hermitian form of constant rank $d$ and let $\cV\subseteq f_*(\cH)$ be an $\cL$--submodule which is locally a direct summand and which is finite locally free as an $\cO$--module. Then,
\begin{enumerate}
\item \label{hermitian_orthogonal_rank_i} $\cV^{\perp_h}$ is also locally a direct summand,
\item \label{hermitian_orthogonal_rank_ii} $\rank_\cO(\cV^{\perp_h}) = 2d-\rank_\cO(\cV)$ as a locally constant function, and
\item \label{hermitian_orthogonal_rank_iii} $(\cV^{\perp_h})^{\perp_h} = \cV$.
\end{enumerate}
Additionally, if $\cV\subseteq \cW$ are two such $\cL$--submodules, then
\begin{enumerate}
\setcounter{enumi}{3}
\item \label{hermitian_orthogonal_rank_iv} $\gap_\cL(\cV\subseteq \cW) = \gap_\cL(\cW^{\perp_h} \subseteq \cV^{\perp_h})$.
\end{enumerate}
\end{lem}
\begin{proof}
\noindent \ref{hermitian_orthogonal_rank_i}: Since $h\colon f_*(\cH) \times f_*(\cH) \to \cL$ is regular, it induces an isomorphism of $\cO$--modules
\begin{align*}
\tilde{h} \colon f_*(\cH) &\iso \cHom_\cL(f_*(\cH),\cL) \\
v &\mapsto h(v,\und)
\end{align*}
which is $i^*$--semilinear as a map of $\cL$--modules. Now, we assume we are working sufficiently locally such that $f_*(\cH) = \cV \oplus \cW$ for another $\cL$--submodule $\cW$. In turn, this means
\[
\cHom_\cL(f_*(\cH),\cL) = \tilde{h}(\cV)\oplus \tilde{h}(\cW)
\]
where both $\tilde{h}(\cV)$ and $\tilde{h}(\cW)$ are still $\cL$--submodules. Then, since $f_*(\cH)$ is a finite locally free $\cL$--modules of rank $d$, we have the canonical double dual isomorphism
\[
f_*(\cH) \cong \cHom_{\cL}(\cHom_{\cL}(f_*(\cH),\cL),\cL)
\]
sending $v\mapsto \textrm{ev}_v$. Thus also
\[
f_*(\cH) \cong \cHom_{\cL}(\tilde{h}(\cV),\cL) \oplus \cHom_{\cL}(\tilde{h}(\cW),\cL)
\]
and tracing the various isomorphisms, one sees that $\cV^{\perp_h} = \cHom_{\cL}(\tilde{h}(\cW),\cL)$ and so it is a direct summand of $f_*(\cH)$. 

\noindent \ref{hermitian_orthogonal_rank_ii}: Since $\cV \subseteq f_*(\cH)$ is an $\cL$--submodule, after sufficient localization we can assume we are in the setting of \Cref{split_hermitian_perp} where we have a submodule of the form
\[
\cV_1\times \cV_2 \subseteq \cO^d\times\cO^d.
\]
Additionally, since we started with a module which is finite locally free as a $\cO$--module, the modules $\cV_1,\cV_2\subseteq \cO^d$ can be assumed to be finite locally free of constant rank. Then, $\rank_\cO(\cV_1\times\cV_2) = \rank_\cO(\cV_1)+\rank_\cO(\cV_2)$ and
\begin{align*}
\rank_\cO\big((\cV_1\times\cV_2)^{\perp_{h_d}}\big) &= \rank_\cO(\cV_2^\perp\times\cV_1^\perp) \\
&= \rank_\cO(\cV_2^\perp)+ \rank_\cO(\cV_1^\perp) \\
&= d-\rank_\cO(\cV_2) + d-\rank_\cO(\cV_1) \\
&= 2d - \rank_\cO(\cV_1\times\cV_2).
\end{align*}
Thus, globally we have $\rank_\cO(\cV^{\perp_h})=2d-\rank_\cO(\cV)$ as claimed.

\noindent \ref{hermitian_orthogonal_rank_iii}: From the definition one sees that $\cV \subseteq (\cV^{\perp_h})^{\perp_h}$, and then we must have equality since they are both finite locally free modules of the same rank by applying \ref{hermitian_orthogonal_rank_ii} twice. Alternatively, one could once again work sufficiently locally to be in the setting of \Cref{split_hermitian_perp} where it is obvious to calculate that
\[
\big((\cV_1\times\cV_2)^\perp\big)^\perp = (\cV_2^\perp\times\cV_1^\perp)^\perp = \cV_1\times\cV_2.
\]

\noindent \ref{hermitian_orthogonal_rank_iv}: We again work sufficiently locally where $\cV = \cV_1\times\cV_2$ and $\cW=\cW_1\times\cW_2$ for constant rank $\cO$--modules $\cV_i$ and $\cW_i$. Then, abbreviating $\rank_\cO(\cV_i)=r(\cV_i)$, we simply compute that
\begin{align*}
& \gap_\cL(\cW_2^\perp\times\cW_1^\perp \subseteq \cV_2^\perp\times \cV_1^\perp) \\
=& \min(\{d-r(\cV_2)-(d-r(\cW_2)), d-r(\cV_1)-(d-r(\cW_1))\}) \\
=& \min(\{r(\cW_2)-r(\cV_2),r(\cW_1)-r(\cV_1)\}) \\
=& \gap_\cL(\cV_1\times\cV_2\subseteq \cW_1\times\cW_2).
\end{align*}
Hence, since the equality holds locally we also have
\[
\gap_\cL(\cV\subseteq \cW) = \gap_\cL(\cW^{\perp_h} \subseteq \cV^{\perp_h})
\]
globally as claimed.
\end{proof}

\Cref{hermitian_orthogonal_rank} shows that we obtain an order two automorphism
\begin{align}
\ConFlag_{\cL,\cH} &\iso \ConFlag_{\cL,\cH} \label{eq_ConFlag_perp} \\
\overline{\cV} &\mapsto \overline{\cV}^{\perp_h}. \nonumber
\end{align}

However, since the perpendicular operation is order reversing it sends lowered flags to raised flags and vice versa
\begin{cor}\label{perp_lowered_to_raised}
Let $(f\colon L \to S,\cH,h)$ be a regular hermitian form. There are mutually inverse sheaf isomorphism
\begin{align*}
\LowFlag_{\cL,\cH} &\iso \RaiFlag_{\cL,\cH} & \RaiFlag_{\cL,\cH} &\iso \LowFlag_{\cL,\cH} \\
\overline{\cV}&\mapsto \overline{\cV}^{\perp_h} & \overline{\cW}&\mapsto \overline{\cW}^{\perp_h}
\end{align*}
which fit into the commutative diagram
\[
\begin{tikzcd}
\LowFlag_{\cL,\cH} \ar[r,"\sim"] & \RaiFlag_{\cL,\cH} \ar[r,"\sim"] & \LowFlag_{\cL,\cH} \\
\ConFlag_{\cL,\cH} \ar[r,"\eqref{eq_ConFlag_perp}"] \ar[u] & \ConFlag_{\cL,\cH} \ar[r,"\eqref{eq_ConFlag_perp}"] \ar[u] & \ConFlag_{\cL,\cH}. \ar[u]
\end{tikzcd}
\]
\end{cor}
\begin{proof}
It is immediate from \Cref{hermitian_orthogonal_rank}\ref{hermitian_orthogonal_rank_iii} that if these maps are well defined, then they are mutually inverse sheaf isomorphisms and fit into the proposed commutative diagrams. So, consider an $\cL$--lowered flag $0\subseteq \cV_1 \subseteq \ldots \cV_\ell \subseteq f_*(\cH)$ in $\LowFlag_{\cL,\cH}$. If for some $T \in \Sch_S$ we have
\[
\gap_{\cL|_T}(\cV_{i+1}^{\perp_h}|_T \subseteq \cV_i^{\perp_h}|_T)=0,
\]
for some index $i$, then this means by \Cref{hermitian_orthogonal_rank}\ref{hermitian_orthogonal_rank_iv} that $\gap_{\cL|_T}(\cV_i|_T \subseteq \cV_{i+1}|_T)=0$, so since we started with an $\cL$--lowered flag we have $\cV_i|_T = f_*(\cH)|_T = \cV_{i+1}|_T$. Then applying the perpendicular operator again, we see
\[
\cV_{i+1}^{\perp_h}|_T = 0 = \cV_i^{\perp_h}|_T
\]
so the resulting flag is $\cL$--raised. Symmetrically, the perpendicular of a raised flag is a lowered flag. Thus, the proof is finished.
\end{proof}

This shows that naively applying the perpendicular operation does not quite give us what we want. Instead, we use the universal property of the sheafification of $\ConFlag_{\cL,\cH}$ to define an order two isomorphism of $\LowFlag_{\cL,\cH}$.
\begin{lem}\label{pi_h_automorphism}
There is a unique order two automorphism $\pi_h \colon \LowFlag_{\cL,\cH} \iso \LowFlag_{\cL,\cH}$ which restricts to \eqref{eq_ConFlag_perp} on $\ConFlag_{\cL,\cH}$. It is equivalently defined by the commutativity of the diagram
\[
\begin{tikzcd}
\LowFlag_{\cL,\cH} \ar[r,"\pi_h"] & \LowFlag_{\cL,\cH} \\
\ConFlag_{\cL,\cH} \ar[r,"\eqref{eq_ConFlag_perp}"] \ar[u] & \ConFlag_{\cL,\cH} \ar[u]
\end{tikzcd}
\]
or as the composition
\[
\pi_h \colon \LowFlag_{\cL,\cH} \iso \RaiFlag_{\cL,\cH} \xrightarrow{\varphi} \LowFlag_{\cL,\cH}
\]
where the first isomorphism is from \Cref{perp_lowered_to_raised} and $\varphi$ is uniquely induced by both $\RaiFlag_{\cL,\cH}$ and $\LowFlag_{\cL,\cH}$ being sheafifications of $\ConFlag_{\cL,\cH}$ (it is analogous to $\varphi_{\cFlag}$ as in \Cref{raised_flags}).
\end{lem}
\begin{proof}
This is immediate since $\LowFlag_{\cL,\cH} = (\ConFlag_{\cL,\cH})^\sharp$ by \Cref{sheafification_of_ConFlag_LH}.
\end{proof}

\begin{example}\label{pi_h_behaviour_example}
By combining the description of $\varphi_{\cFlag}$ appearing in \Cref{raised_flags} with the perpendicular operation, we can give an example of how the map $\pi_h$ behaves. Say we have a section of $\LowFlag_{\cL,\cH}(S)$ which restricts to flags of constant length over a decomposition $S=X\sqcup Y \sqcup W$. Writing the flag vertically as in \Cref{raised_flags}, $\pi_h$ will appear as
\[
\begin{array}{ccc}
(f_*(\cH)|_X & f_*(\cH)|_Y & f_*(\cH)|_W) \\
(f_*(\cH)|_X & f_*(\cH)|_Y & \cV_\ell|_W) \\
 & \vdots & \\
(f_*(\cH)|_X & \cV_{j_2}|_Y & \cV_{j_2}|_W) \\
 & \vdots & \\
(\cV_{j_1}|_X & \cV_{j_1}|_Y & \cV_{j_1}|_W) \\
 & \vdots & \\
(\cV_1|_X & \cV_1|_Y & \cV_1|_W)\\ 
(0 & 0 & 0)
\end{array} \mapsto
\begin{array}{ccc}
(f_*(\cH)|_X & f_*(\cH)|_Y & f_*(\cH)|_W) \\
(f_*(\cH)|_X & f_*(\cH)|_Y & \cV_1^{\perp_h}|_W) \\
 & \vdots & \\
(f_*(\cH)|_X & \cV_1^{\perp_h}|_Y & \cV_{\ell+1-j_2}^{\perp_h}|_W) \\
 & \vdots & \\
(\cV_1^{\perp_h}|_X & \cV_{j_2+1-j_1}^{\perp_h}|_Y & \cV_{\ell+1-j_1}^{\perp_h}|_W) \\
 & \vdots & \\
(\cV_{j_1}^{\perp_h}|_X & \cV_{j_2}^{\perp_h}|_Y & \cV_\ell^{\perp_h}|_W)\\ 
(0 & 0 & 0),
\end{array}
\]
simply applying the perpendicular operation to the non-zero proper submodule portion within each column.
\end{example}

The sheaf we actually want to work with is the subsheaf of fixed points of $\pi_h$.
\begin{defn}
Let $(f\colon L \to S,\cH,h)$ be a regular hermitian form of constant rank $d$. The presheaf of \emph{constant rank hermitian flags} of $\cH$ is $\ConFlag_{(\cH,h)} \colon \Sch_S \to \Sets$ which sends $T\in \Sch_S$ to
\[
\left\{(0\subseteq \cV_1 \subseteq \ldots \subseteq \cV_\ell \subseteq f_*(\cH)) \in \ConFlag_{\cL,\cH}(T) \mid \overline{\cV}^{\perp_h} = \overline{\cV} \right\}.
\]

The sheaf of \emph{hermitian lowered flags} of $\cH$ is
\begin{align*}
\LowFlag_{(\cH,h)} \colon \Sch_S &\to \Sets \\
T &\mapsto \{\overline{\cV} \in \LowFlag_{\cL,\cH}(T) \mid \pi_h(\overline{\cV}) = \overline{\cV} \}.
\end{align*}
\end{defn}

\begin{lem}\label{sheafification_of_hermitian_ConFlag}
The canonical inclusion $\ConFlag_{(\cH,h)} \inj \LowFlag_{(\cH,h)}$ satisfies the universal property of sheafification, showing that
\[
(\ConFlag_{(\cH,h)})^\sharp = \LowFlag_{(\cH,h)}.
\]
\end{lem}
\begin{proof}
This follows immediately from \Cref{sheafification_of_ConFlag_LH} since $\pi_h$ restricts to the perpendicular operation on $\ConFlag_{\cL,\cH}$ by definition in \Cref{pi_h_automorphism}.
\end{proof}

\begin{example}\label{example_split_ConFlag_Hh}
Here we show that our $\cL$--gap based definitions have solved the problems discussed after \Cref{split_hermitian_perp}. Recall from \Cref{split_hermitian_perp} that for the split hermitian form $(f\colon L \to S,\cH,h_d)= (f\colon S\sqcup S \to S,(\cO^d,\cO^d),h_d)$ of \eqref{eq_split_hermitian_form} we have that
\[
(\cV_1\times\cV_2)^{\perp_{h_d}} = \cV_2^\perp \times \cV_1^\perp
\]
for $\cV_1,\cV_2 \subseteq \cO^d$ submodules. Therefore, for $T\in \Sch_S$, a section of $\ConFlag_{(\cH,h_d)}(T)$ is of the form
\[
0 \subseteq \cV_1\times\cV_\ell^\perp \subseteq \cV_2\times \cV_{\ell-1}^\perp \subseteq \ldots \subseteq \cV_{\ell-1}\times\cV_2^\perp \subseteq \cV_\ell\times \cV_1^\perp \subseteq \cO^d|_T\times\cO^d|_T
\]
where each $\cV_j\times \cV_{\ell+1-j}^\perp$ is of constant $\cO|_T$--rank and the $\cL|_T$--gaps between subsequent components are strictly larger than zero. The submodules $\cV_j$ themselves may not be of constant $\cO|_T$--rank, but we may restrict further until they are. Then, it is immediate that the assumption on $\cL$--gaps implies that $\rank_\cO(\cV_i) < \rank_\cO(\cV_{i+1})$ for all $i$, and so projecting onto the first factor produces a section
\[
(0\subseteq \cV_1\subseteq \ldots \subseteq \cV_\ell \subseteq \cO^d|_T) \in \ConFlag_{\cO^d}(T).
\]
However, some localization was required to make the ranks of the first factors constant. 

In the converse direction, it is clear that if we start with a section $(0\subseteq \cV_1\subseteq \ldots \subseteq \cV_\ell \subseteq \cO^d|_T) \in \ConFlag_{\cO^d}(T)$ then
\[
(0\subseteq \cV_1\times\cV_\ell^\perp \subseteq \ldots \subseteq \cV_\ell\times\cV_1^\perp \subseteq \cO^d\times\cO^d) \in \ConFlag_{(\cH,h_d)}(T)
\]
is a valid symmetric flag with strictly positive $\cL|_T$--gaps. Hence we have an injective map of presheaves
\begin{align*}
\ConFlag_{\cO^d} &\inj \ConFlag_{((\cO^d,\cO^d),h_d)} \\
(0\subseteq \cV_1\subseteq \ldots \subseteq \cV_\ell \subseteq \cO^d) &\mapsto (0\subseteq \cV_1\times\cV_\ell^\perp \subseteq \ldots \subseteq \cV_\ell\times\cV_1^\perp \subseteq \cO^d\times\cO^d)
\end{align*}
which is not presheaf surjective, but is locally surjective. It therefore induces an isomorphism of the respective sheafifications
\begin{equation}\label{eq_ConFlag_iso}
\LowFlag_{\cO^d} \iso \LowFlag_{((\cO^d,\cO^d),h_d)}.
\end{equation}
In fact, this isomorphism can be written a bit more explicitly. For a section in $\LowFlag_{\cO^d}$, denote
\[
\varphi_{\cFlag}(0\subseteq \cV_1 \subseteq \ldots \subseteq \cV_\ell \subseteq \cO^d) = (0\subseteq \cW_1\subseteq \ldots \subseteq \cW_\ell \subseteq \cO^d)
\]
to be its image in $\RaiFlag_{\cO^d}$. Then, the isomorphism \eqref{eq_ConFlag_iso} acts as
\[
(0\subseteq \cV_1 \subseteq \ldots \subseteq \cV_\ell \subseteq \cO^d) \mapsto (0\subseteq \cV_1\times\cW_\ell^\perp \subseteq \ldots \subseteq \cV_\ell\times\cW_1^\perp \subseteq \cO^d\times\cO^d).
\]
However, checking that this is indeed how the isomorphism appears is a bit cumbersome, so we omit those details.
\end{example}

The action of $\GL_d\rtimes\ZZ/2\ZZ$ on $\cH_d$ induces an action on $\ConFlag_{((\cO^d,\cO^d),h_d)}$ in which $a\in \GL_d$ acts by
\begin{align}
&a\cdot (0\subseteq \cV_1\times\cV_\ell^\perp \subseteq \ldots \subseteq \cV_\ell\times\cV_1^\perp \subseteq \cO^d\times\cO^d) \label{GL_d_action_on_ConFlag_Hd} \\
= &(0\subseteq a\cV_1\times(a^{-1})^t\cV_\ell^\perp \subseteq \ldots \subseteq a\cV_\ell\times(a^{-1})^t\cV_1^\perp \subseteq \cO^d\times\cO^d) \nonumber
\end{align}
and $\overline{1}\in \ZZ/2\ZZ$ acts by
\begin{align*}
&\overline{1}\cdot (0\subseteq \cV_1\times\cV_\ell^\perp \subseteq \ldots \subseteq \cV_\ell\times\cV_1^\perp \subseteq \cO^d\times\cO^d)\\
= &(0\subseteq \cV_\ell^\perp\times\cV_1 \subseteq \ldots \subseteq \cV_1^\perp\times\cV_\ell \subseteq \cO^d\times\cO^d).
\end{align*}
By \Cref{sheafification_of_hermitian_ConFlag}, this extends uniquely to an action on $\LowFlag_{((\cO^d,\cO^d),h_d)}$

\begin{lem}\label{hermitian_flag_twist}
Consider a regular hermitian form $(L\to S,\cH,h)$ of constant rank $d$ and denote its corresponding $(\GL_d\rtimes\ZZ/2\ZZ)$--torsor by $\cK$. Then,
\[
\LowFlag_{(\cH,h)} \cong \cK \wedge^{\GL_d\rtimes\ZZ/2\ZZ} \LowFlag_{((\cO^d,\cO^d),h_d)}.
\]
\end{lem}
\begin{proof}
The hermitian forms $(L\to S,\cH,h)$ and $(S\sqcup S \to S,(\cO^d,\cO^d),h_d)$ are locally isomorphic and any isomorphism between them extends to an isomorphism between $\LowFlag_{(\cH,h)}$ and $\LowFlag_{((\cO^d,\cO^d),h_d)}$. Thus, $\LowFlag_{(\cH,h)}$ will be the twist of $\LowFlag_{((\cO^d,\cO^d),h_d)}$ by the $\bAut(S\sqcup S \to S,(\cO^d,\cO^d),h_d) = \GL_d\rtimes \ZZ/2\ZZ$--torsor corresponding to $(L\to S,\cH,h)$, i.e., exactly
\[
\LowFlag_{(\cH,h)} \cong \cK \wedge^{\GL_d\rtimes\ZZ/2\ZZ} \LowFlag_{((\cO^d,\cO^d),h_d)}
\]
as claimed.
\end{proof}

\begin{lem}\label{GL_action_on_LowFlag_Hh}
Let $(f\colon L\to S,\cH,h)$ be a regular hermitian form of constant rank $d$. The natural action of $\GL_{(\cH,h)}$ on $f_*(\cH)$ induces an action on $\LowFlag_{(\cH,h)}$. For $B\in \GL_{(\cH,h)}$, it behaves as
\[
B\cdot (0 \subseteq \cV_1 \subseteq \ldots \subseteq \cV_\ell \subseteq f_*(\cH)) = (0 \subseteq B\cV_1 \subseteq \ldots \subseteq B\cV_\ell \subseteq f_*(\cH)).
\]
\end{lem}
\begin{proof}
We argue that the action of $\GL_{(\cH,h)}$ on $f_*(\cH)$ induces an action on $\ConFlag_{(\cH,h)}$, which will then induce the action claimed above via sheafification. To see that we obtain a well-defined action on $\ConFlag_{(\cH,h)}$, we need to check that the action of $\GL_{(\cH,h)}$ preserves the properties of a flag being symmetric and having strictly positive $\cL$--gaps. 

For symmetry, $\GL_{(\cH,h)}$ is defined by the property that for $B\in \GL_{(\cH,h)}$ we have $h(Bx,By)=h(x,y)$, and therefore
\[
(B\cV)^{\perp_h} = B(\cV^{\perp_h}).
\]
Thus, for a section $(0 \subseteq \cV_1 \subseteq \ldots \subseteq \cV_\ell \subseteq f_*(\cH)) \in \ConFlag_{(\cH,h)}$ we will have that
\[
(B\cV_j)^{\perp_h} = B(\cV_j^{\perp_h}) = B\cV_{\ell+1-j},
\]
and so the resulting flag is also symmetric.

Regarding $\cL$--gaps, for any two subsequent components in a section of $\ConFlag_{(\cH,h)}$, they are locally of the form
\[
\cV_1\times\cV_2 \subseteq \cW_1\times\cW_2 \subseteq \cO^d\times \cO^d
\]
and likewise a section $B\in \GL_{(\cH,h)}$ locally assumes the form $(B_1,(B_1^{-1})^t)$ for some $B_1 \in \GL_d$. Then, applying $B_1$ preserves the $\cO$--ranks of $\cV_1$ and $\cW_1$ while applying $(B_1^{-1})^t$ preserves the ranks of $\cV_2$ and $\cW_2$. Therefore, it is clear that
\begin{align*}
& \gap_\cL((B_1,(B_1^{-1})^t)(\cV_1\times\cV_2)\subseteq (B_1,(B_1^{-1})^t)(\cW_1\times\cW_2)) \\
=& \gap_\cL(\cV_1\times\cV_2\subseteq \cW_1\times\cW_2).
\end{align*}
Since the action of $\GL_{(\cH,h)}$ locally fixes the $\cL$--gap, it fixes it globally as well. It therefore also preserves the property of $\cL$--gap being strictly positive, and so we are done.
\end{proof}

\subsection{Outer Severi-Brauer Sheaves}
We now return to considering an Azumaya algebra $(f\colon L \to S,\cB,\tau)$ of constant degree $d$ with involution of the second kind. Similarly to how we considered $\cL = f_*(\cO|_L)$--submodules which are of constant $\cO$--rank in the previous section, here will we be concerned with $\cL$--ideals of $f_*(\cB)$ which are locally direct summands and which are finite locally free as $\cO$--modules. Given a right ideal $\cI\subseteq f_*(\cB)$, its left annihilator
\begin{align*}
\lann\cI \colon \Sch_S &\to \Ab \\
T &\mapsto \{b\in f_*(\cB)(T) \mid b\cI|_T =0 \}
\end{align*}
is a left ideal of $f_*(\cB)$. Using the involution $\tau$ we can obtain another right ideal $\tau(\lann\cI) \subseteq \cB$. This operation is also order reversing and will therefore send lowered flags to raised flags. Thus, we work with presheaves of constant rank objects here as well after a preparatory lemma.

\begin{lem}\label{left_ann_facts}
Let $(f\colon L \to S,\cB,\tau)$ be an Azumaya algebra with involution of the second kind of degree rank $d$. Thus, $f_*(\cB)$ is of constant rank $2d^2$. Let $\cI\subseteq f_*(\cB)$ be a right $\cL$--ideal which is locally a direct summand and which is finite locally free as a $\cO$--module. Then,
\begin{enumerate}
\item \label{left_ann_facts_i} $\tau(\lann\cI)$ is a right $\cL$--ideal which is also locally a direct summand,
\item \label{left_ann_facts_v} $\tau\big(\lann(\tau(\lann\cI))\big) = \cI$, i.e., this is an order two operation,
\item \label{left_ann_facts_ii} $\tau(\lann\cI)$ is also finite locally free as an $\cO$--module,
\item \label{left_ann_facts_iii} $\rank_\cO(\cI)$ as a locally constant function is always divisible by $d$, and
\item \label{left_ann_facts_iv}$\rank_\cO(\tau(\lann\cI)) = 2d^2-\rank_\cO(\cI)$ which is then also divisible by $d$.
\end{enumerate}
Additionally, for two such right $\cL$--ideals $\cI_1\subseteq \cJ$,
\begin{enumerate}
\setcounter{enumi}{5}
\item \label{left_ann_facts_vi} $\gap_\cL(\cI\subseteq \cJ)=\gap_\cL(\tau(\lann\cJ)\subseteq \tau(\lann\cI))$.
\end{enumerate}
\end{lem}
\begin{proof}
Since all the claims are local in nature, we may assume that we are working with $(f\colon S\sqcup S \to S, (\Mat_d(\cO),\Mat_d(\cO),\tau_d)$ as in \eqref{split_second_involution} and that $\cI\subseteq f_*(\Mat_d(\cO),\Mat_d(\cO)) = \Mat_d(\cO)\times\Mat_d(\cO)$ is of constant $\cO$--rank and a direct summand. Therefore, it is of the form
\[
\cI = \cI_1\times\cI_2
\]
where $\cI_1,\cI_2 \subseteq \Mat_d(\cO)$ are two $\cO$--right ideals of constant rank which are direct summands of $\Mat_d(\cO)$.

\noindent\ref{left_ann_facts_i}: By \Cref{submodule_ideal_inverse}, each $\cI_i = \cI_{\cV_i} = \cHom_\cO(\cO^d,\cV_i)$ for direct summands $\cV_i \subseteq \cO^d$. Equivalently,
\[
\cI_i = \left\{ \begin{bmatrix} c_1 & \cdots & c_d \end{bmatrix} \mid c_j \in \cV_i \right\}
\]
writing the $c_j$ as columns. Thus, it is clear that
\[
\lann\cI_i = \left\{ \begin{bmatrix} r_1 \\ \vdots \\ r_d \end{bmatrix} \mid r_j \in \cV_i^\perp \right\}
\]
writing the $r_j$ as rows. Thus, $(\lann\cI_i)^t = \cHom_\cO(\cO^d,\cV_i^\perp)$ which is a direct summand of $\Mat_d(\cO)$ since $\cV_i^\perp \subseteq \cO^d$ is a direct summand. Overall, we have that
\[
\tau_d(\lann\cI) = \cHom_\cO(\cO^d,\cV_2^\perp)\times \cHom_\cO(\cO^d,\cV_1^\perp)
\]
which is an $\cL$--ideal and a direct summand of $\Mat_d(\cO)\times\Mat_d(\cO)$.

\noindent\ref{left_ann_facts_v}: This follows from \ref{left_ann_facts_i} since $(\cV_i^\perp)^\perp = \cV_i$.

\noindent\ref{left_ann_facts_ii}: This is clear from \ref{left_ann_facts_i} since each $\cV_i^\perp$ are finite locally free $\cO$--modules.

\noindent\ref{left_ann_facts_iii}: We know that $\rank(\cI_i)= r_id$ for some integers $0\leq r_i \leq d$, and therefore
\[
\rank_\cO(\cI) = (r_1+r_2)d
\]
which is divisible by $d$.

\noindent\ref{left_ann_facts_iv}: Once again using the description in \ref{left_ann_facts_i}, we know that $\rank_\cO(\cV_i^\perp)=d-\rank_\cO(\cV_i)$ and therefore
\[
\rank_\cO((\lann\cI_i)^t) = d^2-r_id.
\]
Together, this means that
\begin{align*}
\rank_\cO(\tau(\lann\cI)) &= d^2-r_1d + d^2-r_2d \\
&= 2d^2 - \rank_\cO(\cI)
\end{align*}
as claimed.

\noindent \ref{left_ann_facts_vi}: We may work with $\cI_1\times\cI_2 \subseteq \cJ_1\times\cJ_2$ and set $\rank_\cO(\cJ_i)=s_i$. Then,
\begin{align*}
& \gap_\cL(\tau(\lann(\cJ_1\times\cJ_2))\subseteq \tau(\lann(\cI_1\times\cI_2))) \\
=& \gap_\cL((\lann\cJ_2)^t\times(\lann\cJ_1)^t \subseteq (\lann\cI_2)^t\times(\lann\cI_1)^t) \\
=& \min(\{d^2-r_2-(d^2-s_2),d^2-r_1-(d^2-s_1)\}) \\
=& \min(\{s_2-r_2,s_1-r_1\}) \\
=& \gap_\cL(\cI_1\times\cI_2\subseteq \cJ_1\times\cJ_2)
\end{align*}
as claimed.
\end{proof} 

\begin{defn}
Let $(f\colon L \to S,\cB,\tau)$ be an Azumaya algebra with involution of the second kind of constant degree $d$. We define
\[
\ConIdeal_{\cL,\cB} \colon \Sch_S \to \Sets
\]
to be the subsheaf of $\ConFlag_{\cL,\cH}$ where the flags additionally consist of right $\cL$--ideals. They will still be locally direct summands with constant $\cO$--rank and with strictly positive $\cL$--gap between subsequent components of the flags.
\end{defn}

By \Cref{left_ann_facts}, we have an order two automorphism of $\ConIdeal_{\cL,\cB}$ given by
\begin{align}
\ConIdeal_{\cL,\cB} &\iso \ConIdeal_{\cL,\cB} \label{ConIdeal_LB_involution} \\
(0\subseteq \cI_1 \subseteq \ldots \subseteq \cI_\ell \subseteq f_*(\cB)) &\mapsto (0\subseteq \tau(\lann\cI_1) \subseteq \ldots \subseteq \tau(\lann\cI_\ell) \subseteq f_*(\cB)). \nonumber
\end{align}

The next step is to define a sheaf which is the sheafification of $\ConIdeal_{\cL,\cB}$.
\begin{defn}
Let $(f\colon L \to S,\cB,\tau)$ be an Azumaya algebra with involution of the second kind of constant degree $d$. We define the sheaf
\[
\LowGSB_{\cL,\cB} \colon \Sch_S \to \Sets
\]
to be the subsheaf of $\LowGSB_{f_*(\cB)}$ where each component of the flag is a right $\cL$--ideal and the flag is $\cL$--lowered. In particular, each component will be locally a direct summand and which is finite locally free as an $\cO$--module.

Symmetrically, we define $\RaiGSB_{\cL,\cB}$ to be the sheaf of $\cL$--raised flags of right $\cL$--ideals in $\RaiGSB_{f_*(\cB)}$.
\end{defn}

\begin{lem}\label{sheafification_ConIdeal_LB}
The canonical inclusion $\ConIdeal_{\cL,\cB} \inj \LowGSB_{\cL,\cB}$ has the universal property of sheafification and so $(\ConIdeal_{\cL,\cB})^\sharp \cong \LowGSB_{\cL,\cB}$. Thus, the order two automorphism \eqref{ConIdeal_LB_involution} induces an order two automorphism
\[
\pi_\cB \colon \LowGSB_{\cL,\cB} \iso \LowGSB_{\cL,\cB}.
\]
\end{lem}
\begin{proof}
It is sufficient to show the claim about sheafification and thus it suffices to see that any section of $\LowGSB_{\cL,\cB}$ locally belongs to the image of $\ConIdeal_{\cL,\cB}$. However, the sections of $\LowGSB_{f_*(\cH)}$ which consist of right $\cL$--ideals will belong to $\LowFlag_{\cL,\cB}$, and therefore locally belong to $\ConFlag_{\cL,\cB}$. Then, because they consist of right ideals they will locally belong to $\ConIdeal_{\cL,\cB}$, so we are done.
\end{proof}
This map $\pi_\cB$ has a similar description as $\pi_h$ in \Cref{pi_h_behaviour_example} except with ideals $\cI$ in place of submodules $\cV$ and $\tau(\lann\cI)$ in place of $\cV^{\perp_h}$.

\begin{defn}\label{defn_ConIdeal_Bt}
Let $(L \to S,\cB,\tau)$ be an Azumaya algebra with involution of the second kind of constant degree $d$. We define
\[
\ConIdeal_{(\cB,\tau)} \colon \Sch_S \to \Sets
\]
to be the subpresheaf of $\ConIdeal_{\cL,\cB}$ consisting of flags fixed by the order two automorphism of \eqref{ConIdeal_LB_involution}. Similarly, we define
\[
\LowGSB_{(\cB,\tau)} \colon \Sch_S \to \Sets
\]
to be the subsheaf of $\LowGSB_{\cL,\cB}$ consisting of flags fixed by the automorphism $\pi_\cB$ of \Cref{sheafification_ConIdeal_LB}. This is the \emph{generalized Severi-Brauer sheaf} of $(L\to S,\cB,\tau)$.
\end{defn}

\begin{lem}
The sheaf $\LowGSB_{(\cB,\tau)}$ defined above is in fact a sheaf. Furthermore, the inclusion $\ConIdeal_{(\cB,\tau)} \inj \LowGSB_{(\cB,\tau)}$ has the universal property of sheafification, thus showing that
\[
(\ConIdeal_{(\cB,\tau)})^\sharp \cong \LowGSB_{(\cB,\tau)}.
\]
\end{lem}
\begin{proof}
The is the analogous statement to \Cref{sheafification_of_hermitian_ConFlag} and follows from a similar argument.
\end{proof}

\begin{example}\label{example_split_outer_ConIdeal}
If we look at the split Azumaya algebra with with involution of the second kind as in \eqref{split_second_involution}, for $T\in \Sch_S$ the sections of $\ConIdeal_{\cL,(\Mat_d(\cO),\Mat_d(\cO))}(T)$ are $\cL|_T$--lowered flags of the form
\[
0 \subseteq \cI_1\times\cJ_1 \subseteq \ldots \subseteq \cI_\ell\times\cJ_\ell \subseteq \Mat_d(\cO)|_T\times\Mat_d(\cO)|_T
\]
where each $\cI_j$ and $\cJ_j$ are right ideals of constant rank in $\Mat_d(\cO)|_T$. We also know from the calculation in \Cref{left_ann_facts}\ref{left_ann_facts_i} that
\[
\tau_d(\lann(\cI_j\times\cJ_j)) = (\lann\cJ_j)^t \times (\lann\cI_j)^t.
\]
Thus, the sections of $\ConIdeal_{(\Mat_d(\cO),\Mat_d(\cO)),\tau_d)}(T)$ appear as
\[
0 \subseteq \cI_1\times (\lann\cI_\ell)^t \subseteq \ldots \subseteq \cI_\ell\times (\lann\cI_1)^t \subseteq \Mat_d(\cO)|_T\times\Mat_d(\cO)|_T.
\]
Since these flags have strictly positive $\cL$--gaps, the map of presheaves
\begin{align*}
\ConIdeal_{\Mat_d(\cO)} &\to \ConIdeal_{((\Mat_d(\cO),\Mat_d(\cO)),\tau_d)} \\
(0 \subseteq \cI_1 \subseteq \ldots \subseteq \cI_\ell \subseteq \Mat_d(\cO)) &\mapsto (0 \subseteq \cI_1\times (\lann\cI_\ell)^t \subseteq \ldots \subseteq \cI_\ell\times (\lann\cI_1)^t \subseteq (\Mat_d(\cO)|_T)^2)
\end{align*}
is injective and locally surjective (though not presheaf surjective). Thus, it induces an isomorphism of the sheafifications
\begin{equation}\label{eq_ConIdeal_iso}
\LowGSB_{\Mat_d(\cO)} \iso \LowGSB_{((\Mat_d(\cO),\Mat_d(\cO)),\tau_d)}.
\end{equation}
Similar to the description given for \eqref{eq_ConFlag_iso}, denoting
\[
\varphi_{\cFlag}(0\subseteq \cI_1 \subseteq \ldots \subseteq \cI_\ell \subseteq \Mat_d(\cO)) = (0\subseteq \cJ_1 \subseteq \ldots \subseteq \cJ_\ell \subseteq \Mat_d(\cO)) \in \RaiGSB_{\Mat_d(\cO)},
\]
the map \eqref{eq_ConIdeal_iso} appears as
\begin{align*}
&(0\subseteq \cI_1 \subseteq \ldots \subseteq \cI_\ell \subseteq \Mat_d(\cO)) \\
\mapsto &(0\subseteq \cI_1\times(\lann\cJ_\ell)^t \subseteq \ldots \subseteq \cI_\ell\times(\lann\cJ_1)^t \subseteq \Mat_d(\cO)\times\Mat_d(\cO)).
\end{align*}

There is a natural action of $\PGL_d\rtimes\ZZ/2\ZZ$ on $\ConIdeal_{((\Mat_d(\cO),\Mat_d(\cO)),\tau_d)}$ where $\varphi \in \PGL_d$ acts by
\begin{align*}
&\varphi \cdot (0 \subseteq \cI_1\times (\lann\cI_\ell)^t \subseteq \ldots \subseteq \cI_\ell\times (\lann\cI_1)^t \subseteq \Mat_d(\cO)|_T\times\Mat_d(\cO)|_T) \\
= &(0 \subseteq \varphi\cI_1\times \theta'(\varphi)(\lann\cI_\ell)^t \subseteq \ldots \subseteq \varphi\cI_\ell\times \theta'(\varphi)(\lann\cI_1)^t \subseteq \Mat_d(\cO)|_T\times\Mat_d(\cO)|_T),
\end{align*}
using the outer automorphism $\theta'$ of $\PGL_d$ defined in \Cref{reductive_groups}, and $\overline{1}\in \ZZ/2\ZZ$ acts by
\begin{align*}
&\overline{1} \cdot (0 \subseteq \cI_1\times (\lann\cI_\ell)^t \subseteq \ldots \subseteq \cI_\ell\times (\lann\cI_1)^t \subseteq \Mat_d(\cO)|_T\times\Mat_d(\cO)|_T) \\
= &(0 \subseteq (\lann\cI_\ell)^t\times \cI_1 \subseteq \ldots \subseteq (\lann\cI_1)^t\times \cI_\ell \subseteq \Mat_d(\cO)|_T\times\Mat_d(\cO)|_T).
\end{align*}
This action extends to a $\PGL_d\rtimes \ZZ/2\ZZ$--action on $\LowGSB_{((\Mat_d(\cO),\Mat_d(\cO)),\tau_d)}$.
\end{example}

\begin{lem}\label{severi_brauer_twist}
Let $(L\to S,\cB,\tau)$ be an Azumaya algebra with involution of the second kind of constant degree $d$ and let $\cK'$ be the corresponding $\PGL_d\rtimes \ZZ/2\ZZ$--torsor. Then,
\[
\LowGSB_{(\cB,\tau)} \cong \cK' \wedge^{\PGL_d\rtimes \ZZ/2\ZZ} \LowGSB_{((\Mat_d(\cO),\Mat_d(\cO)),\tau_d)}.
\]
\end{lem}
\begin{proof}
This is the analogous fact to \Cref{hermitian_flag_twist} and follows by a similar argument.
\end{proof}

\begin{lem}\label{U_action_on_GSB_Btau}
Let $(f\colon L \to S,\cB,\tau)$ be an Azumaya algebra with involution of the second kind of constant degree $d$. The natural action of $\bU_{(\cB,\tau)}$ on $f_*(\cB)$ by conjugation extends to an action on $\LowGSB_{(\cB,t)}$. This extended action is also by conjugation, or equivalently by left multiplication since the flags are comprised of right ideals. That is, for $b\in \bU_{(\cB,\tau)}$ it acts by
\[
b\cdot (0\subseteq \cI_1 \subseteq \ldots \subseteq \cI_\ell \subseteq f_*(\cB)) = (0\subseteq b\cI_1 \subseteq \ldots \subseteq b\cI_\ell \subseteq f_*(\cB)).
\]
\end{lem}
\begin{proof}
We show that this action restricts to a well-defined action on $\ConIdeal_{(\cB,\tau)}$ which in turn means that it is well-defined on $\LowGSB_{(\cB,\tau)}$ through sheafification. Thus, we argue that the action on a section of $\ConIdeal_{(\cB,\tau)}$ preserves the properties of being a symmetric flag of ideals and of having strictly positive $\cL$--gaps. However, $\bU_{(\cB,\tau)}$ is defined by the property that for $b\in \bU_{(\cB,\tau)}$ we have $b=\tau(b^{-1})$. Therefore, if $(0\subseteq \cI_1 \subseteq \ldots \subseteq \cI_\ell \subseteq f_*(\cB)) \in \ConIdeal_{(\cB,\tau)}$ is a section, then
\[
\tau(\lann(b\cI_j)) = \tau(\lann(\cI_j)b^{-1}) = \tau(b^{-1})\cdot \tau(\lann\cI_j) = b\cI_{\ell+1-j}
\]
and so the resulting section is also symmetric. Regarding $\cL$--gaps, the action of $\bU_{(\cB,\tau)}$ sufficiently locally is by module isomorphism on each factor of $\Mat_d(\cO)\times\Mat_d(\cO)$ and so the ranks of submodules are preserved, which implies that the action of $\bU_{(\cB,\tau)}$ preserves $\cL$--gaps. Hence, the action is well-defined on $\ConIdeal_{(\cB,\tau)}$ and we are done.
\end{proof}

When the Azumaya algebra is the endomorphism algebra of a regular hermitian form $(f\colon L \to S,\cH,h)$ of constant rank $d$, the sheaves $\LowFlag_{(\cH,h)}$ and $\LowGSB_{(\cEnd_{\cO|_L}(\cH),\tau_h)}$ will coincide as expected by sending each subspace occurring in a flag to a suitable right ideal. In particular, let $\cV \subseteq f_*(\cH)$ be an $\cL$--subspace and define the right $\cL$--ideal
\[
\cI_{\cL,\cV} = \cHom_{\cL}(f_*(\cH),\cV) \subseteq \cEnd_{\cL}(f_*(\cH)) \cong f_*(\cEnd_{\cO|_L}(\cH)),
\]
where the isomorphism $\cEnd_{\cL}(f_*(\cH)) \cong f_*(\cEnd_{\cO|_L}(\cH))$ is due to \cite[Tag 01SB]{Stacks}.
\begin{lem}\label{V_to_I_outer}
With notation as above,
\begin{enumerate}
\item \label{V_to_I_outer_i} if $\cV$ is locally a direct summand of $f_*(\cH)$ then $\cI_{\cL,\cV}$ is locally a direct summand of $f_*(\cEnd_{\cO|_L}(\cH))$,
\item \label{V_to_I_outer_ii} if $\cV$ is finite locally free as an $\cO$--module then so is $\cI_{\cL,\cV}$, and
\item \label{V_to_I_outer_iii} $\tau_h(\lann\cI_{\cL,\cV}) = \cI_{\cL,(\cV^{\perp_h})}$.
\end{enumerate}
Additionally, if $\cV \subseteq \cW$ are two such $\cL$--submodules, then
\begin{enumerate}
\setcounter{enumi}{3}
\item \label{V_to_I_outer_iv} $\gap_\cL(\cI_{\cL,\cV}\subseteq \cI_{\cL,\cW}) = d\cdot \gap_\cL(\cV\subseteq \cW)$. In particular, if $\gap_\cL(\cV\subseteq \cW)$ is strictly positive then so is $\gap_\cL(\cI_{\cL,\cV}\subseteq \cI_{\cL,\cW})$.
\end{enumerate}
\end{lem}
\begin{proof}
Working locally, we may assume that $\cV\subseteq f_*(\cH)$ is a direct summand which is of constant $\cO$--rank and show that $\cI_{\cL,\cV}$ has similar properties.

\noindent\ref{V_to_I_outer_i}: If $f_*(\cH) = \cV \oplus \cW$ for another $\cL$--submodule $\cW$, then
\[
\cEnd_{\cL}(f_*(\cH)) = \cHom_{\cL}(f_*(\cH),\cV) \oplus \cHom_{\cL}(f_*(\cH),\cW)
\]
and we see that $\cHom_{\cL}(f_*(\cH),\cV)$ is a direct summand as required.

\noindent\ref{V_to_I_outer_ii}: Localizing further, we may assume that we are working with the split regular hermitian form of \eqref{eq_split_hermitian_form}. Then $\cH_d = (\cO^d,\cO^d)$ and $\cV$ takes the form
\[
\cV = \cV_1\times\cV_2 \subseteq \cO^d\times\cO^d
\]
for constant rank $\cO$--submodules $\cV_i\subseteq \cO^d$. Then, since $\cI_{\cL,\cV}$ consists of $\cO\times\cO$--linear maps, we have that
\[
\cI_{\cL,\cV} = \cHom_\cO(\cO^d,\cV_1)\times\cHom_\cO(\cO^d,\cV_2) = \cI_{\cV_1}\times\cI_{\cV_2}
\]
which is locally free of constant rank $d\rank_\cO(\cV_1)+d\rank_\cO(\cV_2)$.

\noindent\ref{V_to_I_outer_iii}: It also suffices to show that this holds for the split hermitian form. For $\cV= \cV_1\times\cV_2 \subseteq \cO^d\times\cO^d$ we have that $\cI_{\cL,\cV} = \cI_{\cV_1}\times\cI_{\cV_2}$ and the calculations in \Cref{left_ann_facts}\ref{left_ann_facts_i} show that
\begin{align*}
\tau_d(\lann\cI_{\cL,\cV}) &= \cHom_\cO(\cO^d,\cV_2^\perp)\times\cHom_\cO(\cO^d,\cV_1^\perp) \\
&= \cHom_{\cO\times\cO}(\cO^d\times\cO^d,\cV_2^\perp\times\cV_1^\perp) \\
&= \cHom_{\cO\times\cO}(\cO^d\times\cO^d,(\cV_1\times\cV_2)^{\perp_{h_d}}) \\
&= \cI_{\cL,(\cV^{\perp_{h_d}})}
\end{align*}
as desired.

\noindent\ref{V_to_I_outer_iv}: Localizing as in \ref{V_to_I_outer_ii} our $\cL$--submodules appear as $\cV=\cV_1\times\cV_2\subseteq \cW_1\times\cW_2=\cW$ and our ideals appear as
\[
\cI_{\cL,\cV} = \cI_{\cV_1}\times\cI_{\cV_2} \subseteq \cI_{\cW_1}\times\cI_{\cW_2} = \cI_{\cL,\cW}.
\]
Using the calculations above we compute
\begin{align*}
&\subrank_\cL(\cI_{\cL,\cV}\subseteq \cI_{\cL,\cW}) \\
=& \min(\{d\rank_\cO(\cW_1)-d\rank_\cO(\cV_1),d\rank_\cO(\cW_2)-d\rank_\cO(\cV_2)\}) \\
&= d\cdot \min(\{\rank_\cO(\cW_1)-\rank_\cO(\cV_1),\rank_\cO(\cW_2)-\rank_\cO(\cV_2)\}) \\
&= d\cdot \gap_\cL(\cV\subseteq \cW)
\end{align*}
and therefore $\gap_\cL(\cI_{\cL,\cV}\subseteq \cI_{\cL,\cW}) = d\cdot \gap_\cL(\cV\subseteq \cW)$ holds globally as well, as claimed.
\end{proof}

\begin{prop}\label{iso_LowFlag_Hh_LowGSB}
Let $(f\colon L \to S, \cH, h)$ be a regular hermitian form of constant rank $d$ and consider the associated Azumaya algebra with involution of the second kind $(f\colon L \to S, \cEnd_{\cO|_L}(\cH),\tau_h)$. Then, the map
\begin{align*}
\LowFlag_{(\cH,h)} &\to \LowGSB_{(\cEnd_{\cO|_L}(\cH),\tau_h)} \\
(0 \subseteq \cV_1 \subseteq \ldots \subseteq \cV_\ell \subseteq \cH) &\mapsto (0 \subseteq \cI_{\cL,\cV_1} \subseteq \ldots \subseteq \cI_{\cL,\cV_\ell} \subseteq \cH),
\end{align*}
is well defined, i.e. it sends $\cL$--lowered symmetric flags to $\cL$--lowered symmetric flags of ideals, and it is an isomorphism of sheaves. Additionally, this map is equivariant under the actions of $\GL_{(\cH,h)}\cong \bU_{(\cEnd_{\cO|_L}(\cH),\tau_h)}$ given in \Cref{GL_action_on_LowFlag_Hh} and \Cref{U_action_on_GSB_Btau}.
\end{prop}
\begin{proof}
To see that the map is well-defined, we instead argue that the analogous map on presheaves
\begin{align*}
\ConFlag_{(\cH,h)} &\to \ConIdeal_{(\cEnd_{\cO|_L}(\cH),\tau_h)} \\
(0 \subseteq \cV_1 \subseteq \ldots \subseteq \cV_\ell \subseteq \cH) &\mapsto (0 \subseteq \cI_{\cL,\cV_1} \subseteq \ldots \subseteq \cI_{\cL,\cV_\ell} \subseteq \cH),
\end{align*}
is a well-defined presheaf morphism. If this is the case, it will induce the claimed map by sheafification. The fact that the map between presheaves is well-defined follows directly from \Cref{V_to_I_outer} and the fact that symmetric flags are send to symmetric flags follows because the condition that $\cV_i = \cV_{\ell+1-i}^{\perp_h}$ implies that
\[
\cI_{\cL,\cV_i} = \cI_{\cL,(\cV_{\ell+1-i}^{\perp_h})} = \tau_h(\lann\cI_{\cL,\cV_{\ell+1-i}}).
\]

Now, to see that the map is an isomorphism, we may show that it is an isomorphism locally when the hermitian form becomes the split hermitian form. In that case, we have a commutative diagram
\[
\begin{tikzcd}
\LowFlag_{((\cO^d,\cO^d),h_d)} \ar[r] & \LowGSB_{((\Mat_d(\cO),\Mat_d(\cO)),\tau_d)}  \\
\LowFlag_{\cO^d} \ar[r,"\sim"] \ar[u,swap,"\eqref{eq_ConFlag_iso}"] & \LowGSB_{\Mat_d(\cO)} \ar[u,swap,"\eqref{eq_ConIdeal_iso}"]
\end{tikzcd}
\]
where the vertical maps are isomorphisms and the bottom map is the isomorphism of \Cref{flags_to_SB}. Thus, our map in question, the top map, is an isomorphism as well. 

To see that the map is equivariant as claimed, it is immediate to check that for $b\in \GL_{(\cH,h)}$ and $\cV\subset f_*(\cH)$ an appropriate submodule, we have
\[
\cI_{\cL,(b\cV)} = \cHom_{\cL}(f_*(\cH),b\cV) = b\cHom_{\cL}(f_*(\cH),\cV) = b\cdot \cI_{\cL,\cV}.
\]
This concludes the proof.
\end{proof}

\subsection{Isomorphisms with Parabolic Subgroups}
Mirroring the steps of \Cref{isos_with_parabolics_inner}, we first consider the case when $\bG = \GL_{(\cH,h)} \cong \bU_{(\cEnd(\cH),\tau_h)}$ for a regular hermitian form $(L\to S,\cH,h)$. Given a section
\[
\overline{\cV} = (0\subseteq \cV_1 \subseteq \ldots \subseteq \cV_\ell \subseteq f_*(\cH)) \in \LowFlag_{(\cH,h)},
\]
we set $\bStab(\overline{\cV}) \subseteq \bG$ to be the stabilizer of this section under the action of \Cref{GL_action_on_LowFlag_Hh}. Likewise, given a section
\[
\overline{\cI} = (0\subseteq \cI_1 \subseteq \ldots \subseteq \cI_\ell \subseteq f_*(\cEnd_{\cO|_L}(\cH))) \in \LowGSB_{(\cEnd(\cH),\tau_h)},
\]
we set $\bStab(\overline{\cI})\subseteq \bG$ to be the stabilizer under the action of \Cref{U_action_on_GSB_Btau}.

\begin{lem}
The stabilizers $\bStab(\overline{\cV})$ and $\bStab(\overline{\cI})$ are parabolic subgroups of $\bG = \GL_{(\cH,h)} \cong \bU_{(\cEnd(\cH),\tau_h)}$.
\end{lem}
\begin{proof}
We may show that the subgroup is parabolic locally to conclude that it is a parabolic subgroup. Sufficiently locally, we are in the case of \Cref{example_split_ConFlag_Hh} and $\GL_{((\cO^d,\cO^d),h_d)} \cong \GL_d$ where a section $a\in \GL_d$ acts as in \eqref{GL_d_action_on_ConFlag_Hd} on flags of submodules. Thus,
\begin{align*}
&\bStab(0\subseteq \cV_1\times\cV_\ell^\perp \subseteq \ldots \subseteq \cV_\ell\times\cV_1^\perp \subseteq \cO^d\times\cO^d) \\
= &\bStab(0\subseteq \cV_1\subseteq \ldots \subseteq \cV_\ell \subseteq \cO^d)
\end{align*}
which we know is a parabolic subgroup by \Cref{flag_stabilizers_are_parabolic}. The claim for $\bStab(\overline{\cI})$ then follows from \Cref{iso_LowFlag_Hh_LowGSB}.
\end{proof}

\begin{lem}\label{iso_hermitian_flags_parabolics}
Let $(L\to S,\cH,h)$ be a regular hermitian form of constant rank $d$ and set $\bG = \GL_{(\cH,h)} = \bU_{(\cEnd(\cH),\tau_h)}$. The two maps $\LowFlag_{(\cH,h)} \to \cPar_\bG$ and $\LowGSB_{(\cEnd(\cH),\tau_h)} \to \cPar_\bG$ which send a flag to its stabilizer are both isomorphisms. Furthermore, they fit into the commutative diagram of isomorphisms
\[
\begin{tikzcd}
\LowFlag_{(\cH,h)} \ar[rr] \ar[dr] & & \LowGSB_{(\cEnd(\cH),\tau_h)} \ar[dl] \\
& \cPar_\bG & 
\end{tikzcd}
\]
where the horizontal map is the isomorphism of \Cref{iso_LowFlag_Hh_LowGSB}.
\end{lem}
\begin{proof}
Working sufficiently locally, this becomes exactly the situation of \Cref{flag_isomorphism_with_Par}. Thus, it follows globally as well.
\end{proof}

In general, for an Azumaya algebra with involution of the second kind $(L\to S,\cB,\tau)$ which is not an endomorphism algebra and a section $\overline{\cI} \in \LowGSB_{(\cB,\tau)}$ we also set $\bStab(\overline{\cI}) \subseteq \bU_{(\cB,\tau)}$ to be the stabilizer under the action of \Cref{U_action_on_GSB_Btau}.
\begin{cor}\label{iso_outer_GSB_to_parabolics}
Let $(L\to S,\cB,\tau)$ be an Azumaya algebra with involution of the second kind of constant degree $d$. Then, the map
\[
\LowGSB_{(\cB,\tau)} \to \cPar_{\bU_{(\cB,\tau)}},
\]
which sends a flag to its stabilizer, is an isomorphism of sheaves.
\end{cor}
\begin{proof}
Locally, this becomes an isomorphism of the form $\LowGSB_{(\cEnd(\cH),\tau_h)} \to \cPar_{\GL_{(\cH,h)}}$ as in \Cref{iso_hermitian_flags_parabolics} and thus it is also an isomorphism globally.
\end{proof}

\subsection{Type Morphisms for the Outer Case}
Here we define type morphisms for hermitian flags and outer Severi-Brauer sheaves which are compatible with the type morphism \eqref{type_map_parabolics} of \cite{SGA3} and thus land in $\cOD{\bG}$ for an appropriate group $\bG$. We obtain these maps by twisting the maps defined in \Cref{type_morphisms_inner}.

We start with the split hermitian form of constant rank $d=n+1$ of \eqref{eq_split_hermitian_form} and consider the composition
\[
\ConFlag_{((\cO^d,\cO^d),h_d)} \iso \ConFlag_{\cO^d} \xrightarrow{\eqref{eq_ConFlag_type}} \cP_n
\]
of the isomorphism from \Cref{example_split_ConFlag_Hh} with the type map defined in the inner case. Since $\cP_n$ is a sheaf, this induces a type morphism
\begin{equation}\label{eq_split_hermitian_type}
\LowFlag_{((\cO^d,\cO^d),h_d)} \to \cP_n.
\end{equation}
Recall that $\ZZ/2\ZZ$ acts on $\cP_n$ via the isomorphism $\und^\vee \colon \cP_n \to \cP_n$, and thus by acting trivially with $\GL_d$ we obtain a $\GL_d\rtimes\ZZ/2\ZZ$--action on $\cP_n$.
\begin{lem}
In the setting above, the map
\[
\LowFlag_{((\cO^d,\cO^d),h_d)} \to \cP_n
\]
is $\GL_d\rtimes\ZZ/2\ZZ$--equivariant.
\end{lem}
\begin{proof}
It is sufficient to show that the composition $\ConFlag_{((\cO^d,\cO^d),h_d)} \to \cP_n$ is equivariant. We know that both the isomorphism $\ConFlag_{((\cO^d,\cO^d),h_d)} \iso \ConFlag_{\cO^d}$ and the map \eqref{eq_ConFlag_type} are equivariant under the action of $\GL_d$, so we only need to check $\ZZ/2\ZZ$--equivariance. Consider a section
\[
0\subseteq \cV_1\times\cV_\ell^\perp \subseteq \ldots \subseteq \cV_\ell\times\cV_1^\perp \subseteq \cO^d\times\cO^d
\]
of $\ConFlag_{((\cO^d,\cO^d),h_d)}$. Under the map in consideration, this is sent to the tuple
\[
(\rank(\cV_1),\rank(\cV_2),\ldots,\rank(\cV_\ell)).
\]
However, after applying the action of $\overline{1}\in \ZZ/2\ZZ$ we have the section
\[
0\subseteq \cV_\ell^\perp\times\cV_1 \subseteq \ldots \subseteq \cV_1^\perp\times\cV_\ell \subseteq \cO^d\times\cO^d
\]
which is sent to the tuple
\begin{align*}
&(\rank(\cV_\ell^\perp),\ldots,\rank(\cV_2^\perp),\rank(\cV_1^\perp)) \\
=&(d-\rank(\cV_\ell),\ldots,d-\rank(\cV_2),d-\rank(\cV_1)) \\
=& (\rank(\cV_1),\rank(\cV_2),\ldots,\rank(\cV_\ell))^\vee.
\end{align*}
Thus, we also have $\ZZ/2\ZZ$--equivariance as desired. 
\end{proof}

\begin{defn}
let $(L\to S,\cH,h)$ be a regular hermitian form of constant rank $d=n+1$ and let $\cK$ be its associated $\GL_d\rtimes \ZZ/2\ZZ$--torsor. Set $\bG = \GL_{(\cH,h)}$. Due to equivariance, we may twist the map \eqref{eq_split_hermitian_type} by the torsor $\cK$. Using \Cref{hermitian_flag_twist} and \eqref{eq_OD_G}, we obtain a map 
\[
\widetilde{t_{\cFlag}}\colon \LowFlag_{(\cH,h)} \to \cOD{\bG}
\]
which we call the \emph{type morphism} for the sheaf $\LowFlag_{(\cH,h)}$.
\end{defn}

We now work with the split Azumaya algebra with involution of the second kind of constant rank $d=n+1$ as in \eqref{split_second_involution} and \Cref{example_split_outer_ConIdeal}. We have a composition
\[
\ConIdeal_{((\Mat_d(\cO),\Mat_d(\cO)),\tau_d)} \iso \ConIdeal_{\Mat_d(\cO)} \xrightarrow{\eqref{eq_ConIdeal_type}} \cP_n
\]
where the first map is the isomorphism of \Cref{example_split_outer_ConIdeal}. This induces a type map
\begin{equation}\label{eq_split_outer_SB_type}
\LowGSB_{((\Mat_d(\cO),\Mat_d(\cO)),\tau_d)} \to \cP_n.
\end{equation}
The sheaf $\LowGSB_{((\Mat_d(\cO),\Mat_d(\cO)),\tau_d)}$ has the $\PGL_d\rtimes\ZZ/2\ZZ$--action described in \Cref{example_split_outer_ConIdeal}, and by acting trivially with $\PGL_d$ we also have a $\PGL_d\rtimes\ZZ/2\ZZ$--action on $\cP_n$.

\begin{lem}
In the setting above, the map
\[
\LowGSB_{((\Mat_d(\cO),\Mat_d(\cO)),\tau_d)} \to \cP_n
\]
is $\PGL_d\rtimes\ZZ/2\ZZ$--equivariant.
\end{lem}
\begin{proof}
It is clear that it is equivariant for the $\PGL_d$--action. To check equivariance for the $\ZZ/2\ZZ$--action, we check that the underlying map
\[
\ConIdeal_{((\Mat_d(\cO),\Mat_d(\cO)),\tau_d)} \to \cP_n
\]
is equivariant. A section
\[
0 \subseteq \cI_1\times (\lann\cI_\ell)^t \subseteq \ldots \subseteq \cI_\ell\times(\lann\cI_1)^t \subseteq \Mat_d(\cO)\times\Mat_d(\cO)
\]
is mapped to the tuple
\[
\left(\frac{\rank(\cI_1)}{d},\frac{\rank(\cI_2)}{d},\ldots,\frac{\rank(\cI_\ell)}{d}\right).
\]
After applying $\overline{1}\in\ZZ/2\ZZ$, we obtain the section
\[
0 \subseteq (\lann\cI_\ell)^t\times\cI_1 \subseteq \ldots \subseteq (\lann\cI_1)^t\times\cI_\ell \subseteq \Mat_d(\cO)\times\Mat_d(\cO)
\]
which is mapped to the tuple
\begin{align*}
&\left(\frac{\rank((\lann\cI_\ell)^t)}{d},\ldots,\frac{\rank((\lann\cI_2)^t)}{d},\frac{\rank((\lann\cI_1)^t)}{d}\right) \\
= &\left(\frac{d^2-\rank(\cI_\ell)}{d},\ldots,\frac{d^2-\rank(\cI_2)}{d},\frac{d^2-\rank(\cI_1)}{d}\right) \\
= &\left(d-\frac{\rank(\cI_\ell)}{d},\ldots,d-\frac{\rank(\cI_2)}{d},d-\frac{\rank(\cI_1)}{d}\right) \\
= &\left(\frac{\rank(\cI_1)}{d},\frac{\rank(\cI_2)}{d},\ldots,\frac{\rank(\cI_\ell)}{d}\right)^\vee.
\end{align*}
This shows $\ZZ/2\ZZ$--equivariance as claimed.
\end{proof}

\begin{defn}\label{defn_outer_GSB_type}
Let $(L\to S, \cB,\tau)$ be an Azumaya algebra with involution of the second type of constant degree $d$ and let $\cK'$ be its associated $\PGL_d\rtimes\ZZ/2\ZZ$--torsor. Set $\bG = \bU_{(\cB,\tau)}$. Due to equivariance, we may twist the map \eqref{eq_split_outer_SB_type} by the torsor $\cK'$. Using \Cref{severi_brauer_twist} and \eqref{eq_OD_G}, we obtain a map
\[
\widetilde{t_{\cGSB}}\colon \LowGSB_{(\cB,\tau)} \to \cOD{\bG}
\]
which we call the \emph{type morphism} for the sheaf $\LowGSB_{(\cB,\tau)}$.
\end{defn}

\begin{lem}\label{type_commutativity_outer}
Let $(L\to S,\cH,h)$ be a regular hermitian form of constant rank $d=n+1$. Set $\bG = \GL_{(\cH,h)} = \bU_{(\cEnd(\cH),\tau_h)}$. The commutative diagram of \Cref{iso_hermitian_flags_parabolics} extends to include type morphisms in the following way. 
\[
\begin{tikzcd}
\LowFlag_{(\cH,h)} \ar[rr] \ar[dr] \ar[dd,swap,"\widetilde{t_{\cFlag}}"] & & \LowGSB_{(\cEnd(\cH),\tau_h)} \ar[dl] \ar[dd,"\widetilde{t_{\cGSB}}"] \\[-5ex]
 & \cPar_\bG \ar[d,"t"] & \\
\cOD{\bG} \ar[r,"\und^c"] & \cOD{\bG} \ar[r,"\und^c"] & \cOD{\bG}.
\end{tikzcd}
\]
In particular, since the map $(\und^c)$ is order $2$, the diagram
\[
\begin{tikzcd}
\LowFlag_{(\cH,h)} \ar[rr] \ar[dr,swap,"\widetilde{t_{\cFlag}}"] & & \LowGSB_{(\cEnd(\cH),\tau_h)} \ar[dl,"\widetilde{t_{\cGSB}}"] \\
 & \cOD{\bG} & 
\end{tikzcd}
\]
commutes.
\end{lem}
\begin{proof}
Locally this becomes the statement of \Cref{type_commutativity_inner}, and so the claims hold globally as well.
\end{proof}

\begin{cor}
Let $(L\to S,\cB,\tau)$ be an Azumaya algebra with involution of the second kind of constant degree $d$ and let $\bG = \bU_{(\cB,\tau)}$. We have a commutative diagram
\[
\begin{tikzcd}
\LowGSB_{(\cB,\tau)} \ar[r] \ar[d,"\widetilde{t_{\cGSB}}"] & \cPar_\bG \ar[d,"t"] \\
\cOD{\bG} \ar[r,"\und^c"] & \cOD{\bG}
\end{tikzcd}
\]
\end{cor}
\begin{proof}
This holds globally since it holds locally by \Cref{type_commutativity_outer}.
\end{proof}

\section{Idempotents and Levi Subgroups}\label{idempotent_section}
For $\bG$ a linear algebraic group, we denote by $\cPL_\bG \colon \Sch_S \to \Sets$ the sheaf of pairs of parabolic and Levi subgroups. In detail, for $T\in \Sch_S$, we have
\[
\cPL_\bG(T) = \{(\bP,\bL) \mid \bP \subset \bG|_T \text{ is a parabolic, } \bL \subset \bP \text{ is a Levi subgroup}\}.
\]
Of course, there is an obvious surjective map
\begin{align}
\cPL_\bG &\surj \cPar_\bG \label{PL_to_Par_surj} \\
(\bP,\bL) &\mapsto \bP.\nonumber
\end{align}
In \Cref{outer_case} we showed how $\cPar_\bG$ can be realized as $\LowGSB_{(\cB,\tau)}$ for an Azumaya algebra with involution of the second type $(L\to S,\cB,\tau)$, or as $\LowFlag_{(\cH,h)}$ if the algebra is the endomorphism algebra of a regular hermitian form $(L\to S,\cH,h)$. We now develop analogous descriptions of $\cPL_\bG$ in terms of Azumaya algebras and flags.

\subsection{Inner Case}
We begin with the inner case when $\cPar_\bG$ is isomorphic to $\LowGSB_\cA$ or $\LowFlag_\cE$ as in \Cref{inner_case}.
\subsubsection{Direct Summands and Idempotents}
Let $\cE$ be a finite locally free $\cO$--module of constant rank $d = n+1$. We wish to consider a sheaf of direct sum decompositions of $\cE$. Naively, these are tuples of submodules $(\cV_1,\ldots,\cV_{\ell+1})$ such that $\cE = \cV_1\oplus\ldots\oplus\cV_{\ell+1}$. From such a tuple we can construct a flag of partial direct sums
\[
0 \subseteq \cV_1 \subseteq \cV_1\oplus \cV_2 \subseteq \ldots \subseteq \cV_1\oplus\ldots\oplus\cV_\ell \subseteq \cE
\]
and we wish to use this construction to define a map into $\cPar_{\GL_\cE} \cong \LowFlag_\cE$. However, we therefore need to impose suitable conditions on our tuples of submodules in order to ensure that the resulting flags will be lowered flags. This motivates the following definition.

\begin{defn}
Let $\cE$ be a finite locally free $\cO$--module of constant rank $d$. Let $(\cV_1,\ldots,\cV_{\ell+1})$ be a tuple of submodules of $\cE$ such that $\cE = \cV_1\oplus\ldots\oplus \cV_{\ell+1}$. We call such a tuple \emph{lowered} if for all $T\in \Sch_S$ and all $1\leq i \leq \ell$,
\[
\cV_i|_T = 0 \Rightarrow \cV_j|_T = 0 \text{ for all } j\geq i.
\]
We define a functor $\LowStiefel_\cE \colon \Sch_S \to \Sets$, called the \emph{Stiefel functor of $\cE$}, which behaves on $T\in \Sch_S$ by
\[
T\mapsto \{(\cV_1,\ldots,\cV_{\ell+1}) \mid \ell\geq 0, \cV_{\ell+1}\neq 0, \cE|_T=\bigoplus_{j=1}^{\ell+1}\cV_j, \text{and the tuple is lowered}\}
\]
with the restriction along a map $T' \to T$ in $\Sch_S$ being defined by
\[
(\cV_1,\ldots,\cV_{\ell+1}) \mapsto (\cV_1|_{T'},\ldots,\cV_k|_{T'})
\]
where $k \in \{1,\ldots,\ell+1\}$ is the largest integer such that $\cV_k|_{T'}\neq 0$. I.e., since $(\cV_1,\ldots,\cV_{\ell+1})$ is lowered, it naively restricts to a length $\ell+1$ tuple of the form $(\cV_1|_{T'},\ldots,\cV_k|_{T'},0,\ldots,0)$ and we truncate the zeroes.
\end{defn}

The definition of $\LowStiefel_\cE$ using lowered objects and restriction involving truncation is reminiscent of the definition of the sheaf $\LowFlag_\cE$ in \Cref{defn_flag_sheaf}, and we show below that indeed $\LowStiefel_\cE$ is a sheaf as well. It similarly can be viewed as the sheafification of a presheaf of direct sum decompositions involving summands of constant rank.
\begin{defn}
Let $\cE$ be a finite locally free $\cO$--module of constant rank $d$. We define the presheaf $\ConStiefel_\cE \colon \Sch_S \to \Sets$ which behaves as
\[
T \mapsto \{(\cV_1,\ldots,\cV_{\ell+1}) \mid \ell\geq 0, \cV_i\neq 0, \cE|_T = \oplus_{i=1}^{\ell+1} \cV_i \text{ and each } \cV_i \text{ is constant rank}\}.
\]
\end{defn}

\begin{lem}\label{LowStiefel_sheaf}
The functor $\LowStiefel_\cE \colon \Sch_S \to \Sets$ defined above is a sheaf. Furthermore, the inclusion
\[
\ConStiefel_\cE \inj \LowStiefel_\cE
\]
satisfies the universal property of sheafification, showing that $(\ConStiefel_\cE)^\sharp \cong \LowStiefel_\cE$.
\end{lem}
\begin{proof}
If we know that any section $(\cV_1,\ldots,\cV_{\ell+1})\in \LowStiefel_\cE(T)$ must have $\ell<d$, then we may mirror the proof of \Cref{lowered_flags_sheaf}. However, $\cV_{\ell+1}\neq 0$ by assumption and therefore when we find a cover $\{U_i\to T\}_{i\in I}$ over which all $\cV_j$ are of constant rank, there must exist $i_0 \in \cI$ with $\cV_{\ell+1}|_{U_{i_0}} \neq 0$. Thus, $\rank(\cV_j|_{U_{i_0}})\geq 1$ for all $1\leq j \leq \ell+1$, so
\[
d = \rank(\cE|_{U_{i_0}}) = \sum_{j=1}^{\ell+1} \rank(\cV_j|_{U_{i_0}}) \geq \ell+1,
\]
meaning that $\ell < d$. Therefore, we may follow the arguments of \Cref{lowered_flags_sheaf} and see that $\LowStiefel_\cE$ is a sheaf.

Once we know that $\LowStiefel_\cE$ is a sheaf, it is clear that because each component in a tuple is locally of constant rank that therefore any section of $\LowStiefel_\cE$ locally belongs to $\ConStiefel_\cE$, which demonstrates that $(\ConStiefel_\cE)^\sharp \cong \LowStiefel_\cE$ and we are done.
\end{proof}

\begin{rem}\label{rhoStiefel}
Symmetrically, a tuple $(\cV_1,\ldots,\cV_{\ell+1})$ with $\cE=\oplus_{i=1}^{\ell+1} \cV_i$ is \emph{raised} if for all $T\in \Sch_S$,
\[
\cV_i|_T = 0 \Rightarrow \cV_j|_T = 0 \text{ for all } j\leq i
\]
and these form a sheaf $\RaiStiefel_\cE \colon \Sch_S \to \Sets$ whose restriction maps involve truncation on the left. We also have that the inclusion map $\ConStiefel_\cE \inj \RaiStiefel_\cE$ realizes $(\ConStiefel_\cE)^\sharp \cong \RaiStiefel_\cE$ and therefore we have a uniquely defined isomorphism of sheaves
\begin{equation}\label{eq_low_to_rai_Stiefel}
\rhoStiefel \colon \LowStiefel_\cE \iso \RaiStiefel_\cE
\end{equation}
which is constant on sections from $\ConStiefel_\cE$. Similar to the raising isomorphism described in \Cref{raised_flags}, the isomorphism $\rhoStiefel$ has a ``sliding up" effect. For example, for a lowered tuple
\[
(\cV_1,\ldots,\cV_{\ell+1}) \in \LowStiefel_\cE(S)
\]
of length $\ell+1$, there will be a disjoint decomposition $X= W_{j_1} \sqcup W_{j_2} \sqcup \ldots \sqcup W_{j_m} \sqcup W_{\ell+1}$ such that the restriction of the tuple to $W_j$ is of constant length $j$. Thus, writing the tuple vertically, it will appear as
\[
\begin{tikzpicture}
\matrix (M)[matrix of math nodes]{
\cV_{\ell+1} & = & (0 & 0 & \ldots & 0 & \cV_{\ell+1}|_{W_{\ell+1}}) \\
\vdots & \quad & & & \vdots & & \\
\cV_{j_m} & = & (0 & 0 & \ldots & \cV_{j_m}|_{W_{j_m}} & \cV_{j_m}|_{W_{\ell+1}}) \\
\vdots & \quad & & & \vdots & & \\
\vdots & \quad & & & \vdots & & \\
\cV_{j_2} & = & (0 & \cV_{j_2}|_{W_{j_2}} & \ldots & \cV_{j_2}|_{W_{j_m}} & \cV_{j_2}|_{W_{\ell+1}}) \\
\vdots & \quad & & & \vdots & & \\
\cV_{j_1} & = & (\cV_{j_1}|_{W_{j_1}} & \cV_{j_1}|_{W_{j_2}} & \ldots & \cV_{j_1}|_{W_{j_m}} & \cV_{j_1}|_{W_{\ell+1}}) \\
\vdots & \quad & & & \vdots & & \\
\cV_1 & = & (\cV_1|_{W_{j_1}} & \cV_1|_{W_{j_2}} & \ldots & \cV_1|_{W_{j_m}} & \cV_1|_{W_{\ell+1}}) \\
};
\draw[thick,black](M-7-2.south east)--(M-8-4.north west);
\draw[thick,black](M-8-4.north west)--(M-6-4.north west);
\draw[thick,black](M-6-4.north west)--(M-6-4.north east);
\draw[thick,black](M-6-4.north east)--([xshift=-0.9ex,yshift=-1ex]M-5-5.north west);
\draw[thick,black]([xshift=-0.9ex,yshift=-1ex]M-5-5.north west)--([xshift=1ex,yshift=-1ex]M-5-5.north east);
\draw[thick,black]([xshift=1ex,yshift=-1ex]M-5-5.north east)--(M-3-6.north west);
\draw[thick,black](M-3-6.north west)--(M-3-6.north east);
\draw[thick,black](M-3-6.north east)--(M-1-7.north west);
\end{tikzpicture}
\]
with zeroes above the staircase pattern. This flag is mapped by $\rhoStiefel$ to the raised flag
\[
\begin{tikzpicture}
\matrix (M)[matrix of math nodes]{
\cV'_{\ell+1} & = & (\cV_{j_1}|_{W_{j_1}} & \cV_{j_2}|_{W_{j_2}} & \ldots & \cV_{j_m}|_{W_{j_m}} & \cV_{\ell+1}|_{W_{\ell+1}})\\
\vdots & \quad & & & \vdots & & \\
\cV'_{\ell+2-{j_1}} & = & (\cV_1|_{W_{j_1}} & \cV_{j_2+1-j_1}|_{W_{j_2}} & \ldots & \cV_{j_m+1-j_1}|_{W_{j_m}} & \cV_{\ell+2-j_1}|_{W_{\ell+1}})\\
\vdots & \quad & & & \vdots & & \\
\cV'_{\ell+2-{j_2}} & = & (0 & \cV_1|_{W_{j_2}} & \ldots & \cV_{j_m+1-j_2}|_{W_{j_m}} & \cV_{\ell+2-j_2}|_{W_{\ell+1}} )\\
\vdots & \quad & & & \vdots & & \\
\vdots & \quad & & & \vdots & & \\
\cV'_{\ell+2-{j_m}} & = & (0 & 0 & \ldots & \cV_1|_{W_{j_m}} & \cV_{\ell+2-j_m}|_{W_{\ell+1}})\\
\vdots & \quad & & & \vdots & & \\
\cV'_1 & = & (0 & 0 & \ldots & 0 & \cV_1|_{W_{\ell+1}})\\
};
\draw[thick,black](M-3-3.south west)--(M-3-4.south west);
\draw[thick,black](M-3-4.south west)--([xshift=-3ex]M-5-4.south west);
\draw[thick,black]([xshift=-3ex]M-5-4.south west)--([xshift=3ex]M-5-4.south east);
\draw[thick,black]([xshift=3ex]M-5-4.south east)--([xshift=-1ex,yshift=-0.5ex]M-6-5.south west);
\draw[thick,black]([xshift=-1ex,yshift=-0.5ex]M-6-5.south west)--([xshift=1ex,yshift=-0.5ex]M-6-5.south east);
\draw[thick,black]([xshift=1ex,yshift=-0.5ex]M-6-5.south east)--([xshift=-3.4ex]M-8-6.south west);
\draw[thick,black]([xshift=-3.4ex]M-8-6.south west)--([yshift=-0.2ex]M-8-7.south west);
\draw[thick,black]([yshift=-0.2ex]M-8-7.south west)--([xshift=-3ex]M-10-7.south west);
\end{tikzpicture}
\]
which now has zeroes below the staircase pattern.
\end{rem}

We will make use of the isomorphism $\rhoStiefel$ from the preceding remark in the next \Cref{outer_case_idempotents} where we will need the following small lemma.
\begin{lem}\label{raising_tuple_preserves_length}
The isomorphism $\rhoStiefel \colon \LowStiefel_\cE \iso \RaiStiefel_\cE$ of \eqref{eq_low_to_rai_Stiefel} preserves the length of a tuple. That is, given a section $(\cV_1,\ldots,\cV_{\ell+1}) \in \LowStiefel_\cE(T)$ for some $T\in \Sch_S$, it has the property that $\cV_j\neq 0$ by definition for all $1\leq j \leq \ell+1$. Then, the section
\[
\rhoStiefel(\cV_1,\ldots,\cV_{\ell+1}) = (\cW_1,\ldots,\cW_{\ell+1})
\]
is also of length $\ell+1$ and $\cW_j\neq 0$ for all $1\leq j \leq \ell+1$.
\end{lem}
\begin{proof}
We consider a cover $\{U_i\to T\}_{i\in I}$ over which the lowered tuple $(\cV_1,\ldots,\cV_{\ell+1})$ restricts to tuples in $\ConStiefel_\cE$. Since $\cV_{\ell+1}\neq 0$ globally and the tuple is lowered, there exists at least one index $i_0 \in I$ for which
\[
(\cV_1|_{U_{i_0}},\ldots,\cV_{\ell+1}|_{U_{i_0}}) \in \ConStiefel_\cE(U_{i_0})
\]
is a length $\ell+1$ tuple of non-zero constant rank direct summands. Furthermore, for all $i\in I$, the restriction has length $\leq \ell+1$. Now, let
\[
\rhoStiefel(\cV_1,\ldots,\cV_{\ell+1}) = (\cW_1,\ldots,\cW_m)
\]
for some integer $m$. Since $\rhoStiefel$ acts as the identity on $\ConStiefel_\cE$, we also have that $(\cW_1,\ldots,\cW_m)|_{U_i}$ has length $\leq \ell+1$ for all $i\in I$ while
\[
(\cW_1,\ldots,\cW_m)|_{U_{i_0}} = (\cV_1|_{U_{i_0}},\ldots,\cV_{\ell+1}|_{U_{i_0}})
\]
has length exactly $\ell+1$. Thus, if $m>\ell+1$ this would imply that $\cW_m = 0$ which in turn implies that $\cW_1=0$ since this tuple is raised, however this contradicts the fact that $\cW_1|_{U_{i_0}} = \cV_1|_{U_{i_0}} \neq 0$. Thus, $m\leq \ell+1$ and since $\cW_{\ell+1} \neq 0$ because $\cW_{\ell+1}|_{U_{i_0}} = \cV_{\ell+1}|_{U_{i_0}} \neq 0$ we conclude that $m=\ell+1$. Lastly, we then likewise have for all $1\leq j \leq \ell+1$ that
\[
\cW_j|_{U_{i_0}} = \cV_j|_{U_{i_0}} \neq 0
\]
showing that $\cW_j\neq 0$ for all $j$. Thus, we are done.
\end{proof}

Now, as was our goal when defining lowered tuples, the construction which takes a lowered tuple of direct summands to a flag produces our desired map of sheaves.
\begin{lem}\label{Stiefel_to_Flag_surjection}
Let $\cE$ be a finite locally free $\cO$--module of constant rank $d$. We have a well-defined map of sheaves
\begin{align*}
\LowStiefel_\cE &\to \LowFlag_\cE \\
(\cV_1,\ldots,\cV_{\ell+1}) &\mapsto (0\subseteq \cV_1 \subseteq \cV_1\oplus\cV_2 \subseteq \ldots \subseteq \cV_1\oplus\ldots\oplus\cV_\ell \subseteq \cE).
\end{align*}
In particular, a lowered tuple produces a lowered flag and this construction is compatible with the restriction maps within the sheaves. Furthermore, this is a surjective maps of sheaves.
\end{lem}
\begin{proof}
Let $(\cV_1,\ldots,\cV_{\ell+1}) \in \LowStiefel_\cE$ be a lowered tuple and consider the associated flag
\[
0\subseteq \cV_1 \subseteq \cV_1\oplus\cV_2 \subseteq \ldots \subseteq \cV_1\oplus\ldots\oplus\cV_\ell \subseteq \cE.
\]
Let $T\in \Sch_S$ and assume that for some $1\leq i \leq \ell$
\[
(\cV_1\oplus\ldots\oplus\cV_i)|_T = (\cV_1\oplus\ldots\oplus\cV_i\oplus \cV_{i+1})|_T.
\]
In particular, this means that $\cV_{i+1}|_T \subseteq (\cV_1\oplus\ldots\oplus\cV_i)|_T$. However, the direct sum decomposition $\cE = \bigoplus_{j=1}^{\ell+1} \cV_j$ restricts as well, so we have $\cE|_T = \bigoplus_{j=1}^{\ell+1} \cV_j|_T$. Therefore, the above inclusion forces $\cV_{i+1}|_T = 0$, and then since the tuple is lowered, $\cV_j|_T = 0$ for all $j\geq i+1$ as well. This means that
\[
\cE|_T = \bigoplus_{j=1}^{\ell+1} \cV_j|_T = \bigoplus_{j=1}^i \cV_j|_T
\]
and so
\[
(\cV_1\oplus\ldots\oplus\cV_i)|_T = \cE|_T = (\cV_1\oplus\ldots\oplus\cV_i\oplus \cV_{i+1})|_T.
\]
Thus, the resulting flag is lowered as desired.

It is clear from construction that this map is compatible with the restriction maps of the two sheaves.

To see that this is a surjective map of sheaves, consider a flag
\[
0 \subseteq \cW_1 \subseteq \ldots \subseteq \cW_\ell \subseteq \cE
\]
in $\LowFlag_\cE$. By assumption, the $\cW_j$ are locally direct summands of $\cE$, and since surjectivity is a local property we may assume that they are already direct summands. Thus, we may find $\cV_{\ell+1}$ such that $\cE=\cW_\ell \oplus \cV_{\ell+1}$, and then we may find $\cV_\ell$ such that $\cW_\ell = \cW_{\ell-1}\oplus \cV_\ell$, and so on ending with $\cV_1 = \cW_1$. This produces a tuple $(\cV_1,\ldots,\cV_{\ell+1})$ with the property that
\[
\cV_1\oplus\ldots\oplus \cV_j = \cW_j
\]
for all $j$. We claim that this is a lowered tuple. Indeed, if for some $T\in \Sch_S$ and $j$ we have $\cV_j|_T=0$, this means that $\cW_{j-1}|_T = \cW_j|_T$ which implies that $\cW_{j-1}|_T = \cE|_T = \cW_j|_T$ since the flag is lowered. This also implies that $\cW_k|_T = \cE|_T$ for all $k\geq j$, and therefore $\cV_k|_T = 0$ for all such $k$. This shows that the tuple is lowered and so $(\cV_1,\ldots,\cV_{\ell+1}) \in \LowStiefel_\cE$ is a section which maps to $(0 \subseteq \cW_1 \subseteq \ldots \subseteq \cW_\ell \subseteq \cE) \in \LowFlag_\cE$. Thus, we have a surjective map of sheaves.
\end{proof}

Equivalently, a section in $\LowStiefel_\cE$ can be viewed as an ordered tuple of pairwise orthogonal idempotents in $\cEnd_\cO(\cE)$, namely the projections onto each of the direct summands. This allows us to generalize to non-neutral Azumaya algebras.
\begin{defn}\label{defn_IdempA}
Let $\cA$ be an Azumaya $\cO$--algebra of constant degree $d$. A tuple $(e_1,\ldots,e_{\ell+1})$ of sections of $\cA$ such that $e_i^2=e_i$ for all $i$, $e_ie_j=0$ for $i\neq j$, and $\sum_{i=1}^{\ell+1}e_i = 1$ is called a \emph{tuple of pairwise orthogonal idempotents}. Such a tuple is additionally called \emph{lowered} if for all $T \in \Sch_S$ and all $1\leq i \leq \ell$,
\[
e_i|_T = 0 \Rightarrow e_j|_T = 0 \text{ for all } j\geq i.
\]
We define the functor $\LowIdemp_\cA \colon \Sch_S \to \Sets$ to send $T\in \Sch_S$ to
\[
\left\{(e_1,\ldots,e_{\ell+1})\mid \begin{array}{c}\ell\geq 0,\; e_i \in \cA(T), e_{\ell+1}\neq 0, \text{this is a lowered} \\ \text{tuple of pairwise orthogonal idempotents}\end{array}\right\}.
\]
Restriction along a morphism $T' \to T$ in $\Sch_S$ is defined by
\[
(e_1,\ldots,e_{\ell+1}) \mapsto (e_1|_{T'},e_2|_{T'},\ldots,e_k|_{T'})
\]
where $k\in \{1,\ldots,\ell+1\}$ is the largest integer such that $e_k|_{T'}\neq 0$. I.e., since $(e_1,\ldots,e_{\ell+1})$ is lowered, it naively restricts to a length $\ell+1$ tuple of the form $(e_1|_{T'},\ldots,e_k|_{T'},0,\ldots,0)$ and we truncate the zeroes.
\end{defn}

\begin{lem}
The functor $\LowIdemp_\cA\colon \Sch_S \to \Sets$ is a sheaf.
\end{lem}
\begin{proof}
Ultimately, we mirror the proof of \Cref{lowered_flags_sheaf} here as well. In order to do so, we check that any tuple of pairwise orthogonal idempotents $(e_1,\ldots,e_{\ell+1})$ has $\ell < d$. However, this follows because such a tuple defines a direct sum decomposition
\[
\cA = e_1\cA \oplus \ldots e_{\ell+1}\cA
\] 
by right ideals which each have rank divisible by $d$. Therefore, for a cover over which the ideals $e_i\cA$ become constant rank there will be at least one scheme in the cover over which $\rank(e_i\cA)\geq d$ for all $i$, and therefore
\[
d^2 = \rank(\cA) = \sum_{i=1}^{\ell+1} \rank(e_i\cA) \geq d(\ell+1)
\]
and therefore $\ell < d$. Thus, as was done in \Cref{LowStiefel_sheaf} we may adapt the proof of \Cref{lowered_flags_sheaf} and see that $\LowIdemp_\cA$ is a sheaf.
\end{proof}

In the case when $\cA = \cEnd_\cO(\cE)$ is an endomorphism algebra, we can relate $\LowIdemp_\cA$ to $\LowStiefel_\cE$ in the following way (which we eventually show produces an isomorphism). For an idempotent $e\in \cEnd_\cO(\cE)$, we set $\cI_e = e\cdot \cEnd_\cO(\cE) \subseteq \cEnd_\cO(\cE)$, which is a right ideal which is also a direct summand (since $\cEnd_\cO(\cE) = e\cEnd_\cO(\cE) \oplus (1-e)\cEnd_\cO(\cE)$). 

\begin{lem}\label{LowIdemp_to_LowStiefel_iso}
Let $\cE$ be a finite locally free $\cO$--module of constant rank $d$. There is an isomorphism of sheaves
\begin{align*}
\LowIdemp_{\cEnd_\cO(\cE)} &\iso \LowStiefel_\cE \\
(e_1,\ldots,e_{\ell+1}) &\mapsto (\Img(e_1),\ldots,\Img(e_{\ell+1})) \\
(\pi_{\cV_1},\ldots,\pi_{\cV_{\ell+1}}) &\;\reflectbox{$\mapsto$}\; (\cV_1,\ldots,\cV_{\ell+1})
\end{align*}
where $\pi_{\cV_i}$ is the projection onto $\cV_i$ sending other $\cV_j$ to zero. Furthermore, this isomorphism fits into a diagram
\[
\begin{tikzcd}
\LowIdemp_{\cEnd_\cO(\cE)} \ar[r,"\sim"] \ar[d,two heads] & \LowStiefel_\cE \ar[d,two heads] \\
\LowGSB_{\cEnd_\cO(\cE)} \ar[r,"\sim"] & \LowFlag_\cE
\end{tikzcd}
\]
where the downward map on the right is the surjection of \Cref{Stiefel_to_Flag_surjection}, the bottom map is the isomorphism of \Cref{flags_to_SB}, and so the downward map on the left takes the form
\begin{align*}
\LowIdemp_{\cEnd_\cO(\cE)} &\to \LowGSB_{\cEnd_\cO(\cE)} \\
(e_1,\ldots,e_{\ell+1}) &\mapsto (0 \subseteq \cI_{e_1} \subseteq \ldots \subseteq \cI_{(e_1+\ldots+e_\ell)} \subseteq \cEnd_\cO(\cE)).
\end{align*}
\end{lem}
\begin{proof}
It is immediate that the composition
\[
(\cV_1,\ldots,\cV_{\ell+1}) \mapsto (\pi_{\cV_1},\ldots,\pi_{\cV_{\ell+1}}) \mapsto (\Img(\pi_{\cV_1}),\ldots,\Img(\pi_{\cV_{\ell+1}}))
\]
is the identity.

In the converse direction, let $\pi_i$ be the projection onto $\Img(e_i)$. Since $(e_1,\ldots,e_{\ell+1})$ is a pairwise orthogonal tuple of idempotents, this means that $\pi_i\circ e_j = 0$ if $i\neq j$. Additionally, $\pi_i\circ e_i = e_i$, and therefore
\[
\pi_i = \pi_i \circ \Id_\cE = \pi_i(e_1+\ldots+e_{\ell+1}) = \pi_i\circ e_i = e_i.
\]
Thus the composition
\[
(e_1,\ldots,e_{\ell+1}) \mapsto (\Img(e_1),\ldots,\Img(e_{\ell+1})) \mapsto (\pi_{\Img(e_1)},\ldots,\pi_{\Img(e_{\ell+1})})
\]
is also the identity. Hence, we have an isomorphism as claimed.

To see that the downward map on the left has the claimed form, we follow a tuple $(e_1,\ldots,e_{\ell+1}) \in \LowIdemp_{\cEnd_\cO(\cE)}$ through the diagram. First, going right, it gets sent to the section
\[
(\Img(e_1),\ldots,\Img(e_{\ell+1})) \in \LowStiefel_\cE.
\]
Then, because the $e_i$ are pairwise orthogonal, $\Img(e_i)\oplus \Img(e_j) = \Img(e_i+e_j)$ and so going down we arrive at the flag
\[
0 \subseteq \Img(e_1) \subseteq \Img(e_1+e_2) \subseteq \ldots \subseteq \Img(e_1+\ldots+e_\ell)\subseteq \cE.
\]
Finally, going left we obtain a flag of right ideals whose $i^\text{th}$ component is
\[
\cHom_\cO(\cE,\Img(e_1+\ldots+e_i)).
\]
We argue that this is $\cI_{(e_1+\ldots+e_i)}$. For any $\varphi \in \cEnd_\cO(\cE)$, the composition $(e_1+\ldots+e_i)\circ \varphi$ will have image contained in $\Img(e_1+\ldots+e_i)$, and therefore
\[
\cI_{(e_1+\ldots+e_i)}\subseteq \cHom_\cO(\cE,\Img(e_1+\ldots+e_i)).
\]
Conversely, since $e_1+\ldots+e_i$ is itself an idempotent, any section in $\Img(e_1+\ldots+e_i)$ is fixed by $e_1+\ldots+e_i$. Thus, if $\varphi \in \cHom_\cO(\cE,\Img(e_1+\ldots+e_i))$ then
\[
(e_1+\ldots+e_i)\circ \varphi = \varphi,
\]
showing that $\cHom_\cO(\cE,\Img(e_1+\ldots+e_i)) \subseteq \cI_{(e_1+\ldots+e_i)}$. Thus,
\[
\cHom_\cO(\cE,\Img(e_1+\ldots+e_i)) = \cI_{(e_1+\ldots+e_i)}
\]
and the downward map on the left of the diagram has the claimed form.
\end{proof}

\begin{lem}\label{Idemp_to_GSB_surj}
Let $\cA$ be an Azumaya $\cO$--algebra of constant degree $d$. We have a surjective map of sheaves
\begin{align*}
\LowIdemp_\cA &\to \LowGSB_\cA \\
(e_1,\ldots,e_{\ell+1}) &\mapsto (0\subseteq e_1\cA \subseteq (e_1+e_2)\cA \subseteq \ldots \subseteq (e_1+\ldots+e_\ell)\cA \subseteq \cA).
\end{align*}
\end{lem}
\begin{proof}
Given a lowered tuple of pairwise orthogonal idempotents in $\LowIdemp_\cA$, they induce a direct sum decomposition of $\cA$ corresponding to the section
\[
(e_1\cA,\ldots,e_{\ell+1}\cA) \in \LowStiefel_\cA
\]
and it is clear this defines a map of sheaves $\LowIdemp_\cA \to \LowStiefel_\cA$. We may then compose this with the map of \Cref{Stiefel_to_Flag_surjection} to obtain a map $\LowIdemp_\cA \to \LowFlag_\cA$ which sends $(e_1,\ldots,e_{\ell+1})$ to
\[
0 \subseteq e_1\cA \subseteq e_1\cA\oplus e_2\cA \subseteq \ldots \subseteq e_1\cA \oplus\ldots\oplus e_\ell\cA \subseteq \cA.
\]
Each component of this flag is a right ideal, so this lands in $\LowGSB_\cA$. Finally, since the idempotents are pairwise orthogonal we have that
\[
e_1\cA \oplus\ldots\oplus e_i\cA = (e_1+\ldots+e_i)\cA
\]
for all $i$. Therefore, our proposed map of sheaves is well-defined.

Locally, this becomes the map $\LowIdemp_{\cEnd_\cO(\cE)} \surj \LowGSB_{\cEnd_\cO(\cE)}$ of \Cref{LowIdemp_to_LowStiefel_iso} which is surjective, so the map is surjective globally as well.
\end{proof}

\begin{rem}\label{rhoIdemp}
There is of course also a sheaf of raised tuples of idempotents, $\RaiIdemp_\cA$, and a unique raising isomorphism
\[
\rhoIdemp \colon \LowIdemp_\cA \iso \RaiIdemp_\cA
\]
which appears similarly to $\rhoStiefel$ as described in \Cref{rhoStiefel}, just with idempotents in place of submodules. It also preserves the length of a tuple. In the case of an endomorphism algebra, we have a commutative diagram of isomorphisms
\[
\begin{tikzcd}
\LowIdemp_{\cEnd_\cO(\cE)} \ar[r,"\rhoIdemp"] \ar[d] & \RaiIdemp_{\cEnd_\cO(\cE)} \ar[d] \\
\LowStiefel_\cE \ar[r,"\rhoStiefel"] & \RaiStiefel_\cE
\end{tikzcd}
\]
where the downward map on the left is the isomorphism of \Cref{LowIdemp_to_LowStiefel_iso} and the downward map on the right is its analogue.
\end{rem}

\subsubsection{Parabolics and Levis}\label{inner_parabolics_and_levis}
Consider $\bG = \GL_{1,\cA}$ for an Azumaya $\cO$--algebra of constant degree $d=n+1$. Given a section $\ve=(e_1,\ldots,e_{\ell+1}) \in \LowIdemp_\cA(S)$, we will associate it to a parabolic and Levi pair by first using it to define a cocharacter of $\GL_{1,\cA}$ and then using the limit subgroup construction of \Cref{limit_subgroups}. As an initial remark, we point out that the construction in this section for the inner case works equally well using the straightforward cocharacter
\begin{align*}
\lambda_{\ve} \colon \GG_m &\to \GL_{1,\cA} \\
t &\mapsto \sum_{i=1}^{\ell+1} t^{-i}e_i.
\end{align*}
However, a more complicated construction will be required in \Cref{Outer_parabolics_and_levis} and so for compatibility we will use that construction here as well.

Recall, for example from \cite[Expose XXII, 1.11]{SGA3} and the fact that $\GG_m = D_S(\ZZ)$ in their notation, that as a sheaf on $\Sch_S$, there character group of $\GG_m$
\[
\cHom_{\textrm{Grp}}(\GG_m,\GG_m) = \underline{\ZZ}
\]
is the sheaf of locally constant integers. In particular, for a scheme $T \in \Sch_S$ and a section $\underline{a} \in \underline{\ZZ}(T)$, we write
\begin{align*}
\GG_m|_T &\to \GG_m|_T \\
t &\mapsto t^{\underline{a}}
\end{align*}
for the associated character which is defined by the property that over a cover $\{U_i \to T\}_{i\in I}$ where $\underline{a}|_U=n_i$ is constant, it behaves as $t\mapsto t^{n_i}$ as usual. Thus, for a section $x\in \GG_m(T)$ we have that $x^{\underline{a}}$ is the gluing of the local elements $(x|_{U_i})^{n_i}$. Rephrasing this, these characters have the property that $(x^{\underline{a}})|_U = (x|_U)^{\underline{a}|_U}$ for all $U \in \Sch_T$. 

Likewise, given a section $\ve \in \LowIdemp_\cA(T)$, the length of the tuple is a locally constant positive integer by construction and therefore we have a section $\underline{a}(\ve) \in \underline{\ZZ}(T)$ encoding the length of $\ve$. We use this to define the cocharacter
\begin{align}
\lambda_{\ve} \colon \GG_m &\to \GL_{1,\cA} \label{inner_cocharacter} \\
t &\mapsto \sum_{i=1}^{\ell+1} t^{\underline{a}(\ve)+1-2i}e_i. \nonumber
\end{align}
where the terms $1$ and $2i$ in $\underline{a}(\ve)+1-2i$ are constant sections of $\underline{\ZZ}$. We define our cocharacters this way in order to have the following lemma necessary later on.
\begin{lem}\label{cocharacter_raised_formula}
Let $\cA$ be an Azumaya $\cO$--algebra of constant degree $d$ (or more generally a separable algebra) and consider a section $\ve = (e_1,\ldots,e_{\ell+1}) \in \LowIdemp_\cA$. Let $\rhoIdemp(\ve) = (e_1',\ldots,e_{\ell+1}')$ be the associated raised section. Then, we have an alternative expression for the cocharacter $\lambda_{\ve}$, namely
\[
\sum_{i=1}^{\ell+1} t^{\underline{a}(\ve)+1-2i}e_i = \sum_{i=1}^{\ell+1} t^{-\underline{a}(\ve)-1+2i}e'_{\ell+2-i}.
\] 
\end{lem}
\begin{proof}
Let $\ve = (e_1,\ldots,e_{\ell+1}) \in \LowIdemp_\cA(T)$. It is sufficient to show the equality holds locally over a scheme $U \in \Sch_T$ where the length section $\underline{a}(\ve)$ is constant since we may choose a cover consisting of such schemes. Say $\underline{a}(\ve)|_U = k$. In particular,
\[
(e_1,\ldots,e_{\ell+1})|_U = (e_1|_U,\ldots,e_k|_U) = (e_1',\ldots,e_{\ell+1}')|_U
\]
which means that $(e_{\ell+2-i}')|_U = e_{k+1-i}|_U$ since the raised tuple truncates on the left under restrictions. Now, we may simply compare the expressions
\begin{align*}
&\left.\left(\sum_{i=1}^{\ell+1} t^{\underline{a}(\ve)+1-2i}e_i\right)\right|_U \\
=& t|_U^{k-1}e_1|_U + t|_U^{k-3}e_2|_U + \ldots + t|_U^{-k+1}e_k|_U + t|_U^{-k-1}0 + \ldots + t|_U^{k+1-2(\ell+1)}0
\end{align*}
and
\begin{align*}
&\left.\left(\sum_{i=1}^{\ell+1} t^{-\underline{a}(\ve)-1+2i}e'_{\ell+2-i}\right)\right|_U \\
=& t|_U^{-k+1}e_k|_U + t|_U^{-k+3}e_{k-1}|_U + \ldots + t|_U^{k-1}e_1|_U + t|_U^{k+1}0 + \ldots + t|_U^{-k-1+2(\ell+1)}0
\end{align*}
and see that they are equal.
\end{proof}

We now assign a parabolic and Levi pair $(\bP(\lambda_{\ve}),\bL(\lambda_{\ve}))$ to $\ve$ as reviewed in \Cref{limit_subgroups}.

\begin{thm}\label{Idemp_PL_iso}
Let $\bG = \GL_{1,\cA}$ for an Azumaya $\cO$--algebra of constant degree $d=n+1$. There is an isomorphism of sheaves
\begin{align*}
\LowIdemp_\cA &\iso \cPL_\bG \\
\ve &\mapsto (\bP(\lambda_{\ve}),\bL(\lambda_{\ve}))
\end{align*}
\end{thm}
\begin{proof}
It suffices to show that the map is an isomorphism when $\cA = \Mat_d(\cO)$ and so $\bG = \GL_n$ since the general map with be locally of this form, and a map of sheaves which is locally an isomorphism is an isomorphism globally.

Assuming that $\cA = \Mat_d(\cO)$, it is then sufficient to show that the standard parabolic and Levi pairs, i.e., those where the parabolic subgroup contains the Borel subgroup of invertible upper triangular matrices and the Levi subgroup consists of block diagonal matrices, have a unique preimage in $\LowIdemp_{\Mat_d(\cO)}$, which we will denote by $\LowIdemp_d$. This is sufficient since any other pair $(\bP,\bL)$ is conjugate to a pair of the above form. Say that $(g\bP g^{-1},g\bL g^{-1})$ is a standard pair and so there is a unique $\overline{e} = (e_1,\ldots,e_{\ell+1})$ such that
\[
(\bP(\lambda_{\overline{e}}),\bL(\lambda_{\overline{e}}))=(g\bP g^{-1},g\bL g^{-1}).
\]
By \Cref{conjugating_limits_Levis}, this means that
\[
(\bP,\bL) = (\bP(g^{-1}\lambda_{\overline{e}}g),\bL(g^{-1}\lambda_{\overline{e}}g)).
\]
It is straightforward using the definition of the cocharacter to see that
\[
g^{-1}\lambda_{\overline{e}}g = \lambda_{g^{-1}\overline{e}g}
\]
where $g^{-1}\overline{e}g = (g^{-1}e_1 g,\ldots,g^{-1}e_{\ell+1}g)\in \LowIdemp_d$. Of course, if there was another $\overline{e'} \in \cIdemp_d$ also producing $(\bP,\bL)$, then by using \Cref{conjugating_limits_Levis} again we would have that $g\overline{e'}g^{-1}$ produces the same standard pair as $\overline{e}$, and so by uniqueness $g\overline{e'}g^{-1} = \overline{e}$, meaning that $g^{-1}\overline{e}g$ is the unique section which produces $(\bP,\bL)$.

Furthermore, it is sufficient to assume that $(\bP,\bL)$ is a standard pair of constant type, i.e., that $\bP$ has constant type under the type morphism of \eqref{type_map_parabolics}, since we may otherwise localize until it does have constant type. Thus, we may assume that
\[
\bL \cong \GL_{n_1}\times \GL_{n_2}\times \ldots \GL_{n_p}
\]
sitting inside $\GL_d$ as
\[
\bL = \left\{ \left[ \begin{array}{c:c:c:c} B_1 & 0 & 0 & 0 \\ \hdashline 0 & B_2 & 0 & 0 \\ \hdashline 0 & 0 & \ddots & 0 \\ \hdashline 0 & 0 & 0 & B_p  \end{array}\right] \mid B_i \in \GL_{n_i} \right\}
\]
and then $\bP$ sits inside $\GL_d$ as
\[
\bP = \left\{ \left[ \begin{array}{c:c:c:c} B_1 & * & * & * \\ \hdashline 0 & B_2 & * & * \\ \hdashline 0 & 0 & \ddots & * \\ \hdashline 0 & 0 & 0 & B_p  \end{array}\right] \mid B_i \in \GL_{n_i} \right\}
\]
Let $e_i$ be idempotents corresponding to the block decomposition of $\bL$, i.e., let
\[
e_1 = \left[ \begin{array}{c:c:c:c} I_{n_1} & 0 & 0 & 0 \\ \hdashline 0 & 0 & 0 & 0 \\ \hdashline 0 & 0 & \ddots & 0 \\ \hdashline 0 & 0 & 0 & 0  \end{array}\right], e_2 = \left[ \begin{array}{c:c:c:c} 0 & 0 & 0 & 0 \\ \hdashline 0 & I_{n_2} & 0 & 0 \\ \hdashline 0 & 0 & \ddots & 0 \\ \hdashline 0 & 0 & 0 & 0  \end{array}\right], \ldots, e_p = \left[ \begin{array}{c:c:c:c} 0 & 0 & 0 & 0 \\ \hdashline 0 & 0 & 0 & 0 \\ \hdashline 0 & 0 & \ddots & 0 \\ \hdashline 0 & 0 & 0 & I_{n_p}  \end{array}\right].
\]
Let $\overline{e} = (e_1,\ldots,e_p)$. This tuple of idempotents is lowered since $e_i|_T\neq 0$ for any $T \in \Sch_S$, thus $\ve \in \LowIdemp_d(S)$. Additionally, since $\ve$ is constant length, the exponents of $t$ appearing in $\lambda(\ve)$ are all constant sections. A quick calculation then shows that $\bL = \bL(\lambda_{\overline{e}})$. However, this is not the unique cocharacter whose centralizer is $\bL$. Any cocharacter $\lambda \colon \GG_m \to \GL_d$ whose centralizer is $\bL$ must map into the center of $\bL$, which is isomorphic to $\GG_m^p$. Thus, it is of the form
\[
\lambda(t) = \sum_{i=1}^p t^{\underline{a_i}}e_i
\]
for powers $\underline{a_i} \in \underline{\ZZ}$. In order to have $\bCent_{\GL_d}(\lambda) = \bL$ and not a larger subgroup, all $\underline{a_i}$ must be locally distinct, and in fact any such cocharacter will have centralizer $\bL$. However, not all of these cocharacters come from a section in $\LowIdemp_d$, only those with $\{\underline{a_1},\ldots,\underline{a_p}\} = \{p-1,p-3,\ldots,-p+3,-p+1\}$ and these come from the sections in $\LowIdemp_d(S)$ which are permutations of the ordered tuple $(e_1,\ldots,e_p)$. Thus, we must argue that from among these cocharacters, only $\lambda_{\overline{e}}$ has $\bP$ as its limit parabolic subgroup.

Consider such a cocharacter $\lambda(t) = \sum_{i=1}^p t^{a_i}e_i$ with $\{a_1,\ldots,a_p\} = \{p-1,p-3,\ldots,-p+3,-p+1\}$. Write a matrix $B\in \GL_d$ as a $p\times p$ block matrix
\[
B = \left[ \begin{array}{c:c:c:c} B_{11} & B_{12} & \hdots & B_{1p} \\ \hdashline B_{21} & B_{22} & \hdots & B_{2p} \\ \hdashline \vdots & \vdots & \ddots & \vdots \\ \hdashline B_{p1} & B_{p2} & \hdots & B_{pp}  \end{array}\right]
\]
with square blocks on the diagonal of size according to the block decomposition of $\bL$ and thus possibly rectangular off-diagonal blocks. Performing a calculation analogous to the one in \Cref{example_limit_computation}, we see that
\[
\lambda(t)B\lambda(t^{-1}) = \left[ \begin{array}{c:c:c:c} B_{11} & t^{a_1-a_2}B_{12} & \hdots & t^{a_1-a_p} B_{1p} \\ \hdashline t^{a_2-a_1} B_{21} & B_{22} & \hdots & t^{a_2-a_p}B_{2p} \\ \hdashline \vdots & \vdots & \ddots & \vdots \\ \hdashline t^{a_p-a_1}B_{p1} & t^{a_p-a_2}B_{p2} & \hdots & B_{pp}  \end{array}\right] = [t^{a_i-a_j}B_{ij}]_{i,j=1}^p.
\]
The resulting limit parabolic subgroup will agree with $\bP$ if and only if all the powers of $t$ appearing in the lower diagonal blocks are negative (and so those appearing in the upper diagonal blocks will be positive). However, this occurs exactly when the tuple $(a_1,a_2,\ldots,a_p)$ is decreasing, meaning that $a_i = p+1-2i$ and so
\[
\lambda = \lambda_{\overline{e}}.
\]
Therefore, $\overline{e} \in \LowIdemp_\cA(S)$ is the unique section such that
\[
(\bP(\lambda_{\overline{e}}),\bL(\lambda_{\overline{e}})) = (\bP,\bL),
\]
which finishes the proof.
\end{proof}

\begin{rem}
\Cref{limit_subgroups} describes how any cocharacter $\lambda \colon \GG_m \to \bG$ produces a parabolic-Levi pair $(\bP(\lambda),\bL(\lambda))$. Conversely, \Cref{Idemp_PL_iso} says that every parabolic-Levi pair arises uniquely from this construction when one considers cocharacters associated in a particular way to lowered tuples of pairwise orthogonal idempotents. However, in the case when the base scheme $S$ is connected, the fact that every pair $(\bP,\bL)$ arises from at least some cocharacter was already included in \cite[7.3.1(1)]{Gille}.
\end{rem}

\begin{prop}\label{inner_PL_iso_diagram}
Let $\bG = \GL_{1,\cA}$ for an Azumaya $\cO$--algebra of constant degree $d$. The isomorphism of \Cref{Idemp_PL_iso} fits into the commutative diagram
\[
\begin{tikzcd}
\LowIdemp_\cA \ar[r,"\sim"] \ar[d,two heads] & \cPL_\bG \ar[d,two heads] \\
\LowGSB_\cA \ar[r,"\sim"] & \cPar_\bG
\end{tikzcd}
\]
where the downward maps are those of \Cref{Idemp_to_GSB_surj} and \eqref{PL_to_Par_surj} respectively, and the bottom isomorphism is the map of \Cref{SB_isomorphism_with_Par}.
\end{prop}
\begin{proof}
Let $T\in \Sch_S$ and consider a section $\overline{e}=(e_1,\ldots,e_{\ell+1}) \in \LowIdemp_\cA(T)$. Let $\cI_j = (e_1+\ldots + e_j)\cdot \cA|_T$ so that
\[
\cI_{\overline{e}} = (0 \subseteq \cI_1 \subseteq \ldots \subseteq \cI_\ell \subseteq \cA|_T)
\]
is the image of $\overline{e}$ in $\LowGSB_\cA(T)$. Since $\overline{e}$ maps to $(\bP(\lambda_{\overline{e}}),\bL(\lambda_{\overline{e}}))$ in $\cPL_\bG(T)$, to show that the diagram commutes we need to verify that
\[
\bStab(\cI_{\overline{e}}) = \bP(\lambda_{\overline{e}}).
\] 
Recall from \Cref{prelim_parabolics_and_levis} that for $T'\in \Sch_T$, the group $\bP(\lambda_{\overline{e}})(T')$ is
\[
\left\{ h \in \GL_{1,\cA}(T') \mid \lambda_{\overline{e}}(t)h|_{T_m'}\lambda_{\overline{e}}(t^{-1}) \in \Img(\GL_{1,\cA}(T'_a) \xrightarrow{\textrm{res}} \GL_{1,\cA}(T'_m))\right\}.
\]
where $T'_a \to T'$ and $T'_m \to T'$ are the relatively affine schemes corresponding to the sheaves of commutative $\cO|_{T'}$--algebras $\cO|_{T'}[t]$ and $\cO|_{T'}[t,t^{-1}]$, respectively. Thus, as a first step we will identify $\GL_{1,\cA}(T'_a)$ and $\GL_{1,\cA}(T'_m)$. By definition,
\[
\GL_{1,\cA}(T'_a) = \cA(T'_a)^\times,
\]
and so we must identify $\cA(T'_a)$. Since the morphism $T'_a \to T'$ is affine, for any Zariski cover $\{U_i \to T'\}_{i\in I}$ by affine open subschemes, the pulled back cover $\{T'_a\times_{T'} U_i \to T'_a\}_{i\in I}$ is also a Zariski cover by affine open subschemes. Then, since $\cA$ is quasi-coherent, this means that
\begin{align*}
\cA(T'_a\times_{T'} U_i)&\cong \cA(U_i)\otimes_{\cO(U_i)} \cO(T'_a\times_{T'} U_i) \\
&\cong \cA(U_i)\otimes_{\cO(U_i)} \cO(U_i)[t] \\
&\cong \cA(U_i)[t]
\end{align*}
where the variable $t$ commutes with all elements of $\cA(U_i)$. From here, an easy analysis of the sheaf exact sequence (where one can view $\cA(U_i)[t] \cong \oplus_{i=0}^{\infty} \cA(U_i)$ as a module) shows that
\[
\cA(T'_a) \cong \cA(T')[t].
\]
Likewise, we obtain that
\[
\cA(T'_m) \cong \cA(T')[t,t^{-1}].
\]
Therefore, $\GL_{1,\cA}(T'_a)$ and $\GL_{1,\cA}(T'_m)$ are simply the invertible elements in the the above two algebras, respectively. Additionally,
\[
\lambda_{\overline{e}}(t) = \sum_{i=1}^{\ell+1}t^{-i}e_i \in \GL_{1,\cA}(T'_m).
\]

Now, we will show that  $\bP(\lambda_{\overline{e}}) \subseteq \bStab(\cI_{\overline{e}})$. For a section $h\in \bP(\lambda_{\overline{e}})(T') \subseteq \cA(T')^\times$, we have that
\begin{align*}
\lambda_{\overline{e}}(t)h\lambda_{\overline{e}}(t^{-1}) &= \left(\sum_{i=1}^{\ell+1}t^{-i}e_i\right)h\left(\sum_{i=1}^{\ell+1}t^i e_i\right) \\
&= \sum_{i,j=1}^{\ell+1} t^{-i+j}e_ihe_j \\
&= \sum_{n=-\ell}^\ell t^n \left(\sum_{-i+j=n} e_ihe_j\right)
\end{align*}
must only involve non-negative powers of $t$, and thus
\[
\sum_{-i+j=n} e_ihe_j = 0
\]
for all $-\ell \leq n \leq -1$. Equivalently, by flipping some negative signs and changing indices,
\[
\sum_{i=n+1}^{\ell+1} e_i h e_{i-n} = 0
\]
for all $1\leq n \leq \ell$. Now, for any pair of indices $p>q$, we may take the above expression with $n=p-q$ and multiply on the left and right to obtain
\[
e_p h e_q = e_p \left(\sum_{i=p-q+1}^{\ell+1} e_i h e_{i-p+q} \right) e_q = 0.
\]
In particular, one sees that for any $1\leq k \leq \ell+1$,
\[
(e_{k+1}+\ldots+e_{\ell+1})h(e_1+\ldots+e_k) =0.
\]
Writing $f_k = e_1+\ldots+e_k$, which is also an idempotent, this is $(1-f_k)hf_k = 0$, which can equivalently be written as
\[
hf_k = f_khf_k.
\]
We now claim that the element $h$ stabilizes the right ideal $f_k\cA|_{T'}$. Indeed, any section of $hf_k\cA|_{T'}h^{-1}$ is of the form
\[
(h f_k) a h^{-1} = (f_k h f_k)a h^{-1} = f_k(hf_kah^{-1}) \in f_k\cA|_{T'}.
\]
So $hf_k\cA|_{T'}h^{-1} \subseteq f_k\cA|_{T'}$. The same reasoning applies to $h^{-1} \in \bP(\lambda_{\overline{e}})(T')$ yielding the reverse inclusion, and thus
\[
hf_k\cA|_{T'}h^{-1} = f_k\cA|_{T'}.
\]
This holds for all $1\leq k \leq \ell+1$, showing that in fact
\[
h \in \bStab(\cI_{\overline{e}})(T')
\]
and so we have that $\bP(\lambda_{\overline{e}}) \subseteq \bStab(\cI_{\overline{e}})$ as claimed.

Finally, we argue that $\bStab(\cI_{\overline{e}}) \subseteq \bP(\lambda_{\overline{e}})$. Consider a section $a\in \bStab(\cI_{\overline{e}})(T')$. Writing $f_k = e_1+\ldots e_k$ as above, we have that
\[
a f_k \cA|_{T'} a^{-1} = f_k \cA|_{T'},
\]
which in particular implies that $af_k a^{-1} = f_k b$ for some section $b\in \cA(T')$, or equivalently that $af_k = f_k b a$. However, this means that
\[
f_k a f_k = f_k(f_k b a) = f_k b a = a f_k
\]
since $f_k$ is idempotent. Next, we essentially reverse the calculations from the previous case. For any indices $p>q$ We have that $(1-f_q)a f_q = 0$, and so also $e_p(1-f_q)a f_q e_q = 0$ which expands simply into $e_p a e_q = 0$. Since we have this for all $p>q$, it is immediate that
\[
\lambda_{\overline{e}}(t)a\lambda_{\overline{e}}(t^{-1}) \in \cA(T')[t],
\]
i.e., it does not involve any negative powers of $t$. Likewise, by the same reasoning the same holds for
\[
\lambda_{\overline{e}}(t)a^{-1}\lambda_{\overline{e}}(t^{-1})
\]
which is an inverse for the expression with $a$, and thus
\[
\lambda_{\overline{e}}(t)a\lambda_{\overline{e}}(t^{-1}) \in \GL_{1,\cA}(T'_a)
\]
meaning that $a\in \bP(\lambda_{\overline{e}})(T')$. This shows that we also have $\bStab(\cI_{\overline{e}}) \subseteq \bP(\lambda_{\overline{e}})$ and so
\[
\bStab(\cI_{\overline{e}}) = \bP(\lambda_{\overline{e}}).
\]
Thus, the diagram commutes and the proof is finished.
\end{proof}

\subsection{Outer Case}\label{outer_case_idempotents}
We now consider the general case of $\cPar_\bG$ when $\bG \cong \bU(\cB,\tau)$ for an Azumaya algebra with involution of the second kind. Here, we will need the notion of the \emph{$\cL$--subrank} of an $\cL$--submodule. This will be a locally constant section, i.e., a section of the sheaf $\underline{\{0,\ldots,d\}}$ as in \Cref{Locally_Constant_Sheaves}, which we define through sheafification. In our application, it will be closely related to the $\cL$--gap of \Cref{defn_L_gap}.

\begin{defn}
Let $L\to S$ be a finite \'etale cover of schemes of degree $k$ with corresponding commutative $\cO$--algebra $\cL\colon \Sch_S \to \Rings$. Let $\cM \colon \Sch_S \to \Ab$ be a finite locally free $\cL$--module of constant rank $d$. We define the sheaf of $\cL$--submodules which are locally direct summands,
\begin{align*}
\SubMod_{\cL,\cM} \colon \Sch_S &\to \Sets \\
T &\mapsto \left\{\cV \mid \begin{array}{l} \cV\subseteq \cM|_T \text{ is an }\cL|_T\text{--submodule, and} \\ \text{it is locally a direct summand}\end{array}\right\}.
\end{align*}
Over any $T\in \Sch_S$ where $\cL|_T \cong (\cO|_T)^k$, which happens locally, we will have that $\cM|_T \cong \cM_{1,T}\times\ldots\times \cM_{k,T}$ for finite locally free $\cO|_T$--modules $\cM_{i,T}$ all of constant rank $d$. Based on this, we define a presheaf
\[
\SplitSubMod_{\cL,\cM} \colon \Sch_S \to \Sets
\]
of split constant rank $\cL$--submodules which are direct summands which behaves on $T\in \Sch_S$ as
\[
T \mapsto \begin{cases} \O & \text{if }\cL|_T \not\cong (\cO|_T)^k \\ \left\{ \cV_1\times\ldots\times\cV_k \mid \begin{array}{l} \cV_i\subseteq \cM_{i,T} \text{ is a direct summand,} \\ \text{and it has constant } \cO|_T\text{--rank} \end{array}\right\} & \text{if }\cL|_T \cong (\cO|_T)^k \end{cases}
\]
with restriction maps coming from submodule restriction within $\cM$ (and so they may involve permutations of the factors depending on the isomorphisms $\cL|_T\cong (\cO|_T)^k$ chosen).
\end{defn}
\begin{lem}
The canonical injection $\SplitSubMod_{\cL,\cM} \inj \SubMod_{\cL,\cM}$ has the universal property of sheafification and thus shows that
\[
(\SplitSubMod_{\cL,\cM})^\sharp \cong \SubMod_{\cL,\cM}.
\]
\end{lem}
\begin{proof}
As argued in the proof of \Cref{Flag2_sheafification}, any $\cL$--submodule $\cV \subseteq \cM$ locally takes the form
\[
\cV_1\times\ldots\cV_k \subseteq \cM_{1,T}\times\ldots\times\cM_{k,T} = \cM|_T
\]
over any $T\in \Sch_S$ where $\cL|_T \cong (\cO|_T)^k$. If $\cV$ is also locally a direct summand, then the assumptions on $\cM$ mean that $\cV$ locally belongs to $\SplitSubMod_{\cL,\cM}$. Thus, any section of $\SubMod_{\cL,\cM}$ locally belongs to $\SplitSubMod_{\cL,\cM}$, which justifies the claim.
\end{proof}

\begin{defn}\label{defn_subrank}
Let $L\to S$ be an \'etale cover of schemes of degree $k$ with corresponding commutative $\cO$--algebra $\cL\colon \Sch_S \to \Rings$. Let $\cM \colon \Sch_S \to \Ab$ be a finite locally free $\cL$--module of constant rank $d$. We define a map of presheaves
\[
\SplitSubMod_{\cL,\cM} \to \underline{\{0,\ldots,d\}}
\]
requiring that for each $T\in \Sch_S$ for which $\cL|_T\cong (\cO|_T)^k$, and so $\SplitSubMod_{\cL,\cM}(T)$ is nonempty, it acts as
\[
\cV_1\times\ldots\times\cV_k \mapsto \min(\{\rank_{\cO|_T}(\cV_1),\ldots,\rank_{\cO|_T}(\cV_k)\})
\]
and of course it is simply the inclusion $\O \inj \cN(T)$ when $\SplitSubMod_{\cL,\cM}(T)= \O$.

Since $\underline{\{0,\ldots,d\}}$ is a sheaf, the above map induces a unique map of sheaves
\[
\subrank_\cL \colon \SubMod_{\cL,\cM} \to \underline{\{0,\ldots,d\}}
\]
which we call the \emph{$\cL$--subrank}.
\end{defn}
The $\cL$--subrank is well-defined since the presheaf map is invariant under permutations of the factors and therefore compatible with restrictions. Intuitively, the $\cL$--subrank is the locally constant section which encodes the $\cL$--rank of the largest free $\cL$--module contained in $\cV$. We must use $\cL$--subrank, instead of simply $\cL$--rank, since we must consider $\cL$--submodules of $f_*(\cH)$ which are locally direct summands that are finite locally free as $\cO$--submodules but not necessarily finite locally free as $\cL$--modules. This is to include submodules of the form $\cV_1\times\cV_2 \subseteq \cO^d\times\cO^d$ where $\rank_\cO(\cV_1)\neq \rank_\cO(\cV_2)$.

\subsubsection{Orthogonal Summands of Hermitian Forms}
Let $(f\colon L \to S, \cH,h)$ be a regular Hermitian form of constant rank $d$. We wish to consider certain direct sum decompositions of $f_*(\cH)$ whose partial sums will produce a flag belonging to $\LowFlag_{(\cH,h)}$. For this purpose, we define the following.

\begin{defn}
Let $(f\colon L \to S, \cH,h)$ be a regular Hermitian form of constant rank $d$ and set $\cL = f_*(\cO|_L)$. For $T\in \Sch_S$ we say that a tuple
\[
(\cV_1,\ldots,\cV_{\ell+1}) \in \LowStiefel_{f*(\cH)}(T)
\]
is \emph{$\cL|_T$--lowered} if all $\cV_j$ are $\cL$--submodules and for all $T'\in \Sch_T$ we have that
\[
\subrank_{\cL|_{T'}}(\cV_j|_{T'})=0 \Rightarrow \cV_k|_{T'}=0 \text{ for all } k\geq j.
\]
Additionally, recall the length preserving isomorphism
\[
\rhoStiefel \colon \LowStiefel_{f_*(\cH)} \iso \RaiStiefel_{f_*(\cH)}
\]
of \eqref{eq_low_to_rai_Stiefel}. For $T\in \Sch_S$ we define the subset $\LowStiefel_{(\cH,h)}(T) \subseteq \LowStiefel_{f_*(\cH)}(T)$ to be
\[
\left\{(\cV_1,\ldots,\cV_{\ell+1}) \in \LowStiefel_{f_*(\cH)}(T) \mid \begin{array}{l} \text{the tuple is }\cL\text{--lowered, and} \\ \cV_i = \left(\oplus_{j\neq \ell+2-i}\cV'_j\right)^{\perp_h} \\ \text{where } \rhoStiefel(\cV_1,\ldots,\cV_{\ell+1}) = (\cV'_1,\ldots,\cV'_{\ell+1}) \end{array} \right\}.
\]
\end{defn}
Intuitively, for a tuple $(\cV_1,\ldots,\cV_{\ell+1})$ of constant rank submodules, the condition that $\cV_i = \left(\oplus_{j\neq \ell+2-i}\cV_j\right)^{\perp_h}$ will ensure that the resulting flag of partial sums has the appropriate symmetry property. However, since the length of a tuple may change upon restriction, this naive condition fails to produce a sheaf and therefore we must use the condition with $\rhoStiefel$ in order to obtain a sheaf.
\begin{lem}\label{outer_Stiefel_is_a_sheaf}
Setting $\LowStiefel_{(\cH,h)} \colon \Sch_S \to \Sets$ to send $T\in \Sch_S$ to $\LowStiefel_{(\cH,h)}(T)$ defined above, produces a subsheaf of $\LowStiefel_{f_*(\cH)}$.
\end{lem}
\begin{proof}
First we show that $\LowStiefel_{(\cH,h)}$ is a subpresheaf of $\LowStiefel_{f_*(\cH)}$ by showing that the imposed conditions are stable with respect to restrictions. First, being $\cL$--lowered is stable under restrictions by definition. Thus, assume that we have a section
\[
(\cV_1,\ldots,\cV_{\ell+1}) \in \LowStiefel_{(\cH,h)}(T)
\]
for some $T\in \Sch_S$ and that we have a map of schemes $U\to T$. The restriction of the tuple within $\LowStiefel_{f_*(\cH)}$ is
\[
(\cV_1|_U,\ldots,\cV_k|_U)
\]
for some $1\leq k \leq \ell+1$. Now, $\rhoStiefel$ is an isomorphism of sheaves and thus also compatible with restrictions, however we must keep in mind that restrictions within $\RaiStiefel_{f_*(\cH)}$ involve truncation on the left. Combining this with the fact that $\rhoStiefel$ is length preserving by \Cref{raising_tuple_preserves_length}, then if
\[
\rhoStiefel(\cV_1,\ldots,\cV_{\ell+1}) = (\cV'_1,\ldots,\cV'_{\ell+1})
\]
then
\[
\rhoStiefel(\cV_1|_U,\ldots,\cV_k|_U) = (\cV'_{\ell+2-k}|_U,\ldots,\cV'_{\ell+1}|_U).
\]
Alternatively, we can write that $(\cV_j|_U)' = (\cV'_{\ell+1-k+j})|_U$. Thus, to be a member of the set $\LowStiefel_{(\cH,h)}(U)$ the restricted section must satisfy the conditions that
\begin{align*}
\cV_i|_U &= \left(\bigoplus_{j=1\atop j\neq k+1-i}^k (\cV_j|_U)'\right)^{\perp_h} = \left(\bigoplus_{j=1\atop j\neq k+1-i}^k (\cV'_{\ell+1-k+j})|_U\right)^{\perp_h} \\
&= \left(\bigoplus_{j=\ell+2-k \atop j\neq \ell+2-i}^{\ell+1} (\cV'_j)|_U\right)^{\perp_h} = \left.\left(\bigoplus_{j\neq \ell+2-i} \cV'_j\right)^{\perp_h}\right|_U
\end{align*}
which is true since it is the restriction of the global condition. Thus, we have that $(\cV_1|_U,\ldots,\cV_k|_U) \in \LowStiefel_{(\cH,h)}(U)$ and so $\LowStiefel_{(\cH,h)}$ is a subpresheaf of $\LowStiefel_{f_*(\cH)}$.

Now, to see that $\LowStiefel_{(\cH,h)}$ is a sheaf it only remains to show that it satisfies gluing, i.e., that for $T\in \Sch_S$, a section of $\LowStiefel_{f_*(\cH)}(T)$ which locally belongs to $\LowStiefel_{(\cH,h)}$ also belongs to $\LowStiefel_{(\cH,h)}(T)$. Consider such a section
\[
(\cV_1,\ldots,\cV_{\ell+1}) \in \LowStiefel_{f_*(\cH)}(T)
\]
and a cover $\{U_i \to T\}_{i\in I}$ over which the section belongs to $\LowStiefel_{(\cH,h)}$. First, this means that for all $\cV_j$, the restrictions $\cV_j|_{U_i}$ are $\cL|_{U_i}$--modules. Therefore, each $\cV_j$ is itself an $\cL$--module. It remains to show that the tuple is $\cL|_T$--lowered. Let $T' \in \Sch_T$ and assume that for some index $j$ we have that $\subrank_{\cL|_{T'}}(\cV_j|_{T'})=0$. By setting $U_i' = T'\times_T U_i$ we obtain the cover $\{U_i' \to T'\}_{i\in I}$ of $T'$. Then, for each $i$ we have $\subrank_{\cL|_{U_i'}}(\cV_j|_{U_i'})=0$ which then implies that
\[
\cV_k|_{U_i'} = 0
\]
for all $i\in I$ and all $k \geq j$ by the assumption that the section belongs to $\LowStiefel_{(\cH,h)}$ over each $U_i$. However, together these imply that $\cV_k|_{T'}=0$ over $T'$ as well for each $k\geq j$, so we see that the tuple is $\cL|_T$--lowered, satisfying the first requirement.

To see the second property, we again use the fact that since
\[
\rhoStiefel(\cV_1,\ldots,\cV_{\ell+1}) = (\cV'_1,\ldots,\cV'_{\ell+1}).
\]
we have that
\[
\rhoStiefel(\cV_1|_{U_i},\ldots,\cV_{k_i}|_{U_i}) =(\cV'_{\ell+2-k_i}|_{U_i},\ldots,\cV'_{\ell+1}|_{U_i})
\]
for each $i\in I$. Now, because $(\cV_1|_{U_i},\ldots,\cV_{k_i}|_{U_i}) \in \LowStiefel_{(\cH,h)}(U_i)$ by assumption, we have that
\[
\cV_j|_{U_i} = \left(\bigoplus_{a\neq \ell+2-j} \cV'_a|_{U_i}\right)^{\perp_h}
\]
Since this holds for all $i\in I$, we conclude that
\[
\cV_j = \left(\bigoplus_{a\neq \ell+2-j}\cV'_a\right)^{\perp_h}
\]
which shows that $(\cV_1,\ldots,\cV_{\ell+1}) \in \LowStiefel_{(\cH,h)}(T)$ as desired, and we are done.
\end{proof}

\begin{example}\label{split_outer_stiefel}
Consider the split hermitian form $(S\sqcup S \to S,(\cO^d,\cO^d),h_d)$ of \eqref{eq_split_hermitian_form}. Then, $\cL=\cO\times\cO$ and $\cL$--submodules which are direct summands of $\cO^d\times\cO^d$ are of the form $\cV\times\cW$ for direct summands $\cV,\cW\subseteq \cO^d$. Given a section
\[
(\cV_1\times\cW_1,\ldots,\cV_{\ell+1}\times\cW_{\ell+1}) \in \LowStiefel_{((\cO^d,\cO^d),h_d)}(S)
\]
we may localize it sufficiently such that each $\cV_j$ and $\cW_j$ are constant rank $\cO$--modules and therefore the tuple is both lowered and raised, meaning that
\[
\rhoStiefel(\cV_1\times\cW_1,\ldots,\cV_{\ell+1}\times\cW_{\ell+1}) = (\cV_1\times\cW_1,\ldots,\cV_{\ell+1}\times\cW_{\ell+1})
\]
(though it technically lands in a different sheaf). Additionally, the condition that the tuple is $\cL$--lowered means that each $\cV_i$ and $\cW_i$ must have positive rank, i.e., they are non-zero. Now, for such a fixed tuple we may analyze the symmetry condition more easily. It states that
\begin{align*}
\cV_i\times\cW_i &= \left(\bigoplus_{j\neq \ell+2-i} \cV_j\times\cW_j\right)^{\perp_h} \\
&= \left(\bigoplus_{j\neq \ell-2-i}\cW_j\right)^\perp \times\left(\bigoplus_{j\neq \ell-2-i}\cV_j\right)^\perp
\end{align*}
and so $\cW_i = (\oplus_{j\neq \ell-2-i}\cV_j)^\perp$. This condition also implies that $(\oplus_{j\neq \ell+2-i}\cW_j)^\perp = \cV_i$, so we only need to consider the first condition. Hence, sufficiently locally sections of $\LowStiefel_{((\cO^d,\cO^d),h_d)}$ take the form
\[
(\cV_1\times (\oplus_{j\neq \ell+1} \cV_j)^\perp, \cV_2 \times (\oplus_{j\neq \ell} \cV_j)^\perp, \ldots,\cV_{\ell+1}\times(\oplus_{j\neq 1} \cV_j)^\perp).
\]
Additionally, if we consider the projection onto the first factor, we obtain a section
\[
(\cV_1,\ldots,\cV_{\ell+1}) \in \ConStiefel_{\cO^d}
\]
since $\rank(\cV_i)\geq 1$ for each $i$. Due to the fact we have been working locally, what we have shown is that we have a map
\begin{align*}
\ConStiefel_{\cO^d} &\to \LowStiefel_{((\cO^d,\cO^d),h_d)} \\
(\cV_1,\ldots,\cV_{\ell+1}) &\mapsto (\cV_1\times (\oplus_{j\neq \ell+1} \cV_j)^\perp,\ldots,\cV_{\ell+1}\times(\oplus_{j\neq 1} \cV_j)^\perp)
\end{align*}
which is clearly injective and which is locally surjective. Thus, it satisfies the universal property of sheafification and induces an isomorphism of sheaves
\[
\LowStiefel_{\cO^d} \iso \LowStiefel_{((\cO^d,\cO^d),h_d)}.
\]
We warn that the same formula appearing above for constant rank direct sum decompositions does not apply to describe this isomorphism of sheaves. For example, over a disconnected scheme $S = X\sqcup Y$ where $\cO^d = (\cO|_X^d,\cO|_Y^d)$, if one naively applies the formula to a section
\[
((\cV_1,\cW_1),(\cV_2,\cW_2),(0,\cW_3)) \in \LowStiefel_{\cO^d}
\]
(where the $\cV_i$ and $\cW_j$ are constant rank) one obtains
\[
((\cV_1\times 0,\cW_1\times(\cW_1\oplus\cW_2)^\perp),(\cV_2\times\cV_1^\perp,\cW_2\times(\cW_1\oplus\cW_3)^\perp),(0\times\cV_2^\perp,\cW_3\times(\cW_2\oplus\cW_3)^\perp)
\]
which is not $\cL$--lowered since $\subrank_{\cL|_X}(\cV_1\times 0) = 0$. The correct image of this section is
\[
((\cV_1\times \cV_1^\perp,\cW_1\times(\cW_1\oplus\cW_2)^\perp),(\cV_2\times\cV_2^\perp,\cW_2\times(\cW_1\oplus\cW_3)^\perp),(0,\cW_3\times(\cW_2\oplus\cW_3)^\perp).
\]
This can be corrected by using $\rhoStiefel$ of \eqref{eq_low_to_rai_Stiefel}. In particular, the isomorphism of sheaves is
\begin{align}
\LowStiefel_{\cO^d} &\iso \LowStiefel_{((\cO^d,\cO^d),h_d)} \label{eq_split_Stiefel_outer} \\
(\cV_1,\ldots,\cV_{\ell+1}) &\mapsto (\cV_1\times (\oplus_{j\neq \ell+1} \cV_j')^\perp,\ldots,\cV_{\ell+1}\times(\oplus_{j\neq 1} \cV_j')^\perp) \nonumber
\end{align}
where $\rhoStiefel(\cV_1,\ldots,\cV_{\ell+1}) = (\cV_1',\ldots,\cV_{\ell+1}')$.
\end{example}

The goal is to use the partial sums of sections in $\LowStiefel_{(\cH,h)}$ to build sections of $\LowFlag_{(\cH,h)}$. When taking these partial sums we see the relationship between $\cL$--subrank and $\cL$--gap in our application.
\begin{lem}\label{subrank_and_gap}
Let $L\to S$ be an \'etale cover of schemes of degree $k$ with corresponding commutative $\cO$--algebra $\cL\colon \Sch_S \to \Rings$. Let $\cM \colon \Sch_S \to \Ab$ be a finite locally free $\cL$--module of constant rank $d$. Consider two $\cL$--submodules which are locally direct summands $\cV,\cW \subseteq \cM$ and such that $\cV\cap \cW=0$, so that $\cV\oplus \cW \subseteq \cM$. Then,
\[
\gap_\cL(\cV,\cV\oplus\cW) = \subrank_\cL(\cW).
\]
\end{lem}
\begin{proof}
Since both the $\cL$--gap and the $\cL$--subrank are locally constant sections, it is sufficient to show the equality holds locally. Thus, we may assume that $\cL=\cO^k$, $\cM=(\cO^d)^k$, and that
\[
\cV = \cV_1\times\ldots\times\cV_k \text{ and } \cW=\cW_1\times\ldots\times\cW_k
\]
for $\cV_i,\cW_i \subseteq \cO^d$ direct summands. In particular, this means that
\[
\cV\oplus \cW = (\cV_1\oplus\cW_1)\times\ldots\times(\cV_k\oplus\cW_k).
\]
Then, setting $\rank_\cO(\cV_i)=r_i$ and $\rank_\cO(\cW_i)=s_i$ we can simply compute that
\begin{align*}
&\gap_\cL(\cV,\cV\oplus\cW) \\
=& \min(\{(r_i+s_i)-r_i \mid 1\leq i \leq k\}) \\
=& \min(\{s_i \mid 1\leq i \leq k\}) \\
=& \subrank_\cL(\cW)
\end{align*}
to conclude the proof.
\end{proof}

Now we can see that the partial sum construction works as desired.
\begin{prop}\label{outer_Stiefel_to_Flag_surj}
Let $(f\colon L \to S,\cH,h)$ be a regular Hermitian form of constant rank $d$. We have a well-defined map of sheaves
\begin{align*}
\LowStiefel_{(\cH,h)} &\to \LowFlag_{(\cH,h)} \\
(\cV_1,\ldots,\cV_{\ell+1}) &\mapsto (0\subseteq \cV_1 \subseteq \cV_1\oplus\cV_2 \subseteq \ldots \subseteq \cV_1\oplus\ldots\oplus\cV_\ell \subseteq f_*(\cH)).
\end{align*}
In particular, an $\cL$--lowered tuple produces a symmetric $\cL$--lowered flag. Furthermore, this is a surjective map of sheaves.
\end{prop}
\begin{proof}
We begin by checking that the resulting flag is $\cL$--lowered. Since each component of a tuple $(\cV_1,\ldots,\cV_{\ell+1})$ is an $\cL$--submodule, each component of the flag will be an $\cL$--submodule. Now, assume that for some $T\in \Sch_S$ and for some index $j$, we have
\[
\gap_{\cL|_T}\big((\cV_1\oplus\ldots\oplus\cV_j)|_T \subseteq  (\cV_1\oplus\ldots\oplus\cV_j\oplus\cV_{j+1})|_T\big) = 0.
\]
Using \Cref{subrank_and_gap}, this means that
\[
\subrank_\cL(\cV_{j+1}|_T) = 0
\]
and therefore the condition that the starting tuple was lowered imposes that $\cV_k|_T = 0$ for all $k\geq j+1$. Then, since $\cE = \cV_1\oplus \ldots \oplus \cV_{\ell+1}$, after restricting we have that
\begin{align*}
\cE|_T &= \cV_1|_T \oplus \ldots \oplus \cV_{\ell+1}|_T \\
&= \cV_1|_T \oplus \ldots \oplus \cV_j|_T
\end{align*}
and so
\[
(\cV_1\oplus\ldots\oplus\cV_j)|_T = \cE|_T = (\cV_1\oplus\ldots\oplus\cV_j\oplus\cV_{j+1})|_T.
\]
Hence, we see that the resulting flag is $\cL$--lowered.

Now, we check that the flag has the necessary symmetry property. We may verify this sufficiently locally where the map takes the form
\begin{equation}\label{outer_Stiefel_to_Flag_surj_i}
\LowStiefel_{((\cO^d,\cO^d),h_d)} \to \LowFlag_{((\cO^d,\cO^d),h_d)}.
\end{equation}
Further, we may assume that we are working with constant rank submodules and so using the description in \Cref{split_outer_stiefel} the map appears as
\[
(\ldots,\cV_i\times (\oplus_{j\neq \ell+2-i} \cV_j)^\perp,\ldots) \mapsto (\ldots,\left(\oplus_{k=1}^i \cV_k\right)\times \left(\oplus_{k=1}^i (\oplus_{j\neq \ell+2-k} \cV_j)^\perp \right),\ldots)
\]
where we have only written the $i^\text{th}$ terms. For each summand within the term
\[
\oplus_{k=1}^i (\oplus_{j\neq \ell+2-k} \cV_j)^\perp,
\]
we have inclusions $(\oplus_{j\neq \ell+2-k} \cV_j)^\perp \subseteq (\cV_1\oplus \ldots \oplus \cV_{\ell+1-i})^\perp$ and so
\[
\oplus_{k=1}^i (\oplus_{j\neq \ell+2-k} \cV_j)^\perp \subseteq (\oplus_{m=1}^{\ell+1-i} \cV_m)^\perp.
\]
Now by considering ranks, namely that $\rank((\oplus_{j\neq \ell+2-k} \cV_j)^\perp)=\rank(\cV_{\ell+2-k})$ and
\[
\rank((\oplus_{m=1}^{\ell+1-i} \cV_m)^\perp) = \sum_{k=1}^i \rank(\cV_{\ell+2-k})
\]
we see that we have an equality
\[
\oplus_{k=1}^i (\oplus_{j\neq \ell+2-k} \cV_j)^\perp = (\oplus_{m=1}^{\ell+1-i} \cV_m)^\perp.
\]
Therefore, when we apply the perpendicular operator $\perp_{h_d}$ to the image flag, the new $i^\text{th}$ term will be the perpendicular of the $(\ell+1-i)^\text{th}$ term which is
\begin{align*}
& \left((\oplus_{k=1}^{\ell+1-i} \cV_k)\times (\oplus_{m=1}^{\ell+1-(\ell+1-i)}\cV_m)^\perp\right)^{\perp_{h_d}} \\
=& ((\oplus_{m=1}^i\cV_m)^\perp)^\perp \times (\oplus_{k=1}^{\ell+1-i} \cV_k)^\perp \\
=& (\oplus_{k=1}^i\cV_k)\times (\oplus_{m=1}^{\ell+1-i} \cV_m)^\perp
\end{align*}
which is the $i^\text{th}$ term of the original flag. Hence, we see that the resulting flag is symmetric locally and thus it is also symmetric globally as required.

Finally, to see surjectivity we note that sufficiently locally when our map takes the form \eqref{outer_Stiefel_to_Flag_surj_i} above, it fits into the commutative diagram
\[
\begin{tikzcd}
\LowStiefel_{\cO^d} \ar[d,two heads] \ar[r,"\eqref{eq_split_Stiefel_outer}"] & \LowStiefel_{((\cO^d,\cO^d),h_d)} \ar[d,"\eqref{outer_Stiefel_to_Flag_surj_i}"]  \\
\LowFlag_{\cO^d} \ar[r,"\eqref{eq_ConFlag_iso}"] & \LowFlag_{((\cO^d,\cO^d),h_d)} 
\end{tikzcd}
\]
where the horizontal maps are isomorphisms and the other downward map is the surjection of \Cref{Stiefel_to_Flag_surjection}. Therefore, our map is surjective as well and we are done.
\end{proof}

\subsubsection{Idempotents in Azumaya Algebras with Involution of the Second Kind}
Just as in the inner case, the sections of $\LowStiefel_{(\cH,h)}$ are related to certain tuples of idempotents in $\cEnd_\cL(f_*(\cH))$ and we may define the appropriate notion for any Azumaya algebra with involution of the second kind. Given an idempotent $e\in f_*(\cB)$, we set $\cI_e = e\cdot f_*(\cB) \subseteq f_*(\cB)$, which is an $\cL$--right ideal which is a direct summand (since $f_*(\cH)=ef_*(\cB) \oplus (1-e)f_*(\cB)$) and therefore a finite locally free $\cO$--module.
\begin{defn}
Let $(f\colon L \to S,\cB,\tau)$ be an Azumaya algebra with involution of the second kind of constant rank $d$. We say a tuple $(e_1,\ldots,e_{\ell+1})$ of idempotents in $f_*(\cB)$ is \emph{$\cL$--lowered} if for every $T\in \Sch_S$ and every index $j$,
\[
\subrank_{\cL|_T}(\cI_{e_j}|_T) = 0 \Rightarrow e_k|_T = 0 \text{ for all } k\geq j.
\]
We define the subpresheaf $\LowIdemp_{(\cB,\tau)} \subseteq \LowIdemp_{f_*(\cB)}$ to act on $T\in \Sch_S$ as
\[
T \mapsto \left\{(e_1,\ldots,e_{\ell+1}) \in \LowIdemp_{f_*(\cB)}(T) \mid \begin{array}{l} \text{the tuple is }\cL\text{--lowered, and} \\ \tau(e_i)= e'_{\ell+2-i} \\ \text{where } \rhoIdemp(e_1,\ldots,e_{\ell+1}) = (e_1',\ldots,e_{\ell+1}') \end{array} \right\}.
\]
\end{defn}

\begin{lem}
The presheaf $\LowIdemp_{(\cB,\tau)}$ defined above is a sheaf.
\end{lem}
\begin{proof}
This follows by an analogous argument as in the proof of \Cref{outer_Stiefel_is_a_sheaf} where one compares the right truncation in $\LowIdemp_{f_*(\cB)}$ to the left truncation in $\RaiIdemp_{f_*(\cB)}$.
\end{proof}

\begin{example}\label{split_outer_Idemp}
Consider the split Azumaya algebra with involution $(f\colon S\sqcup S \to S,(\Mat_d(\cO),\Mat_d(\cO)),\tau_d)$ of the second kind as in \Cref{split_second_involution} where the involution $\tau_d$ is the switch transpose. In this case, $f_*((\Mat_d(\cO),\Mat_d(\cO)) = \Mat_d(\cO)\times\Mat_d(\cO)$ and an idempotent in this algebra is of the form
\[
(e,\varepsilon) \in \Mat_d(\cO)\times\Mat_d(\cO)
\]
for idempotents $e,\varepsilon \in \Mat_d(\cO)$. For a section
\[
((e_1,\varepsilon_1),(e_2,\varepsilon_2),\ldots,(e_{\ell+1},\varepsilon_{\ell+1})) \in \LowIdemp_{((\Mat_d(\cO),\Mat_d(\cO)),\tau_d)}
\]
we may localize until all ranks $\rank_\cO(e_j\Mat_d(\cO))$ and $\rank_\cO(\varepsilon_j\Mat_d(\cO))$ are constant and thus
\[
\rhoIdemp((e_1,\varepsilon_1),(e_2,\varepsilon_2),\ldots,(e_{\ell+1},\varepsilon_{\ell+1})) = (e_1,\varepsilon_1),(e_2,\varepsilon_2),\ldots,(e_{\ell+1},\varepsilon_{\ell+1}),
\]
i.e., the section is fixed by $\rhoIdemp$. Now, the symmetry condition says that
\[
(\varepsilon_i^t,e_i^t)=\tau_d(e_i,\varepsilon_i) = (e_{\ell+2-i},\varepsilon_{\ell+2-i})
\]
or equivalently,
\[
\varepsilon_i = (e_{\ell+2-i})^t
\]
for all $i$. Thus, our section takes the form
\[
((e_1,e_{\ell+1}^t),(e_2,e_\ell^t),\ldots,(e_\ell,e_2^t),(e_{\ell+1},e_1^t)).
\]
Hence, when we consider the map of sheaves
\begin{align*}
\LowIdemp_{\Mat_d(\cO)} &\to \LowIdemp_{((\Mat_d(\cO),(\Mat_d(\cO)),\tau_d)} \\
(e_1,\ldots,e_{\ell+1}) &\mapsto ((e_1,(e'_{\ell+1})^t),(e_2,(e'_\ell)^t),\ldots,(e_\ell,(e'_2)^t),(e_{\ell+1},(e'_1)^t))
\end{align*}
where $\rhoIdemp(e_1,\ldots,e_{\ell+1}) = (e'_1,\ldots,e'_{\ell+1})$, we see that it is injective and locally surjective because sufficiently locally any section is fixed by $\rhoIdemp$. Thus, this map is an isomorphism of sheaves.
\end{example}

When we consider the case of $(L\to S,\cB,\tau) = (L\to S,\cEnd_{\cO|_L}(\cH),\tau_h)$ for a regular Hermitian form, the sheaf $\LowIdemp_{(\cEnd_{\cO|_L}(\cH),\tau_h)}$ should be isomorphic to $\LowStiefel_{(\cH,h)}$ in a fact analogous to \Cref{LowIdemp_to_LowStiefel_iso}. To check this, we use the following lemma. Recall that for an $\cL$--subspace $\cV\subseteq f_*(\cB)$ we have the right $\cL$--ideal $\cI_{\cL,\cV} = \Hom_\cL(f_*(\cB),\cV)$.
\begin{lem}\label{idempotent_orthogonal}
Let $(f\colon L\to S,\cH,h)$ be a regular hermitian form of constant rank $d$ and consider the associated Azumaya algebra with adjoint involution of the second kind $(L\to S,\cEnd_{\cO|_L}(\cH),\tau_h)$.

Given an idempotent $e\in f_*(\cEnd_{\cO|_L}(\cH))$ we have an $\cL$--submodule $\cV_e = \Img(e)\subseteq f_*(\cH)$ which is a direct summand. Then,
\begin{enumerate}
\item \label{idempotent_orthogonal_i} $(\cV_e)^{\perp_h} = 
\cV_{\tau_h(1-e)}$, and
\item \label{idempotent_orthogonal_ii} $\cI_e = \cI_{\cL,(\cV_e)}$.
\end{enumerate}
\end{lem}
\begin{proof}
\noindent\ref{idempotent_orthogonal_i}: We first check that $\cV_{\tau_h(1-e)} \subseteq (\cV_e)^{\perp_h}$. Indeed, for any $v\in \cV_e$ we have $v=e(v)$ and so for any section $x\in f_*(\cH)$ we have
\[
h(\tau_h(1-e)(x),v) = h(x,(1-e)(e(v))) = h(x,0) = 0
\]
and so $\cV_{\tau_h(1-e)} \subseteq (\cV_e)^{\perp_h}$. Conversely, if $x\in (\cV_e)^{\perp_h}$, then for all $w\in f_*(\cH)$ we have
\[
0 = h(x,e(w)) = h(\tau_h(e)(x),w)
\]
and so $\tau_h(e)(x) = 0$ because $h$ is regular. Then,
\[
x = (\tau_h(e)+\tau_h(1-e))(x) = \tau_h(1-e)(x)
\]
which shows that $x\in \cV_{\tau_h(1-e)}$ and so $(\cV_e)^{\perp_h} \subseteq \cV_{\tau_h(1-e)}$. Hence
\[
(\cV_e)^{\perp_h} = \cV_{\tau_h(1-e)}
\]
as claimed.

\noindent\ref{idempotent_orthogonal_ii}: Since $\cV_e = \Img(e)$ by definition, it is clear that $\cI_e \subseteq \cI_{\cL,\cV_e}$. Conversely, since $e$ is idempotent, if $\varphi \in \cI_{\cL,\cV_e}$ then $e\circ \varphi = \varphi$, so $\varphi \in \cI_e$ as well. Therefore,
\[
\cI_e = \cI_{\cL,\cV_e}.
\]
\end{proof}

\begin{lem}\label{outer_Stiefel_to_Idemp}
Let $(f\colon L \to S, \cH,h)$ be a regular Hermitian form of constant rank $d$. We have an isomorphism of sheaves
\begin{align*}
\LowIdemp_{(\cEnd_{\cO|_L}(\cH),\tau_h)} &\to \LowStiefel_{(\cH,h)}  \\
(e_1,\ldots,e_{\ell+1}) &\mapsto (\cV_{e_1},\ldots,\cV_{\ell+1}) \\
(\pi_1,\ldots,\pi_{\ell+1}) &\;\reflectbox{$\mapsto$}\; (\cV_1,\ldots,\cV_{\ell+1})
\end{align*}
where for a section $(\cV_1,\ldots,\cV_{\ell+1}) \in \LowStiefel_{(\cH,h)}$ we set $\pi_i$ to be the projection onto $\cV_i$ which is zero on all other $\cV_j$.
\end{lem}
\begin{proof}
If this map is well-defined then it is clear it is an isomorphism since it will be the restriction of the isomorphism occuring in \Cref{LowIdemp_to_LowStiefel_iso}. Therefore, we must argue that the additional conditions coming from $h$ and $\tau_h$ are respected by this map.

By the commutative diagram in \Cref{rhoIdemp}, we know that if we denote
\[
\rhoIdemp(e_1,\ldots,e_{\ell+1}) = (e_1',\ldots,e_{\ell+1}')
\]
then we also have that
\[
\rhoStiefel(\cV_{e_1},\ldots,\cV_{\ell+1}) = (\cV_{e_1'},\ldots,\cV_{e_{\ell+1}'}).
\]
Now, starting with a section $(e_1,\ldots,e_{\ell+1}) \in \LowIdemp_{(\cEnd_{\cO|_L}(\cH),\tau_h)}$ we need to show that 
\[
\cV_{e_i} = \left(\bigoplus_{j\neq \ell+2-i} \cV_{e'_j}\right)^{\perp_h}.
\]
Since the tuple $(e_1',\ldots,e_{\ell+1}')$ consists of pairwise orthogonal idempotents,
\[
\bigoplus_{j\neq \ell+2-i} \cV_{e'_j} = \cV_{(\sum_{j\neq \ell+2-i} e'_j)}
\]
and so using \Cref{idempotent_orthogonal}\ref{idempotent_orthogonal_i} we have
\[
\left(\cV_{(\sum_{j\neq \ell+2-i} e'_j)}\right)^{\perp_h} = \cV_{\tau_h(1-\sum_{j\neq \ell+2-i} e'_j)} = \cV_{\tau_h(e'_{\ell+2-i})}
\]
where finally $\cV_{\tau_h(e'_{\ell+2-i})} = \cV_{e_i}$ as required due to the fact that $\tau_h(e'_{\ell+2-i})=e_i$ because we started with a section in $\LowIdemp_{(\cEnd_{\cO|_L}(\cH),\tau_h)}$.

Conversely, starting with a section $(\cV_1,\ldots,\cV_{\ell+1}) \in \LowStiefel_{(\cH,h)}$ we need to show that
\[
\tau_h(\pi'_{\ell+2-i})=\pi_i
\]
where here as well we have that $(\pi'_1,\ldots,\pi'_{\ell+1})$ are the projections associated to
\[
\rhoStiefel(\cV_1,\ldots,\cV_{\ell+1}) = (\cV'_1,\ldots,\cV'_{\ell+1}).
\]
By assumption we have that for all $i$,
\[
\cV_i = \left(\bigoplus_{j\neq \ell+2-i}\cV'_j\right)^{\perp_h} = \left(\cV_{(\sum_{j\neq \ell+2-i} \pi'_j)}\right)^{\perp_h} = \cV_{\tau_h(\pi'_{\ell+2-i})}
\]
once again using \Cref{idempotent_orthogonal}\ref{idempotent_orthogonal_i}. So, $\tau_h(\pi'_{\ell+2-i})$ is another projection onto $\cV_i$. If it is also zero on all other $\cV_j$, then it will agree with $\pi_i$. However,
\[
\tau_h(\pi'_{\ell+2-i})(\cV_j) = \tau_h(\pi'_{\ell+2-i})(\cV_{\tau_h(\pi'_{\ell+2-j})}) = 0
\]
since $\tau_h(\pi'_{\ell+2-i})$ and $\tau_h(\pi'_{\ell+2-j})$ are orthogonal idempotents. Hence, we conclude that
\[
\pi_i = \tau_h(\pi'_{\ell+2-i})
\]
as required. Thus, the inverse map lands in $\LowIdemp_{(\cEnd_{\cO|_L}(\cH),\tau_h)}$ and so the map is well-defined and we are done.
\end{proof}
We may combine the isomorphism in the previous lemma with the surjection of \Cref{outer_Stiefel_to_Flag_surj} and the isomorphism from \Cref{iso_LowFlag_Hh_LowGSB} to form the diagram
\[
\begin{tikzcd}
\LowIdemp_{(\cEnd_{\cO|_L}(\cH),\tau_h)} \ar[r,"\sim"] \ar[d,two heads] & \LowStiefel_{(\cH,h)} \ar[d,two heads] \\
\LowGSB_{(\cEnd_{\cO|_L}(\cH),\tau_h)} \ar[r,"\sim"] & \LowFlag_{(\cH,h)}.
\end{tikzcd}
\]
Using \Cref{idempotent_orthogonal}\ref{idempotent_orthogonal_ii}, we see that the downward surjection on the left takes the form
\begin{align}
&\LowIdemp_{(\cEnd_{\cO|_L}(\cH),\tau_h)} \surj \LowGSB_{(\cEnd_{\cO|_L}(\cH),\tau_h)} \label{Hermitian_Idemp_to_GSB} \\
&(e_1,\ldots,e_{\ell+1}) \mapsto (0\subseteq \cI_{e_1} \subseteq \cI_{(e_1+e_2)} \subseteq \ldots \subseteq \cI_{(e_1+\ldots+e_\ell)} \subseteq f_*(\cH)). \nonumber
\end{align}

In the general case there is also such a surjection.
\begin{cor}\label{outer_Idemp_GSB_surj}
Let $(f\colon L \to S,\cB,\tau)$ be an Azumaya algebra with involution of the second kind of constant rank $d$. There is a surjective map
\begin{align*}
\LowIdemp_{(\cB,\tau)} &\surj \LowGSB_{(\cB,\tau)} \\
(e_1,\ldots,e_{\ell+1}) &\mapsto (0\subseteq \cI_{e_1} \subseteq \ldots \subseteq \cI_{(e_1+\ldots+e_\ell)} \subseteq f_*(\cB))
\end{align*}
\end{cor}
\begin{proof}
This map locally takes the form of \eqref{Hermitian_Idemp_to_GSB} and so it is well-defined and surjective.
\end{proof}

\subsubsection{Parabolics and Levis}\label{Outer_parabolics_and_levis}
Consider $\bG = \bU_{(\cB,\tau)}$ for an Azumaya algebra $(f\colon L\to S,\cB,\tau)$ with involution of the second kind. Here we define an isomorphism between $\LowIdemp_{(\cB,\tau)}$ and $\cPL_\bG$ using cocharacters analogous to what was done in \Cref{inner_parabolics_and_levis}. Given a section $\ve = (e_1,\ldots,e_{\ell+1}) \in \LowIdemp_{(\cB,\tau)}$ we utilize the same cocharacter defined in \eqref{inner_cocharacter}, namely
\begin{align*}
\lambda_{\ve} \colon \GG_m &\to \GL_{1,f_*(\cB)} \\
t &\mapsto \sum_{i=1}^{\ell+1} t^{\underline{a}(\ve)+1-2i}e_i.
\end{align*}
where $\underline{a}(\ve) \in \underline{\ZZ}$ is the locally constant length of $\ve$. While it is immediate that this cocharacter lands in $\GL_{1,f_*(\cB)}$, we require that it lands in $\bU_{(\cB,\tau)}$. To see this, we utilize \Cref{cocharacter_raised_formula}.
\begin{lem}
Let $(L\to S,\cB,\tau)$ be an Azumaya algebra with involution of the second kind. For a section $\ve \in \LowIdemp_{(\cB,\tau)}$, the cocharacter $\lambda_{\ve}$ is a cocharacter into $\bU_{(\cB,\tau)}$. That is, 
\[
\tau(\lambda_{\ve}(t)) = \lambda_{\ve}(t)^{-1}
\]
for all $t\in \GG_m$.
\end{lem}
\begin{proof}
Let $\ve = (e_1,\ldots,e_{\ell+1}) \in \LowIdemp_{(\cB,\tau)}$ and set $\rhoIdemp(\ve) = (e_1',\ldots,e_{\ell+1}')$. By definition, we have that $\tau(e_i) = e'_{\ell+2-i}$ for all $i$. Now, we simply compute that
\[
\tau(\lambda_{\ve}(t)) = \sum_{i=1}^{\ell+1} t^{\underline{a}(\ve)+1-2i}\tau(e_i) =\sum_{i=1}^{\ell+1} t^{\underline{a}(\ve)+1-2i}e'_{\ell+2-1} 
\]
which by \Cref{cocharacter_raised_formula} is
\[
= \sum_{i=1}^{\ell+1} t^{-\underline{a}(\ve)-1+2i}e_i = \lambda_{\ve}(t^{-1}) = \lambda_{\ve}(t)^{-1},
\]
and so we are done.
\end{proof}

Now that we have a cocharacter $\lambda_{\ve} \colon \GG_m \to \bU_{(\cB,\tau)}$, we can consider its parabolic and Levi pair $(\bP(\lambda_{\ve}),\bL(\lambda_{\ve}))$.

\begin{thm}\label{outer_Idemp_PL_iso}
Let $\bG = \bU_{(\cB,\tau)}$ for an Azumaya algebra with involution of the second kind $(L\to S,\cB,\tau)$ of constant degree $d$. There is an isomorphism of sheaves
\begin{align*}
\LowIdemp_{(\cB,\tau)} &\iso \cPL_\bG \\
\ve &\mapsto (\bP(\lambda_{\ve}),\bL(\lambda_{\ve})).
\end{align*}
\end{thm}
\begin{proof}
We have seen above that this map is well-defined, so it suffices to show that it is an isomorphism locally. Thus, we may assume that
\[
(L\to S,\cB,\tau) = (S\sqcup S \to S, (\Mat_d(\cO),\Mat_d(\cO)),\tau_d).
\]
In this case
\[
\bU_{(\cB,\tau)}(T) = \{(A,(A^{-1})^t) \in \Mat_d(\cO(T))^\times \times\Mat_d(\cO(T))^\times \}
\]
and so we have an isomorphism $\bU_{(\cB,\tau)}\iso \GL_d$ by projecting onto the first factor.

Similarly, given a section $\ve = ((e_1,\varepsilon_1),\ldots,(e_{\ell+1},\varepsilon_{\ell+1})) \in \LowIdemp_{(\cB,\tau)}$, projection onto the first factor yields a section $(e_1,\ldots,e_{\ell+1})\in \LowIdemp_{\Mat_d(\cO)}$ under to the isomorphism of \Cref{split_outer_Idemp}.

Furthermore, the associated cocharacter of $\ve$ can by calculated component wise, i.e.,
\[
\lambda_{\ve}(t) = \left(\sum_{i=1}^{\ell+1} t^{\underline{a}(\ve)+1-2i}e_i,\sum_{i=1}^{\ell+1} t^{\underline{a}(\ve)+1-2i}\varepsilon_i\right),
\]
where the cocharacter associated to $(e_1,\ldots,e_{\ell+1})$ by \eqref{inner_cocharacter} appears in the first factor (justifying our use of that definition in the inner case). Therefore, this construction is also compatible with the projection onto the first factor. In particular, we have a commutative diagram
\[
\begin{tikzcd}
\LowIdemp_{(\cB,\tau)} \ar[r] \ar[d,"\varphi"] & \cHom_{\textrm{Grp}}(\GG_m,\bU_{(\cB,\tau)}) \ar[d,"\rotatebox{90}{$\sim$}"] \ar[r] & \cPL_\bG \ar[d,equals] \\
\LowIdemp_{\Mat_d(\cO)} \ar[r] & \cHom_{\textrm{Grp}}(\GG_m,\GL_d) \ar[r] & \cPL_\bG
\end{tikzcd}
\]
where the $\varphi$ is the isomorphism of \Cref{split_outer_Idemp} and the right side of the diagram sends a cocharacter to its parabolic Levi pair. Then, the composition along the bottom row is the isomorphism of \Cref{Idemp_PL_iso} and so the composition along the top row, which is our proposed map, is an isomorphism as well.
\end{proof}

\begin{cor}
Let $\bG = \bU_{(\cB,\tau)}$ for an Azumaya algebra with involution of the second kind. The isomorphism of \Cref{outer_Idemp_PL_iso} fits into the commutative diagram
\[
\begin{tikzcd}
\LowIdemp_{(\cB,\tau)} \ar[r,"\sim"] \ar[d,two heads] & \cPL_\bG \ar[d,two heads] \\
\LowGSB_{(\cB,\tau)} \ar[r,"\sim"] & \cPar_\bG
\end{tikzcd}
\]
where the downward maps are those of \Cref{outer_Idemp_GSB_surj} and \eqref{PL_to_Par_surj} respectively, and the bottom isomorphism is the map of \Cref{iso_outer_GSB_to_parabolics}.
\end{cor}
\begin{proof}
Locally this becomes the diagram of \Cref{inner_PL_iso_diagram}.
\end{proof}

\section{Severi-Brauer Schemes and Quillen's Construction}
In this final section we extend the classical categorical equivalence between Azumaya algebras and Sever-Brauer schemes to include objects of outer type $A$. This involves defining the notion of an outer type Severi-Brauer scheme in order to correspond to an Azumaya algebra $(\cB,\tau)$ with involution of the second kind. However, we begin by reviewing the correspondence for inner type objects, continuing to use the convention that $d=n+1$.

\subsection{Inner Type}\label{Severi-Brauer_review}
Over our base scheme $S$, a \emph{Severi-Brauer scheme} of constant relative dimension $n$ is a scheme $f\colon P \to S$ in $\Sch_S$ such that it is \'etale locally isomorphic to relative projective space $\PP_S^n \to S$ for some $n$. Since $\bAut(\PP_S^n) \cong \PGL_d$, there is a formal categorical equivalence between Severi-Brauer schemes of constant relative dimension $n$ and Azumaya $\cO$--algebras of constant degree $d$ due to the fact that both of these categories are equivalent to the category of $\PGL_d$--torsors. However, there is also a more direct and concrete equivalence between these two categories, namely
\begin{align}
\mathfrak{Azu}_d \leftrightarrow &\mathfrak{SB}_n \label{eq_usual_Quillen} \\
\cA \mapsto &\SB(\cA) \nonumber \\
f_*\big(\cEnd_{\cO|_P}(\cF(P))\big) \reflectbox{$\mapsto$} &P. \nonumber
\end{align}
Some explanation of notation is required. The Severi-Brauer scheme of $\cA$, denoted $\SB(\cA)$, is the scheme of finite locally free right ideals in $\cA$ of constant rank $d$ which are locally direct summands. Framing this in our context, since $\cA$ is of inner type $A_n$, we have the commutative diagram of \Cref{inner_type_GSB_Par}
\[
\begin{tikzcd}
\LowGSB_\cA \ar[r] \ar[d,"t_{\cGSB}"] & \cPar_{\GL_{1,\cA}} \ar[d,"t"] \\
\cP_n \ar[r,"\und^c"] & \cP_n
\end{tikzcd}
\]
where $\cP_n = \cOf(\cDyn(\GL_{1,\cA}))$. Since $\cPar_{\GL_{1,\cA}}$ is representable as mentioned in \Cref{prelim_parabolics_and_levis} so is $\LowGSB_{\cA}$ since the top map in the diagram is the isomorphism of \Cref{SB_isomorphism_with_Par}. Say $\LowGSB_{\cA}$ is represented by $X\in \Sch_S$. The sheaf $\cP_n$ is represented by $(\sP_n)_S$ which is $2^n$ disjoint copies of $S$, indexed by subsets of $\{1,\ldots,n\}$ written as increasing tuples. Thus we have a map of schemes
\[
X \to \bigsqcup_{r\in \sP_n} S_r
\]
corresponding to $t_{\cGSB}$ where all $S_r=S$. The type morphism sends a flag of ideals to its tuple of ranks divided by $d$, and so if we consider the fiber of this map over $S_{(1)}$, we obtain a scheme
\[
f\colon X_{(1)} \to S
\]
which represents finite locally free right ideals in $\cA$ of constant rank $d$ which are locally direct summands. We define $\SB(\cA) = X_{(1)}$. It is well known that if we start with $\cA = \Mat_d(\cO)$, then $X_{(1)} \cong \PP_S^n$ and so the schemes produced through this process for general Azumaya $\cO$--algebras are indeed Severi-Brauer schemes.

In the other direction, we have Quillen's construction, which we review from \cite[Theorem 81]{Kollar}. Given a Severi-Brauer scheme $f\colon P \to S$, there is a unique non-split extension of $\cO|_P$--modules
\[
0 \to \cO|_P \to \cF(P) \to \cT_{P/S} \to 0
\]
where $\cT_{P/S}$ is the tangent bundle of $P$ over $S$. Koll\'ar notes that the bundle $\cF(P)^*$, i.e., the dual, appears in \cite{Quillen}, leading to the name of the construction. When $P=\PP_S^n$, this unique non-split extension takes the form
\[
0 \to \cO|_{\PP_S^n} \to \cO_{\PP_S^n}(1)^d \to \cT_{\PP_S^n/S} \to 0
\]
which is the Euler sequence for relative projective space. Since $\cF(P)$ locally becomes isomorphic to $\cO_{\PP_S^n}(1)^d$, it is constant rank $d$ and thus the endomorphism algebra $\cEnd_{\cO|_P}(\cF(P))$ is an Azumaya algebra over $P$ of constant degree $d$. The pushforward retains this property, and so we set
\[
\cA_P = f_*(\cEnd_{\cO|_P}(\cF(P)))
\]
which is called the Azumaya $\cO$--algebra associated to $P$. In the case of $\PP_S^n$ it is clear that $\cEnd_{\cO|_{\PP_S^n}}(\cO_{\PP_S^n}(1)^d) \cong \Mat_d(\cO|_{\PP_S^n})$ and so $\cA_{\PP_S^n} = \Mat_d(\cO)$, showing that for general Severi-Brauer schemes $P$ the algebra $\cA_P$ is indeed an Azumaya $\cO$--algebra.

Let $P$ be a Severi-Brauer scheme. Using the equivalence between Severi-Brauer schemes and Azumaya algebras, say $P=\SB(\cA)$, we define the opposite Severi-Brauer scheme to be $P^\op = \SB(\cA^\op)$. In \cite{Kollar} (for Severi-Brauer varities over a field), this is called the dual and denoted $P^\vee$. Given a right ideal $\cJ\subset \cA^\op$ of constant rank $d$ which is locally a direct summand, its left annihilator $\lann\cJ$ is a left ideal of $\cA^\op$ of constant rank $d^2-d = d(d-1) = dn$ which is also locally a direct summand. Then, since left ideals in $\cA^\op$ are also right ideals in $\cA$ itself, this corresponds to a section of $\LowGSB_\cA$ with type $(n)$. Therefore, we see that
\[
\SB(\cA^\op) \cong X_{(n)}
\]
is also a fiber of the map $X \to \sqcup_{r\in \sP_n} S_r$ corresponding to $t_{\cGSB}$.

\subsection{A Variation on Quillen's Construction}\label{Quillen_variation}
Here we will work with relative projective space and various vector bundles on relative projective space and we will twist both the scheme and vector bundles simultaneously. In order to keep track of this simultaneous twisting and not get lost in isomorphisms, we will work with these objects' incarnations as stacks over $\Sch_S$. The punchline of this section is that for a Severi-Brauer scheme $f\colon P \to S$, the pushforward $f_*(\cF(P))$ can be given an algebra structure without first taking the endomorphism algebra, and $f_*(F(\SB(A))) \cong A$.

\subsubsection{The Relative Projective Stack}
For any $S$--scheme $f\colon X \to S$, the functor $\Sch_X \to \Sch_S$ sending an $X$--scheme $g\colon U \to X$ to the $S$--scheme $f\circ g \colon U \to S$ and which acts obviously on morphisms, gives $\Sch_X$ the structure of a stack over $\Sch_S$. This is the same stack one obtains when constructing a stack from the sheaf $h_X \colon \Sch_S \to \Sets$ represented by $X$. In the case of relative projective space $\PP_S^n \to S$, we may use the characterization of morphisms $U\to \PP_S^n$ from \cite[Tag 01NE]{Stacks} and consider the stack
\[
\PP^n \to \Sch_S
\]
with
\begin{enumerate}
\item objects $(U,\cL,(s_0,\ldots,s_n))$ where $U \in \Sch_S$, $\cL \colon \Sch_U \to \Ab$ is a line bundle, and $s_0,\ldots,s_n\in \cL(U)$ are sections which globally generate $\cL$,
\item morphisms $(g,\psi)\colon (U',\cL',(s_0',\ldots,s_n')) \to (U,\cL,(s_0,\ldots,s_n))$ where $g\colon U'\to U$ is a morphism in $\Sch_S$ and $\psi \colon \cL|_{U'} \iso \cL'$ is an isomorphism of $\cO|_{U'}$--line bundles such that $\psi(s_i|_{U'})=s_i'$, and
\item structure morphism sending $(U,\cL,(s_0,\ldots,s_n)) \mapsto U$ and $(g,\psi)\mapsto g$.
\end{enumerate}
This is not exactly the stack $\Sch_{\PP_S^n} \to \Sch_S$, since in any stack of the form $\Sch_X \to \Sch_S$ the fiber over $U$ does not contain any non-identity isomorphisms, while in $\PP^n(U)$ we have isomorphisms of the form
\[
(U,\cL,(s_0,\ldots,s_n)) \iso (U,\cL,(as_0,\ldots,as_n)) 
\]
for $a\in \GG_m(U)$. Nevertheless, the automorphism group of any object in $\PP^n(U)$ is trivial and so there is an equivalence of stacks $\Sch_{\PP_S^n} \cong \PP^n$, so we choose to work with $\PP^n \to \Sch_S$ in order to avoid working with equivalence classes of tuples up to scaling. Relative projective space $\PP_S^n$ has the line bundle $\cOSone$ with globally generating sections $x_0,\ldots,x_n$. A morphism of schemes $U\to \PP_S^n$ in $\Sch_S$ corresponds to the object $(U,\cOSone|_U,(x_0|_U,\ldots,x_n|_U)) \in \PP^n$ where the restriction to $U$ is along the given morphism.

The automorphism group of relative projective space is $\bAut_S(\PP_S^n) = \PGL_d$ and it acts on the stack $\PP^n \to \Sch_S$ in the following way. Consider $\varphi \in \PGL_d(S)$. The exact sequence of groups $1 \to \GG_m \to \GL_d \to \PGL_d \to 1$ gives us a long exact sequence in cohomology with first connecting morphism
\[
\PGL_d(S) \xrightarrow{\delta} \Pic(S).
\]
Let $\delta(\varphi)=\cI$. This means that there exists an isomorphism $A \colon \cO^d \to \cI^d$ of $\cO$--modules such that the induced isomorphism on the endomorphism algebras
\begin{align*}
\Mat_d(\cO) = \cEnd_\cO(\cO^d) &\iso \cEnd_\cO(\cI^d) = \Mat_d(\cO) \\
T &\mapsto A T A^{-1}
\end{align*}
is the isomorphism $\varphi$, viewing $\PGL_d = \bAut(\Mat_d(\cO))$. We may also view this more literally as matrix multiplication. The map $A$ can be written as a matrix $A\in \Mat_d(\cI)$ with entries in $\cI$, and then its inverse $A^{-1} \in \Mat_d(\cI^*)$ has entries in the dual line bundle. The product $ATA^{-1}$ can then be calculated as usual where the entries of $A^{-1}$ act by evaluation when encountering a section of $\cI$.

For any object $(U,\cL,(s_0,\ldots,s_n)) \in \PP^n$, by restricting and tensoring we get an isomorphism
\[
A|_U \otimes 1 \colon \cL^d = \cO|_U^d \otimes_{\cO|_U} \cL \iso \cI_U^d \otimes_{\cO|_U} \cL = (\cI|_U\otimes_{\cO|_U} \cL)^d
\]
which sends the section $(s_0,\ldots,s_n)$ to some section $(s_0',\ldots,s_n')\in (\cI|_U\otimes_{\cO|_U} \cL)^d$. The section $\varphi \in \PGL_d$ then acts on $\PP^n$ as a $\Sch_S$--stack isomorphism by
\begin{align*}
\PP^n &\xrightarrow{\varphi} \PP^n \\
(U,\cL,(s_0,\ldots,s_n)) &\mapsto (U,\cI|_U\otimes_{\cO|_U} \cL,(s_0',\ldots,s_n')) \\
(g,\psi) &\mapsto (g,1\otimes\psi).
\end{align*}

\subsubsection{Some Vector Bundles on $\PP^n$}
Any sheaf $\cF \colon \Sch_X \to \Sets$ can be viewed as a stack over $\Sch_X$ by placing each section of $\cF(U)$ as an object over $U\in \Sch_S$ and adding morphisms to encode restrictions. We do this for the sheaves $\cO|_{\PP_S^n}$, $\cOSone$, and the tangent bundle $\cT_{\PP_S^n/S}$ on $\Sch_{\PP_S^n}$ and we write them as stacks over $\PP^n$. We abuse notation and denote the resulting stacks the same way.

For the structure sheaf, the global sections $\cO(U)$ are intrinsic to $U$ and do not depend on the chosen morphism to $\PP_S^n$, and so as a stack it becomes
\[
\cO|_{\PP^n} \to \PP^n
\]
with
\begin{enumerate}
\item objects $(U,\cL,(s_0,\ldots,s_n),a\in \cO(U))$ where $(U,\cL,(s_0,\ldots,s_n))\in \PP^n$,
\item morphisms
\[
(g,\psi)\colon (U',\cL',(s_0',\ldots,s_n'),a'\in \cO(U'))\to (U,\cL,(s_0,\ldots,s_n),a\in \cO(U))
\]
where $(g,\psi)$ is a morphism between the underlying objects in $\PP^n$ such that $a|_{U'}=a'$, and
\item with the obvious structure morphism to $\PP^n$.
\end{enumerate}

For the line bundle $\cOSone$, the interpretation of an object $(U,\cL,(s_0,\ldots,s_n)) \in \PP^n$ is that $\cL$ is the restriction of $\cOSone$ to $U$. Therefore, the appropriate stack of sections is
\[
\cOone \to \PP^n
\]
with
\begin{enumerate}
\item objects $(U,\cL,(s_0,\ldots,s_n),\ell\in \cL(U))$ where $(U,\cL,(s_0,\ldots,s_n))\in \PP^n$,
\item morphisms
\[
(g,\psi)\colon (U',\cL',(s_0',\ldots,s_n'),\ell'\in \cL'(U')) \to (U,\cL,(s_0,\ldots,s_n),\ell\in \cL(U))
\]
where $(g,\psi)$ is a morphism between the underlying objects in $\PP^n$ such that $\psi(\ell|_{U'})=\ell'$, and
\item with the obvious structure morphism to $\PP^n$.
\end{enumerate}
The vector bundle $\cOone^d$ then appears similarly, where we instead have objects of the form $(U,\cL,(s_0,\ldots,s_n),v\in \cL^d(U))$.

Finally, for any morphism of schemes $X\to S$, the tangent bundle $\cT_{X/S} \colon \Sch_X \to \Sets$ can be defined by $\cT_{X/S}(Y) = \Ker(h_X(Y[\varepsilon]) \to h_X(Y))$ where $Y[\varepsilon]$ is the scheme of dual numbers over $Y$ and the morphism $Y \to Y[\varepsilon]$ is given by sending $\varepsilon \mapsto 0$ within the structure sheaves. In the context of projective space, this means that the tangent stacks is
\[
\cT_{\PP^n/S} \to \PP^n
\]
with
\begin{enumerate}
\item objects $(U,\cL,(s_0,\ldots,s_n),\cJ,(j_0,\ldots,j_n))$ where $(U,\cL,(s_0,\ldots,s_n))\in \PP^n$ and $\cJ$ is a line bundle of $U[\varepsilon]$ globally generated by $j_0,\ldots,j_n$ and such that $\cJ|_U = \cL$ and $j_i|_U = s_i$,
\item morphisms
\[
(g,\psi,\rho) \colon (U',\cL',(s_0',\ldots,s_n'),\cJ',(j_0',\ldots,j_n')) \to (U,\cL,(s_0,\ldots,s_n),\cJ,(j_0,\ldots,j_n))
\]
where $(g,\psi)$ is a morphism between the underlying objects in $\PP^n$ and $\rho \colon \cJ|_{U'[\varepsilon]} \iso \cJ$ is an isomorphism of $U'[\varepsilon]$--line bundles such that we have $\rho(j_k|_{U'[\varepsilon]})=j_k'$, which also implies that $\rho|_U=\psi$, and
\item with the obvious structure morphism to $\PP^n$.
\end{enumerate}
Similar to $\PP^n$ itself, while this stack does not have any non-trivial automorphisms, it does have some non-trivial isomorphisms in the fiber over an object in $\PP^n$. Namely, for any section $b\in \cO(U)$ we may scale by $1+b\varepsilon \in \GG_m(U[\varepsilon])$ and obtain an isomorphism
\[
(U,\cL,(s_0,\ldots,s_n),\cJ,(j_0,\ldots,j_n)) \iso (U,\cL,(s_0,\ldots,s_n),\cJ,(j_0 +b\varepsilon j_0,\ldots,j_n+b\varepsilon j_n))
\]
in the fiber $\cT_{\PP^n/S}((U,\cL,(s_0,\ldots,s_n))$. This appears as something more familiar for certain objects in this fiber, namely those where $\cJ = \cL\oplus \varepsilon\cL$. All such objects take the form
\[
(U,\cL,(s_0,\ldots,s_n),\cL+\varepsilon\cL,(s_0+\varepsilon t_0,\ldots,s_n+\varepsilon t_n))
\]
and the above object is isomorphic to all objects of the form
\[
(U,\cL,(s_0,\ldots,s_n),\cL+\varepsilon\cL,(s_0+\varepsilon(t_0+bs_0),\ldots,s_n+\varepsilon(t_n+bs_n)))
\]
for all $b\in \cO(U)$. That is to say, the isomorphism class containing these objects can be interpreted as the section $\overline{(t_0,\ldots,t_n)} \in \left(\cL^d/\langle(s_0,\ldots,s_n)\rangle\right)(U)$, where one may like to intuitively think of $\cL^d/\langle(s_0,\ldots,s_n)\rangle$ as the vector bundle perpendicular to the vector $(s_0,\ldots,s_n)$.

The Euler sequence for relative projective space can then be written in terms of $\PP^n$--stacks morphisms between objects. The map $\cO|_{\PP_S^n} \to \cOSone^d$ takes the form
\begin{align*}
\cO|_{\PP^n} &\to \cOone^d \\
(U,\cL,(s_0,\ldots,s_n),a\in\cO(U)) &\mapsto (U,\cL,(s_0,\ldots,s_n),(as_0,\ldots,as_n)\in \cL^d(U))
\end{align*}
and it behaves as one would expected on morphisms. The map $\cOSone^d \to \cT_{\PP_S^n/S}$ appears in terms of stacks as
\begin{align*}
\cOone^d &\to \cT_{\PP^n/S} \\
(U,\cL,(s_0,\ldots,s_n),v\in \cL^d(U)) &\to (U,\cL,(s_0,\ldots,s_n),\cL\oplus \varepsilon\cL,(s_0+\varepsilon v_0,\ldots,s_n +\varepsilon v_n))
\end{align*}
where $v_i$ are the coordinates of $v$. In the intuitive interpretation given above, this simply sends $v\in \cL^d(U)$ to $\overline{v}\in \left(\cL^d/\langle (s_0,\ldots,s_n)\rangle\right)(U)$.

\subsubsection{Twisting by $\PGL_d$ Actions}
For each of $\cO|_{\PP^n}$, $\cOone^d$, and $\cT_{\PP^n/S}$ we define an action of $\PGL_d$ by $\Sch_S$--stack morphisms which act compatibly with the $\PGL_d$ action on $\PP^n$ itself.

Let $\varphi \in \PGL_d(S)$ with associated $\cO$--module isomorphism $A \colon \cO^d \to \cI^d$ for a line bundle $\cI$ on $S$. Recall that it acts on $\PP^n$ by
\[
(U,\cL,(s_0,\ldots,s_n)) \mapsto (U,\cI|_U\otimes_{\cO|_U} \cL,(s_0',\ldots,s_n'))
\]
where the $s_i'$ are the coordinates of $(A\otimes 1)(s_0,\ldots,s_n) \in (\cI|_U\otimes_{\cO|_U} \cL)^d(U)$. The section $\varphi$ acts naturally on the structure sheaf, simply by
\begin{align*}
\cO|_{\PP^n} &\xrightarrow{\varphi} \cO|_{\PP^n} \\
(U,\cL,(s_0,\ldots,s_n),a\in \cO(U)) &\mapsto (U,\cI|_U\otimes_{\cO|_U} \cL,(s_0',\ldots,s_n'),a\in \cO(U)).
\end{align*}

We define the action of $\varphi$ on $\cOone^d$ by
\begin{align}
&\varphi(U,\cL,(s_0,\ldots,s_n),v\in \cL^d(U)) \label{eq_action_on_Oone^d} \\
=&(U,\cI|_U\otimes_{\cO|_U} \cL,(s_0',\ldots,s_n'),(A\otimes 1)(v) \in (\cI|_U\otimes_{\cO|_U}\cL)^d(U)). \nonumber
\end{align}

Finally, the definition of $\cT_{\PP^n/S}$ as consisting of morphisms in $\PP^n$ belonging to certain kernels means that the morphism $\varphi \colon \PP^n \to \PP^n$ naturally induces a compatible morphism on $\cT_{\PP^n/S}$. Namely, this is
\begin{align*}
&\varphi(U,\cL,(s_0,\ldots,s_n),\cJ,(j_0,\ldots,j_n)) \\
=&(U,\cI|_U\otimes_{\cO|_U}\cL,(s_0',\ldots,s_n'),\cI|_{U[\varepsilon]}\otimes_{\cO|_{U[\varepsilon]}}\cJ,(j_0',\ldots,j_n'))
\end{align*}
where $j_k'$ are the coordinates of $(A\otimes 1)(j_0,\ldots,j_n)$. For all of these three actions, $\varphi$ behaves on morphisms as one would expect. The definitions of these actions make it clear that morally $\varphi$ is acting by the associated isomorphism $A$ on all the relevant $d$--tuples and thus it is obvious that all of these actions commute with the respective structure morphisms down to $\PP^n$, which is what we mean by saying the actions are compatible with the action on $\PP^n$. It is also clear that the maps
\[
\cO|_{\PP^n} \to \cOone^d \to \cT_{\PP^n/S}
\]
in the Euler exact sequence are equivariant with respect to these $\PGL_d$--actions. Since the category of stacks satisfies descent, for example as explained in \cite[Example 1.11(i)]{Breen}, these $\PGL_d$--actions can be used to simultaneously twist $\PP^n$ as well as the exact sequence of sheaves on $\PP^n$ with respect to $\PGL_d$--torsors.

\begin{lem}\label{simul_twisting}
Let $P\to S$ be a Severi-Brauer scheme of constant relative dimension $n$ with corresponding $\PGL_d$--torsor $\cK$. Then, the twisting the diagram of stack morphisms
\[
\begin{tikzcd}
\cO|_{\PP^n} \ar[r] \ar[dr] & \cOone^d \ar[r] \ar[d] & \cT_{\PP^n/S} \ar[dl] \\
 & \PP^n &
\end{tikzcd}
\]
produces a diagram of stack morphisms over $\Sch_P$ represented by the exact sequence of $\cO|_P$--modules
\[
0 \to \cO|_P \to \cF(P) \to \cT_{P/S} \to 0.
\]
\end{lem}
\begin{proof}
Since the $\PGL_d$ action on both the structure sheaf and the tangent space of $\PP^n$ were dictated by the action on $\PP^n$, we will obtain
\begin{align*}
\cK\wedge^{\PGL_d} \cO|_{\PP^n} &= \cO|_{\Sch_P}, \text{ and} \\
\cK\wedge^{\PGL_d} \cT_{\PP^n/S} &= \cT_{\Sch_P/S}
\end{align*}
after twisting. Furthermore, since the Euler exact sequence is $\PGL_d$--equivariant we will have morphisms
\[
\cO|_{\Sch_P} \to \cK\wedge^{\PGL_d} \cOone^d \to \cT_{\Sch_P/S}
\]
where all stacks are represented by vector bundles over $P$ since they are vector bundles locally. Therefore, we obtain a sequence of $\cO|_P$--modules on $\Sch_P$,
\[
0 \to \cO|_P \to \cK \wedge^{\PGL_d} \cOSone^d \to \cT_{P/S} \to 0
\]
which is exact and which is non-split since it is exact and non-split locally. However, by \cite[Theorem 81]{Kollar} there is a unique such exact sequence, meaning that
\[
\cK \wedge^{\PGL_d} \cOSone^d \cong \cF(P).
\]
This finishes the proof.
\end{proof}

\subsubsection{Induced Action on the Pushforward}
Throughout the remained of this section we fix $f\colon \PP_S^n \to S$ as the structure morphism of relative projective space. Given a sheaf $\cF \colon \Sch_{\PP_S^n} \to \Sets$, its pushforward with respect to $f$ is the sheaf
\begin{align*}
f_*(\cF) \colon \Sch_S &\to \Sets \\
U &\mapsto \cF(U\times_S \PP_S^n).
\end{align*}
The choices of fiber products $U\times_S \PP_S^n$ for each $U\in \Sch_S$ can be encoded as a functor $\Sch_S \to \Sch_{\PP_S^n} \cong \PP^n$. In particular, we use the functor
\begin{align*}
p \colon \Sch_S &\to \PP^n \\
U &\mapsto (\PP_U^n,\cO_{\PP_U^n}(1),(x_0,\ldots,x_n))
\end{align*}
and which acts naturally on morphisms. Then, if $\fF \to \PP^n$ is the stack associated to the sheaf $\cF$ on $\PP_S^n$, then the pushforward $f_*(\cF)$ on $S$ corresponds to the fiber product
\[
\begin{tikzcd}
p^{-1}(\fF) = \Sch_S\times_{\PP^n}\fF \ar[r] \ar[d] & \fF \ar[d] \\
\Sch_S \ar[r,"p"] & \PP^n.
\end{tikzcd}
\]

Now, given $\varphi \in \PGL_d(S)$ and the corresponding isomorphism $\varphi \colon \PP^n \to \PP^n$, the functor $\varphi\circ p \colon \Sch_S \to \PP^n$ encodes a different choice of fiber product with $\PP_S^n$ for each $U \in \Sch_S$. In particular, if $\varphi$ corresponds to the isomorphism $A\colon \cO^d \to \cI^d$ for a line bundle $\cI$ on $S$, then
\[
(\varphi\circ p)(U) = (\PP_U^n,\cI|_U\otimes_{\cO|_{\PP_U^n}}\cO_{\PP_U^n}(1),((Ax)_1,\ldots,(Ax)_n))
\]
where $x=(x_0,\ldots,x_n) \in \cO_{\PP_U^n}(1)^d$ and $(Ax)_i$ is the $i^\textrm{th}$ component of $(A\otimes 1)(x)$. Going back to schemes for a moment, if the ``obvious" choice of fiber product $U\times_S \PP_S^n$ given by $p$ is the diagram
\[
\begin{tikzcd}
\PP_U^n \ar[r,"\alpha"] \ar[d] & \PP_S^n \ar[d] \\
U \ar[r] & S,
\end{tikzcd}
\]
i.e., $(\PP_U^n \xrightarrow{\alpha} \PP_S^n) = p(U) = (\PP_U^n,\cO_{\PP_U^n}(1),(x_0,\ldots,x_n))$, then the diagram
\[
\begin{tikzcd}
\PP_U^n \ar[r,"\alpha"] \ar[d] & \PP_S^n \ar[r,"\varphi"] & \PP_S^n \ar[d] \\
U \ar[rr] & & S
\end{tikzcd}
\]
corresponds to the choice $(\varphi\circ p)(U)$. Since these are both fiber products for the same diagram, there is a unique isomorphism of schemes $\varphi|_U \colon \PP_U^n \iso \PP_U^n$ between the two fiber products. In particular, for each $U\in \Sch_S$ there is an isomorphism in $\PP^n$
\[
(\varphi|_U,\Id) \colon (\PP_U^n,\cI|_U\otimes_{\cO|_{\PP_U^n}}\cO_{\PP_U^n}(1),((Ax)_1,\ldots,(Ax)_n)) \iso (\PP_U^n,\cO_{\PP_U^n}(1),(x_0,\ldots,x_n))
\]
where the isomorphism between the line bundles is simply the identity because
\[
\cO_{\PP_U^n}(1)|_{\varphi|_U} = \cI|_U\otimes_{\cO|_{\PP_U^n}}\cO_{\PP_U^n}(1)
\]
and the restriction map $\cO_{\PP_U^n}(1)(\varphi|_U)$ sends $x_i \mapsto (Ax)_i$. These morphisms give us a natural isomorphism of functors $\varphi\circ p \iso p$, and therefore also a canonical isomorphism of $\Sch_S$--stacks
\[
\can \colon (\varphi\circ p)^{-1}(\fF) \iso p^{-1}(\fF)
\]
where the isomorphism between the fibers $(\varphi\circ p)^{-1}(\fF)(U) \iso p^{-1}(\fF)(U)$ is given by the restriction maps
\[
\cF(\varphi|_U^{-1}) \colon \cF((\varphi\circ p)(U)) \to \cF(p(U))
\]
of the sheaf $\cF$ along the inverse of the isomorphism. In other words, this is the equivalent of canonical isomorphism $f_*(\varphi^*(\cF)) \cong f_*(\cF)$.

The above is applicable to any stacks $\fF \to \PP^n$, but we now apply it to the case of $\cOone^d \to \PP^n$. It is a standard fact that $f_*(\cOSone) \cong \cO^d$ where the sections $x_0,\ldots,x_n$ form the $\cO$--module basis. In our setup we see the same thing since the objects in the fiber $p^{-1}(\cOone)(U)$ are of the form
\[
(\PP_U^n,\cO_{\PP_U^n}(1),(x_0,\ldots,x_n),v\in \cO_{\PP_U^n}(1)(\PP_U^n)).
\]
Removing unnecessary information, we can consider $p^{-1}(\cOone)$ as having objects of the form
\[
(U,v\in \cO_{\PP_U^n}(1)(\PP_U^n) = \Span_{\cO(U)}(\{x_0,\ldots,x_n\})
\]
which gives an obvious isomorphism $p^{-1}(\cOone) \iso \cO^d$. We may also identify $(\varphi\circ p)^{-1}(\cOone)$ this way.
\begin{lem}
Given $\varphi \in \PGL_d(S)$ as above, there is a natural isomorphism $(\varphi\circ p)^{-1}(\cOone) \cong \cI^d$ over $\Sch_S$. Using both this isomorphism and the identification $p^{-1}(\cOone) \iso \cO^d$, the canonical isomorphism
\[
\can \colon (\varphi\circ p)^{-1}(\cOone) \iso p^{-1}(\cOone)
\]
appears as
\[
A^{-1} \colon \cI^d \iso \cO^d.
\]
\end{lem}
\begin{proof}
The objects in the fiber $(\varphi\circ p)^{-1}(\cOone)(U)$ are of the form
\[
(\PP_U^n,\cI|_U\otimes_{\cO|_{\PP_U^n}}\cO_{\PP_U^n}(1),((Ax)_1,\ldots,(Ax)_n),v\in \big(\cI|_U\otimes_{\cO|_{\PP_U^n}}\cO_{\PP_U^n}(1)\big)(\PP_U^n) )
\]
where the only variation is in the choice of $v$. Therefore, $(\varphi\circ p)^{-1}(\cOone)(U)$ is equivalent to the stack with objects
\[
(U,v\in \big(\cI|_U\otimes_{\cO|_{\PP_U^n}}\cO_{\PP_U^n}(1)\big)(\PP_U^n))
\]
where here
\[
\big(\cI|_U\otimes_{\cO|_{\PP_U^n}}\cO_{\PP_U^n}(1)\big)(\PP_U^n) = \bigoplus_{i=0}^n \cI\cdot x_i
\]
giving an isomorphism $(\varphi\circ p)^{-1}(\cOone) \iso \cI^d$.

The canonical isomorphism
\[
\can \colon (\varphi\circ p)^{-1}(\cOone) \iso p^{-1}(\cOone)
\]
comes from restriction along $\varphi^{-1}$ within $\cOone$. By definition, the restrictions in $\cOone$ along $\varphi$ correspond to maps of the form $A\otimes 1$, and the canonical isomorphism corresponds to $A^{-1}\otimes 1$ which after the pushforward is simply
\[
A^{-1} \colon \cI^d \iso \cO^d
\]
as claimed.
\end{proof}

If we consider $\cOone^d$ as consisting of column vectors and $p^{-1}(\cOone) \cong \cO^d$ as consisting of rows, and likewise for $(\varphi\circ p)^{-1}(\cOone)$ then we obtain $\cO$--module isomorphisms
\begin{align}
p^{-1}(\cOone^d) &\cong \Mat_d(\cO), \text{ and} \label{eq_pushforward_isomorphisms}\\
(\varphi\circ p)^{-1}(\cOone^d) &\cong \Mat_d(\cI). \nonumber
\end{align}
Since the canonical isomorphism $(\varphi\circ p)^{-1}(\cOone^d) \iso p^{-1}(\cOone^d)$ only acts within each row, writing $A^{-1} \in \Mat_d(\cI^*)$ as a matrix with entries in the dual line bundle, this canonical isomorphism simply appears as right multiplication by $A^{-1}$;
\begin{align*}
\can \colon \Mat_d(\cI) &\iso \Mat_d(\cO) \\
T &\mapsto TA^{-1}.
\end{align*}

As a final step, because the action of $\PGL_d$ on $\cOone^d$ is compatible with the action on $\PP^n$, we have the diagram
\[
\begin{tikzcd}
\cOone^d \ar[r,"\varphi"] \ar[d] & \cOone^d \ar[d] \\
\PP^n \ar[r,"\varphi"] & \PP^n
\end{tikzcd}
\]
which can be extended to the diagram
\[
\begin{tikzcd}
\cOone^d \ar[rr,"\varphi'"] \ar[dr] &[-5ex] &[-7ex] \varphi^{-1}(\cOone^d) \ar[rr] \ar[dl] &[-12ex] & \cOone^d \ar[d] \\
 & \PP^n \ar[rrr,"\varphi"] \ar[drr] & & & \PP^n \ar[dl] \\
 & & & \Sch_S \ar[bend left=15,dashed,ull,"p"] \ar[bend right=15,dashed,swap,ur,"p"] &
\end{tikzcd}
\]
where $\varphi'$ is the map uniquely determined by $\varphi^{-1}(\cOone^d)=\cOone^d\times_{\varphi} \PP^n$ being a fiber product and the maps $p\colon \Sch_S \to \PP^n$ are dashed since the diagram is only commutative without them. Hence, for each $\varphi \in \PGL_d(S)$ we get an automorphism
\[
p^{-1}(\cOone^d) \xrightarrow{p^{-1}(\varphi')} (\varphi \circ p)^{-1}(\cOone^d) \xrightarrow{\can} p^{-1}(\cOone^d)
\]
which defines an induced $\PGL_d$ action on $p^{-1}(\cOone^d)$ compatible with the action on $\PP^n$.
\begin{lem}
Using the identification $p^{-1}(\cOone^d) \cong \Mat_d(\cO)$, the induced action of $\PGL_d$ on $p^{-1}(\cOone^d)$ defined above agrees with the usual $\PGL_d$--action on $\Mat_d(\cO)$ by $\cO$--algebra isomorphisms.
\end{lem}
\begin{proof}
In the definition of the $\PGL_d$ action on $\cOone^d$ given in \eqref{eq_action_on_Oone^d}, $\varphi$ acts on the chosen section $v\in (\cI|_{\PP_U^n}\otimes_{\cO|_{\PP_U^n}}\cO|_{\PP_U^n}(1))^d(U)$ by $A\otimes 1$. If we consider $v$ as a column vector, then $A\otimes 1$ may be written as a $d\times d$ matrix with entries in $(\cI|_{\PP_S^n}\otimes_{\cO|_{\PP_S^n}}\cO|_{\PP_S^n}(1))(\PP_S^n)=\cO(S)$ and it acts on $v$ simply by left multiplication. Then, when we apply the isomorphisms of \eqref{eq_pushforward_isomorphisms} the map $p^{-1}(\varphi')$ acts on the column of rows as left multiplication by $A$, i.e. it acts simply as
\begin{align*}
\Mat_d(\cO) &\to \Mat_d(\cI) \\
T &\mapsto AT.
\end{align*}
As discussed above, the canonical isomorphism $(\varphi \circ p)^{-1}(\cOone^d) \to p^{-1}(\cOone^d)$ appears as right multiplication by $A^{-1}$ and so together the induced action of $\varphi$ on $\Mat_d(\cO)$ is given by
\begin{align*}
\varphi \colon \Mat_d(\cO) &\iso \Mat_d(\cO) \\
T &\mapsto ATA^{-1}.
\end{align*}
This is of course equal to the usual action of $\varphi$ on $\Mat_d(\cO)$, so we are done.
\end{proof}

This discussion culminates in the promised punchline.
\begin{prop}\label{new_Quillen}
Let $f\colon P \to S$ be a Severi-Brauer scheme of constant relative dimension $n$. Let $\cF(P)$ be the unique vector bundle on $P$ fitting into the Euler exact sequence. Then, the pushforward $f_*(F(P))$ has a natural structure as an Azumaya $\cO$--algebra. Furthermore, for an Azumaya $\cO$--algebra $\cA$ of constant degree $d=n+1$,
\[
f_*\big(\cF(\SB(\cA))\big) \cong \cA.
\] 
\end{prop}
\begin{proof}
As detailed above, in the split case we have an $\cO$--module isomorphism $f_*(\cF(\PP_S^n))\cong \Mat_d(\cO)$. We use this isomorphism to equip $f_*(\cF(\PP_S^n))$ with the structure of an Azumaya $\cO$--algebra. 

Now, let $\cK$ be the $\PGL_d$--torsor corresponding to $P \to S$. Using \Cref{simul_twisting}, we may simultaneously twist $\PP_S^n$ and $\cF(\PP_S^n)$ by $\cK$ to obtain $P$ and $\cF(P)$. Because the induced $\PGL_d$--action on $f_*(\cF(\PP_S^n))$ is compatible with the action on $\PP_S^n$, we obtain that
\[
f_*(\cF(P)) = f_*\big(\cK\wedge^{\PGL_d}\cF(\PP_S^n)\big) = \cK\wedge^{\PGL_d} f_*(\cF(\PP_S^n)).
\]
The action of $\PGL_d$ on $f_*(\cF(\PP_S^n))$ is by Azumaya $\cO$--algebra automorphisms, so $f_*(\cF(P))$ naturally inherits the structure of an Azumaya $\cO$--algebra.

In the case $P=\SB(\cA)$, then the torsor $\cK$ is also the torsor corresponding to the Azumaya $\cO$--algebra $\cA$. Hence,
\[
f_*(\cF(P)) = \cK\wedge^{\PGL_d} f_*(\cF(\PP_S^n)) = \cK\wedge^{\PGL_d} \Mat_d(\cO) \cong \cA.
\]
This finishes the proof.
\end{proof}

This provides a variation on Quillen's construction, since instead of the usual categorical equivalence of \eqref{eq_usual_Quillen}, we may also use the equivalence
\begin{align*}
\mathfrak{Azu}_d \leftrightarrow &\mathfrak{SB}_n \\
\cA \mapsto &\SB(\cA) \\
f_*\big(\cF(P)\big) \reflectbox{$\mapsto$} &P.
\end{align*}
with the algebra structure on $f_*\big(\cF(P)\big)$ given as in \Cref{new_Quillen}.

\subsection{Outer Type Severi-Brauer Schemes}\label{outer_SB_schemes}
Unsurprisingly, the relationship between inner type Severi-Brauer schemes and our upcoming definition of outer type Severi-Brauer schemes closely resembles the relationship between Azumaya algebras and Azumaya algebras with an involution of the second kind.

Recall that an Azumaya algebra with involution of the second kind is the data of $(f\colon L \to S, \cB, \tau)$ where $L\to S$ is a degree $2$ \'etale cover with a canonical order two $S$--automorphism $i\colon L \to L$, $\cB$ is an Azumaya algebra $\cO|_L$--algebra, and $\tau \colon f_*(\cB) \to f_*(\cB)$ is an $i^*$--semilinear involution of the $f_*(\cO|_L)=\cL$ algebra $f_*(\cB)$. The data of such a semilinear involution is equivalent to an $\cL$--linear anti-isomorphism $f_*(\cB) \to f_*(\cB)\otimes_{i^*}\cL$, which in turn is equivalent to an $\cL$--linear isomorphism
\[
\tilde{\tau}\colon f_*(\cB) \iso f_*(\cB)^\op \otimes_{i^*}\cL.
\]
We consider the category of Azumaya $\cO$--algebras of degree $d$ as a subcategory of the category of Azumaya algebras of degree $d$ with involution of the second kind by sending an Azumaya $\cO$--algebra $\cA$ to
\[
(S\sqcup S \to S, (\cA,\cA^\op),\tau)
\]
where the canonical automorphism of $S\sqcup S$ is the switch and $\tau \colon \cA\times\cA^\op \to \cA\times\cA^\op$ is the switch-opposite, i.e., $\tau(a,b^\op) = (b,a^\op)$. Here, for the equivalent $\cO\times\cO$--linear map we simply have that $\tilde{\tau}=\Id$. Thus, mirroring this setup for the Severi-Brauer scheme $\SB(\cA)$, the associated outer-type Severi-Brauer scheme associated to $\cA$ should be the data of
\[
\SB(\cA)\sqcup \SB(\cA^\op) \to S \sqcup S
\]
along with an $S\sqcup S$--scheme isomorphism, in this case the identity,
\[
\SB(\cA)\sqcup \SB(\cA^\op) \iso \big(\SB(\cA)^\op \sqcup \SB(\cA^\op)^\op) \times_{\sw} (S\sqcup S).
\]
Equivalently, since we are now working with schemes we can write this as the data of an $S$--scheme isomorphism $\tau$ such that the diagram
\[
\begin{tikzcd}
\SB(\cA)\sqcup \SB(\cA^\op) \ar[r,"\tau"] \ar[d] & \big(\SB(\cA)\sqcup \SB(\cA^\op)\big)^\op \ar[d] \\
S\sqcup S \ar[r,"\sw"] & S \sqcup S
\end{tikzcd}
\]
is a fiber product diagram. In the case when $\cA=\Mat_d(\cO)$, we take the split object to be
\[
\begin{tikzcd}
\PP_S^n \sqcup \PP_S^n \ar[r,"\tau_0"] \ar[d] & \PP_S^n \sqcup \PP_S^n \ar[d] \\
S\sqcup S \ar[r,"\sw"] & S \sqcup S
\end{tikzcd}
\]
where $\tau_0$ is simply the switch since we may take $(\PP_S^n)^\op = \PP_S^n$ due to the isomorphism $\Mat_d(\cO)\iso \Mat_d(\cO)^\op$ given by the transpose.

Thus, generalizing this setup to match the situation with $(L\to S,\cB,\tau)$ above, we arrive the following natural definition.
\begin{defn}\label{defn_outer_SB_scheme}
An \emph{outer type Severi-Brauer scheme} of constant relative dimension $n$ is the data of $(f\colon L \to S,P \to L,\tau)$ where
\begin{enumerate}
\item $L\to S$ is a degree $2$ \'etale cover with canonical automorphism $i\colon L \to L$,
\item $P \to L$ is a Severi-Brauer scheme over $L$ of constant relative dimension $n$,
\item $\tau \colon P \iso P^\op$ is an $S$--scheme isomorphism making the diagram
\[
\begin{tikzcd}
P \ar[r,"\tau"] \ar[d] & P^\op \ar[d] \\
L \ar[r,"i"] & L
\end{tikzcd}
\]
commute,
\end{enumerate}
and all such that \'etale locally over $S$ this data is isomorphic to $(S\sqcup S \to S, \PP_S^n\sqcup \PP_S^n, \tau_0)$.
\end{defn}

From an Azumaya algebra with involution of the second kind $(f\colon L \to S, \cB, \tau)$, one may consider $\cB$ simply as an Azumaya $\cO|_L$--algebra and obtain the Severi-Brauer scheme $\SB(\cB) \to L$. However, the involution $\tau \colon f_*(\cB) \to f_*(\cB)$ is only defined on the pushforward and does not immediately lead to a suitable map $\SB(\cB) \iso \SB(\cB^\op)$ needed to define an outer-type Severi-Brauer scheme. Instead, we construct the associated outer type Severi-Brauer scheme as in the inner case, from fibers of the type morphism.

\begin{lem}\label{op_iso_for_outer_SB}
Let $(f\colon L \to S,\cB,\tau)$ be an Azumaya algebra with involution of the second type of constant degree $d$ and set $\bG = \bU_{(\cB,\tau)}$. There is an isomorphism of sheaves $\LowGSB_{(\cB,\tau)} \iso \LowGSB_{(\cB^\op,\tau)}$ which fits into the commutative diagram
\[
\begin{tikzcd}
\LowGSB_{(\cB,\tau)} \ar[r,"\sim"] \ar[d] & \LowGSB_{(\cB^\op,\tau)} \ar[d] \\
\cOD{\bG} \ar[r,"\und^\vee"] & \cOD{\bG}
\end{tikzcd}
\]
where the downward maps are the respective type morphisms as given in \Cref{defn_outer_GSB_type}.
\end{lem}
\begin{proof}
We show that we have an isomorphism fitting into such a diagram for the underlying presheaves of constant rank flags. Set $\cL=f_*(\cO|_L)$. Recall from \Cref{defn_ConIdeal_Bt} that $\LowGSB_{(\cB,\tau)} = (\ConIdeal_{(\cB,\tau)})^\sharp$ where $\ConIdeal_{(\cB,\tau)}$ consists of flags of right $\cL$--ideals of $f_*(\cB)$ with strictly positive $\cL$--gap between subsequent components of the flag and such that
\[
(0\subseteq \cI_1 \subseteq \ldots \subseteq \cI_\ell \subseteq f_*(\cB)) = (0\subseteq \tau(\lann\cI_\ell) \subseteq \ldots \subseteq \tau(\lann\cI_1) \subseteq f_*(\cB)).
\]
For an $\cL$--right ideal $\cI\subseteq f_*(\cB)$, the left annihilator $\lann\cI$ is an $\cL$--left ideal which may therefore be viewed as an $\cL$--right ideal in $f_*(\cB^\op)$. Furthermore, left annihilator of $\lann\cI$ in $f_*(\cB^\op)$ is the same as the right annihilator of $\lann\cI$ in $f_*(\cB)$, which is just $\cI$ viewed as an $\cL$--left ideal in $f_*(\cB^\op)$. Therefore, if we consider
\[
0\subseteq \lann\cI_\ell \subseteq \ldots \subseteq \lann\cI_1 \subseteq f_*(\cB^\op)
\]
as a flag of right ideals, then $\tau\big(\lann(\lann\cI_j)\big) = \tau(\cI_j) = \lann\cI_{\ell+1-j}$ due to the symmetric property of the original flag. The required property on $\cL$--gaps will also be preserved, hence
\[
(0\subseteq \lann\cI_\ell \subseteq \ldots \subseteq \lann\cI_1 \subseteq f_*(\cB^\op)) \in \ConIdeal_{(\cB^\op,\tau)}
\]
and this clearly defines an isomorphism of presheaves $\ConIdeal_{(\cB,\tau)} \iso \ConIdeal_{(\cB^\op,\tau)}$.

Sufficiently locally, when $\ConIdeal_{(\cB,\tau)}$ is isomorphic to $\ConIdeal_{((\Mat_d(\cO),\Mat_d(\cO)),\tau_d)}$ this map appears as
\begin{align*}
&(0 \subseteq \cI_1\times (\lann\cI_\ell)^t \subseteq \ldots \subseteq \cI_\ell\times(\lann\cI_1)^t \subseteq \Mat_d(\cO)\times\Mat_d(\cO)) \\
\mapsto&(0 \subseteq (\lann\cI_\ell)\times \cI_1^t \subseteq \ldots \subseteq (\lann\cI_1)\times \cI_\ell^t \subseteq \Mat_d(\cO)^\op\times\Mat_d(\cO))^\op).
\end{align*}
Since for a right ideal $\cI\subseteq \Mat_d(\cO)$ of constant rank, $\rank_\cO(\lann\cI) = d^2-\rank(\cI)$, we see that we have a commutative diagram
\[
\begin{tikzcd}
\ConIdeal_{((\Mat_d(\cO),\Mat_d(\cO)),\tau_d)} \ar[r,"\sim"] \ar[d] & \ConIdeal_{((\Mat_d(\cO),\Mat_d(\cO))^\op,\tau_d)} \ar[d] \\
\cP_n \ar[r,"\und^\vee"] & \cP_n
\end{tikzcd}
\] 
where the downward maps are those defining the type morphism. Hence, after sheafifying and twisting we obtain an isomorphism $\LowGSB_{(\cB,\tau)} \iso \LowGSB_{(\cB^\op,\tau)}$ which fits into the commutative diagram
\[
\begin{tikzcd}
\LowGSB_{(\cB,\tau)} \ar[r,"\sim"] \ar[d] & \LowGSB_{(\cB^\op,\tau)} \ar[d] \\
\cOD{\bG} \ar[r,"\und^\vee"] & \cOD{\bG}
\end{tikzcd}
\]
as claimed, finishing the proof.
\end{proof}

Because the sheaf $\LowGSB_{(\cB,\tau)}$ is isomorphic to the sheaf of parabolic subgroups $\cPar_{\bG}$ for $\bG = \bU_{(\cB,\tau)}$, we know it is representable by some scheme $X\in \Sch_S$. Likewise, $\LowGSB_{(\cB^\op,\tau)}$ is represented by some $X'$. Thus, translating the diagram of \Cref{op_iso_for_outer_SB}, we have a diagram
\begin{equation}\label{eqn_scheme_diagram_before_fiber}
\begin{tikzcd}
X \ar[r,"\sim"] \ar[d] &[3ex] X' \ar[d] \\
\Of(\Dyn(\bG)) \ar[r,"\Of(\Dyn(i))"] & \Of(\Dyn(\bG))
\end{tikzcd}
\end{equation}
where $\Of(\Dyn(i))$ is as in \eqref{eq_OfDyn_i}. In particular, it acts as the order two isomorphism $i\colon L \to L$ on the $L_{((1),(n))}$ component of $\Of(\Dyn(\bG))$. Note that this is the component which is locally isomorphic to $S_{(1)} \sqcup S_{(n)}$. It is not the $S_{(1,n)}$ component which is isomorphic to $S$ since $(1,n)^\vee = (1,n)$.

\begin{lem}\label{fiber_is_SB(B)}
The fiber over the component $L_{((1),(n))}$ of the map $X \to \Of(\Dyn(\bG))$ is isomorphic to $\SB(\cB) \to L$.
\end{lem}
\begin{proof}
We argue via twisting. The Azumaya algebra with involution of the second kind $(f\colon L \to S,\cB,\tau)$ corresponds to a $\PGL_d\rtimes \ZZ/2\ZZ$ cocycle over some cover of $S$, say over $\{T_i \to S\}_{i\in I}$ and the cocycle $(\varphi_{ij})_{i,j\in I}$. Let $(\sigma_{ij})_{i,j\in I}$ be the image of this cocycle in $\ZZ/2\ZZ$, which is a cocycle corresponding to the \'etale cover $L\to S$.

The fact that $(\varphi_{ij})_{i,j\in I}$ is a cocycle for $(f\colon L \to S,\cB,\tau)$  means that the local $\cO|_{T_{ij}}$--algebra isomorphisms
\[
\begin{tikzcd}
\Mat_d(\cO|_{T_{ij}})\times\Mat_d(\cO|_{T_{ij}}) \ar[r,"\varphi_{ij}"] & \Mat_d(\cO|_{T_{ij}})\times\Mat_d(\cO|_{T_{ij}}) \\
\cO|_{T_{ij}}\times \cO|_{T_{ij}} \ar[r,"\sigma_{ij}"] \ar[u] & \cO|_{T_{ij}}\times \cO|_{T_{ij}} \ar[u]
\end{tikzcd}
\]
glue $\cO\times\cO \to \Mat_d(\cO)\times\Mat_d(\cO)$ into $\cL \to f_*(\cB)$. We now convert this diagram to a diagram of schemes in two simultaneous steps. By \cite[Tag 01SA]{Stacks} we may consider the $\sigma_{ij}$ as isomorphisms of the scheme $T_{ij}\sqcup T_{ij} \to T_{ij}$. Then, under \cite[Tag 01SB]{Stacks}, the algebra $\Mat_d(\cO|_{T_{ij}})\times\Mat_d(\cO|_{T_{ij}})$ corresponds to the matrix algebra $\Mat_d(\cO|_{T_{ij}\sqcup T_{ij}})$ on the disjoint union. Then, since $\cL \to f_*(\cB)$ over $S$ corresponds via \cite[Tag 01SB]{Stacks} to the $\cO|_L$--algebra $\cB$ over $L$, we see that simultaneously twisting $S\sqcup S$ and $\Mat_d(\cO|_{S\sqcup S})$ by the cocycle produces $L$ and $\cB$.

In turn, we replace $\Mat_d(\cO|_{T_{ij}\sqcup T_{ij}})$ by its Sever-Brauer scheme, which is $\PP_{T_{ij}\sqcup T_{ij}}^n \cong \PP_{T_{ij}}^n \sqcup \PP_{T_{ij}}^n$ and obtain a diagram
\[
\begin{tikzcd}
\PP_{T_{ij}}^n \sqcup \PP_{T_{ij}}^n \ar[r,"\varphi_{ij}"] \ar[d] & \PP_{T_{ij}}^n \sqcup \PP_{T_{ij}}^n \ar[d] \\
T_{ij}\sqcup T_{ij} \ar[r,"\sigma_{ij}"] & T_{ij}\sqcup T_{ij}
\end{tikzcd}
\]
where $\PGL_d\rtimes\ZZ/2\ZZ$ acts on $\PP_S^n \sqcup \PP_S^n$ by $\psi\sqcup \theta'(\psi)$ for $\psi\in \PGL_d$ and by the switch as usual for sections of $\ZZ/2\ZZ$. Recall that $\theta'$ is our chosen outer automorphism of $\PGL_d$ corresponding to the inverse transpose on $\GL_d$. Due to the categorical equivalence between Azumaya algebra and Severi-Brauer schemes, our cocycle must twist $\PP_S^n \sqcup \PP_S^n \to S\sqcup S$ into $\SB(\cB)\to L$.

Now, we consider the split case. By \eqref{eq_ConIdeal_iso} we have the isomorphism
\[
\LowGSB_{\Mat_d(\cO)} \iso \LowGSB_{((\Mat_d(\cO),\Mat_d(\cO)),\tau_d)}
\]
which is used to define the type morphism in general. Since $\bU_{((\Mat_d(\cO),\Mat_d(\cO)),\tau_d)} \cong \GL_d$, we know that $L=S\sqcup S$ and
\[
L_{((1),(n))} = S_{(1)} \sqcup S_{(n)}
\]
in $\Of(\Dyn(\GL_d))$. We know for an Azumaya $\cO$--algebra $\cA$ that both the fibers over $S_{(1)}$ and $S_{(n)}$ of $\LowGSB_{\cA}$ are represented by $\SB(\cA)$ and $\SB(\cA^\op)$ respectively. In particular, in the split case the fiber over $S_{(1)} \sqcup S_{(n)}$ is
\[
\PP_S^n \sqcup \PP_S^n \to S \sqcup S.
\]
Here, the induced $\PGL_d$--action on $\PP_S^n \sqcup \PP_S^n$ is by $\psi \sqcup \theta'(\psi)$ since the second $\PP_S^n \cong \SB(\Mat_d(\cO)^\op)$.

Finally, \Cref{severi_brauer_twist} tells us that $\LowGSB_{(\cB,\tau)}$ is the twist of $\LowGSB_{((\Mat_d(\cO),\Mat_d(\cO)),\tau_d)}$ by $(\varphi_{ij})_{i,j\in I}$, and therefore its fiber over $L_{((1),(n))}$ will be the twist of the fiber over $S_{(1)} \sqcup S_{(n)}$ in the split case. By the preceding discussion, this twist is $\SB(\cB)$, so we are done.
\end{proof}

Now when we consider the fiber over $L_{((1),(n)}$ in \eqref{eqn_scheme_diagram_before_fiber}, applying \Cref{fiber_is_SB(B)} twice tells us that we obtain a diagram
\[
\begin{tikzcd}
\SB(\cB) \ar[r,"\tau"] \ar[d] & \SB(\cB^\op) \ar[d] \\
L \ar[r,"i"] & L.
\end{tikzcd}
\]
It is clear from construction that this is locally isomorphic to the split diagram as required in \Cref{defn_outer_SB_scheme}.
\begin{defn}\label{defn_outer_SB}
Let $(f\colon L \to S,\cB,\tau)$ be an Azumaya algebra with involution of the second type of constant degree $d$. We define the data of the diagram appearing above to be the \emph{outer Severi-Brauer scheme} associated to $(f\colon L \to S,\cB,\tau)$, denoted $\SB(\cB,\tau)$.
\end{defn}

There is also an outer version of Quillen's construction which reconstructs an Azumaya algebra with involution of the second type from an outer Severi-Brauer scheme.
\begin{thm}
Let $f\colon L \to T$ be a degree $2$ \'etale cover, $g\colon P\to L$ a Severi-Brauer scheme and $\tau \colon P \iso P^\op$ an isomorphism defining an outer type Severi-Brauer scheme. Then, there is a canonical isomorphism $\cF(P)\iso \tau^*(\cF(P^\op))$ arising from $\tau$ such that the pushfoward
\[
f_*\big(g_*(\cF(P))\big) \iso f_*\big(g_*(\tau^*(\cF(P^\op)))\big) \cong f_*\big(g_*(\cF(P^\op))\big)\otimes_{i^*} \cL
\]
is equivalent to the information of an $i^*\colon \cL \to \cL$ semi-linear involution
\[
\tau' \colon f_*\big(g_*(\cF(P))\big) \to f_*\big(g_*(\cF(P))\big)
\]
making $(L\to S,g_*(\cF(P)),\tau')$ an Azumaya algebra with involution of the second kind.

In the case that $P=\SB(\cB,\tau)$ is the outer type Severi-Brauer scheme as defined in \Cref{defn_outer_SB}, then $g_*(\cF(P)) = \cB$ and $\tau'=\tau$.
\end{thm}
\begin{proof}
The action of $\PGL_d$ defined on the pair $(\PP_S^n,\cOSone)$ used in \Cref{new_Quillen} naturally extends to a $\PGL_d\rtimes \ZZ/2\ZZ$--action on the pair $(\PP_S^n\sqcup \PP_S^n, \cO_{\PP_S^n\sqcup \PP_S^n}(1))$. By acting on the second component by $\theta'(\varphi)$ for $\varphi \in \PGL_d$. Then, if we consider the diagram
\[
\begin{tikzcd}
\cO_{\PP_S^n\sqcup \PP_S^n}(1) \ar[r,"\sw"] \ar[d] & \cO_{\PP_S^n\sqcup \PP_S^n}(1) \ar[d] \\
\PP_S^n\sqcup \PP_S^n \ar[r,"\sw"]& \PP_S^n\sqcup \PP_S^n
\end{tikzcd}
\]
it will be equivariant with repsect to the $\PGL_d\rtimes\ZZ/2\ZZ$-action provided we act by the opposite cocycle on the right. Hence, if we twist by a cocycle for $P\to L$, we obtain
\[
\begin{tikzcd}
\cF(P) \ar[r,"\tau'"] \ar[d] & \cF(P^\op) \ar[d] \\
P \ar[r,"\tau"] & P^\op
\end{tikzcd}
\]
where $\tau'$ is equivalent to information of an isomorphism $\cF(P)\iso \tau^*(\cF(P^\op))$. Pushing this forward along $g$ produces an $\cO|_L$--algebra isomorphism
\[
g_*(\cF(P)) \iso g_*(\tau^*(\cF(P^\op))) \cong i^*(g_*(\cF(P^\op))).
\]
using the Azumaya algebra structure on these pushforwards from \Cref{new_Quillen}. Pushing forward further produces an $\cL$--linear isomorphism
\[
f_*\big(g_*(\cF(P))\big) \iso f_*\big(g_*(\tau^*(\cF(P^\op)))\big) \cong f_*\big(g_*(\cF(P^\op))\big)\otimes_{i^*} \cL
\]
as claimed, which is equivalent to an $\cL$--semi-linear involution $\tau' \colon f_*\big(g_*(\cF(P))\big) \to f_*\big(g_*(\cF(P))\big)$. Looking locally in the split case, since we applied both the switch and acted on the right by the opposite cocycle, the resulting map
\[
\tau' = f_*\big(g_*(\sw)\big) \colon \Mat_d(\cO)\times \Mat_d(\cO) \iso \Mat_d(\cO)\times \Mat_d(\cO)
\]
will be the switch transpose $\tau_d$ as in \eqref{split_second_involution}. Thus, $(L\to S, g_*(\cF(P)),\tau')$ is an Azumaya algebra with involution of the second kind.

Since this process recovers $(S\sqcup S\to S,(\Mat_d(\cO),\Mat_d(\cO),\tau_d)$ in the split case and the induced $\PGL_d\rtimes \ZZ/2\ZZ$--action will be the usual one, it is clear that when we twist by a cocycle corresponding to $\SB(\cB,\tau)$ we recover $(L\to S,\cB,\tau)$ as claimed.
\end{proof}

\end{document}